\documentclass[12pt]{amsart}
\usepackage{pgfplots}
\usepackage[utf8]{inputenc}
\usepackage[colorlinks=true]{hyperref}
\usepackage{amsmath,amscd,amsbsy,amssymb,amsthm,amsfonts}
\usepackage{latexsym,url,bm}
\usepackage{cite,enumitem}
\usepackage{cancel}
\usepackage{xcolor}
\usepackage{tikz}

\newcommand{\doth}[1]{\Dot{\mathcal{H}}^{#1}}

\newcommand{\W}{{\mathbf W}}
\newcommand{\nP}{{\mathbf P}}
\newcommand{\uA}{{\underline A}}
\newcommand{\Astar}{{\underline{A}_{\frac{1}{4}}}}
\newcommand{\uAS}{{\underline{A}^\sharp}}
\newcommand{\uAStar}{{\underline{A}^{\sharp}_{\frac{1}{4}}}}
\newcommand{\ua}{{\underline a}}
\newcommand{\uB}{{\underline B}}
\newcommand{\ub}{{\underline b}}

\newcommand{\hw}{{\hat{w}}}
\newcommand{\hr}{{\hat{r}}}

\usepackage{fullpage}

\newtheorem{theorem}{Theorem}[section]
\newtheorem{lemma}[theorem]{Lemma}
\newtheorem{proposition}[theorem]{Proposition}

\theoremstyle{remark}
\newtheorem{remark}{Remark}
\numberwithin{equation}{section}
\numberwithin{figure}{section}

\allowdisplaybreaks

\title{Low regularity well-posedness for two-dimensional deep gravity water waves with constant vorticity}

\author{Lizhe Wan}
\address{Department of Mathematics, University of Wisconsin - Madison}
\curraddr{}
\email{lwan33@wisc.edu}

\keywords{constant vorticity, holomorphic coordinates, balanced energy estimate.}
\subjclass[2020]{76B15, 35Q31}
\pgfplotsset{compat=1.17}

\begin{document}

\begin{abstract}
  We consider the two-dimensional gravity water waves with nonzero constant vorticity in infinite depth.
  We show that for $s\geq \frac{3}{4}$, the water waves system is locally well-posed in $\mathcal{H}^{s}$, which is the nonzero constant vorticity counterpart of the breakthrough work of Ai-Ifrim-Tataru in \cite{ai2023dimensional}.
  It is also a  $\frac{1}{4}$ improvement in Sobolev regularity compared to the previous result of Ifrim-Tataru in \cite{MR3869381}.
\end{abstract}

\maketitle
\tableofcontents

\section{Introduction} \label{s:Introduction}

We consider gravity water waves with nonzero constant vorticity  but without surface tension in two space dimensions.
As shown in Figure \ref{f:Figure}, the fluid occupies a time-dependent domain $\Omega (t) \subset \mathbb{R}^2$ with infinite depth and a free upper boundary $\Gamma(t)$ which is asymptotically approaching $y=0$.
Denoting the fluid velocity by $\mathbf{u}(t,x,y) = (u(t,x,y), v(t,x,y))$, the pressure by $p(t,x,y)$, and the constant vorticity by $\gamma$, the equations inside $\Omega(t)$ are given by the incompressible Euler equations:
\begin{equation*}
\left\{
             \begin{array}{lr}
            u_t +uu_x +vu_y = -p_x &  \\
            v_t + uv_x +vv_y = -p_y -g& \\
            u_x +v_y =0 & \\
            \omega = u_y -v_x = -\gamma.
             \end{array}
\right.
\end{equation*}
On the boundary $\Gamma(t)$ we have the dynamic boundary condition
\begin{equation*}
    p =0,
\end{equation*}
and the kinematic boundary condition
\begin{equation*}
    \partial_t +\mathbf{u}\cdot \nabla \text{ is tangent to }\Gamma(t).
\end{equation*}
Here $g>0$ is the gravitational constant.

\begin{figure} 
  \centering
  \begin{tikzpicture}
    \begin{axis}[ xmin=-12, xmax=12.5, ymin=-2.5, ymax=1.5, axis x
      line = none, axis y line = none, samples=100 ]

      \addplot+[mark=none,domain=-10:10,stack plots=y]
      {0.2*sin(deg(2*x))*exp(-x^2/20)};
      \addplot+[mark=none,fill=gray!20!white,draw=gray!30,thick,domain=-10:10,stack
      plots=y] {-1.5-0.2*sin(deg(2*x))*exp(-x^2/20)} \closedcycle;
      \addplot+[black, thick,mark=none,domain=-10:10,stack plots=y]
      {1.5+0.2*sin(deg(2*x))*exp(-x^2/20)};

      \draw[->] (axis cs:-3,-2) -- (axis cs:-3,1) node[left] {\(y\)};
      \draw[->] (axis cs:-11,0) -- (axis cs:12,0) node[below] {\(x\)};
      \filldraw (axis cs:-3,-1.5)  node[above left]
      {\(\)}; \node at (axis cs:0,-0.75) {\(\Omega(t)\)}; \node at
      (axis cs:6,0.3) {\(\Gamma(t)\)};

    \end{axis}
  \end{tikzpicture}
  \caption{The fluid domain.} \label{f:Figure}
\end{figure}
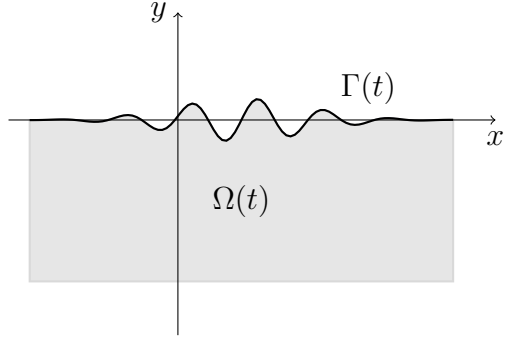

\subsection{Water waves in holomorphic coordinates}
Let $\mathbf{P}$ be the projection onto negative frequencies, with the definition
\begin{equation*}
    \mathbf{P} := \frac{1}{2}(\mathbf{I} - iH),
\end{equation*}
where $H$ denotes the Hilbert transform.
We define holomorphic functions on $\mathbb{R}$ to be the functions whose Fourier transforms are supported on $(-\infty,0]$;
equivalently in the language of complex analysis, they admit a bounded holomorphic extension onto the lower half-space.
This can be further described by the relation $\nP f = f$.
Similarly, we define $\bar{\nP}$ to be the projection onto positive frequencies:
\begin{equation*}
    \bar{\nP} := \frac{1}{2}(\mathbf{I} + iH) = \mathbf{I} -\nP.
\end{equation*}
Functions such that $\bar{\nP} f = f$ are called anti-holomorphic functions.
Anti-holomorphic functions are complex conjugates of holomorphic functions.

As done in the work of Ifrim-Tataru \cite{MR3869381} and Ifrim-Rowan-Tataru-Wan \cite{MR4462478},  we rely here on the holomorphic coordinates and use   holomorphic position/velocity potential variables  $(W(t,\alpha),Q(t,\alpha))$ to express the water waves system.
It is formulated in the following system of equations:
\begin{equation}
\left\{
             \begin{array}{lr}
             W_t + (W_\alpha +1)\underline{F} +i\dfrac{\gamma}{2}W = 0 &  \\
             Q_t - igW +\underline{F}Q_\alpha +i\gamma Q +\mathbf{P}\left[\dfrac{|Q_\alpha|^2}{J}\right]- i\dfrac{\gamma}{2}T_1 =0,&
             \end{array}
\right.\label{e:CVWW}
\end{equation}
where $J := |1+ W_\alpha|^2$ is the Jacobian, and
\begin{equation*}
\begin{aligned}
&F: = \mathbf{P}\left[\frac{Q_\alpha - \Bar{Q}_\alpha}{J}\right], \quad &F_1 = \mathbf{P}\left[\frac{W}{1+\Bar{W}_\alpha}+\frac{\Bar{W}}{1+W_\alpha}\right],\\
&\underline{F}: =F- i \frac{\gamma}{2}F_1,  \quad &T_1: = \mathbf{P}\left[\frac{W\Bar{Q}_\alpha}{1+\Bar{W}_\alpha}-\frac{\Bar{W}Q_\alpha}{1+W_\alpha}\right].
\end{aligned}
\end{equation*}

The system \eqref{e:CVWW} has a conserved energy
\begin{equation}
  \mathcal{E}(W,Q) =  \Re \int g|W|^2(1+W_\alpha)-iQ\bar{Q}_\alpha + \gamma Q_\alpha (\Im W)^2 - \frac{\gamma^3}{2i}|W|^2(1+W_\alpha) \,d\alpha, \label{HamiltonianEnergy}
\end{equation}
and also a conserved horizontal momentum
\begin{equation*}
    \mathcal{P}(W,Q) = \int  \left\{ \frac{1}{i}(\bar{Q}W_\alpha - Q\bar{W}_\alpha) - \gamma |W|^2 + \frac{\gamma}{2}(W^2\bar{W}_\alpha - \bar{W}^2 W_\alpha)\right\} d\alpha.
\end{equation*}
In the case of zero vorticity, the system \eqref{e:CVWW} is reduced to the irrotational gravity water waves system which was previously studied in the same formulation as \cite{hunter2016two, MR3499085, ai2023dimensional,  MR4483135,   MR3625189}.
Here for simplicity, we assume the constant vorticity $\gamma>0$.
In the case of the negative vorticity, one can always switch the sign of the variable $\alpha \rightarrow -\alpha$, so that the vorticity can be made positive after the change.

The system \eqref{e:CVWW} does not have a complete scaling.
The space-time scaling
\begin{equation}
    (W(t,\alpha),Q(t,\alpha)) \rightarrow (\lambda^{-2}W(\lambda t, \lambda^2 \alpha), \lambda^{-3}Q(\lambda t, \lambda^2 \alpha)), \label{SpacetimeScaling}
\end{equation}
leaves the gravity $g$ unchanged, but it changes the voricity $\gamma$ to $\lambda \gamma$.
One the other hand, the pure space scaling
\begin{equation*}
    (W(t,\alpha),Q(t,\alpha)) \rightarrow (\lambda^{-2}W( t, \lambda^2 \alpha), \lambda^{-3}Q( t, \lambda^2 \alpha)),
\end{equation*}
does not change the constant vorticity $\gamma$, but the  gravitational constant $g$ becomes $\lambda^{-1}g$.

A simplified model of \eqref{e:CVWW} is its linearization around the zero solution
\begin{equation}
\left\{
             \begin{array}{lr}
             w_t + q_\alpha = 0 &  \\
             q_t +i\gamma q - igw =0,&
             \end{array}
\right.\label{e:ZeroLinear}
\end{equation}
restricted to holomorphic functions.
\eqref{e:ZeroLinear} is a linear dispersive equation that can be written as
\begin{equation*}
    w_{tt}+ i \gamma w_t + igw_\alpha =0.
\end{equation*}
Its dispersion relation is given by
\begin{equation*}
    \tau^2 + \gamma\tau +g\xi =0, \quad \xi\leq 0,
\end{equation*}
whose graph is the intersection of a lateral parabola with the left half plane.
The two intersection points are $(0,0)$ and $(0, -\gamma)$.

The conserved energy of \eqref{e:ZeroLinear} is given by
\begin{equation}
    \mathcal{E}_0(w,q) = \int g|w|^2 - iq\bar{q}_\alpha\, d\alpha = g\|w\|_{L^2}^2 + \|q\|_{\dot{H}^{\frac{1}{2}}}^2. \label{ConservedEnergy}
\end{equation}
This conserved energy suggests the functional framework to study \eqref{e:CVWW}.
The system \eqref{e:ZeroLinear} is well-posed in $\mathcal{H}: = L^2 \times \dot{H}^{\frac{1}{2}}$ space.
To measure the higher regularity of the solution we will use the spaces $\mathcal{H}^{s}$ endowed with the norm
\begin{equation*}
    \| (w,q)\|_{\mathcal{H}^{s}} := \|\langle D\rangle^s(w,q) \|^2_{L^2 \times \dot{H}^{\frac{1}{2}}}, \quad s\in \mathbb{R}.
\end{equation*}
We also define the corresponding homogeneous spaces $\doth{s}$ given by
\begin{equation*}
    \|(w,q)\|_{\doth{s}} := \| |D|^s(w,q) \|^2_{L^2 \times \dot{H}^{\frac{1}{2}}}, \quad s\in \mathbb{R}.
\end{equation*}
 The system \eqref{e:CVWW} is  fully nonlinear.
 By differentiation, it can be converted into a quasilinear system.
 As in the Hunter-Ifrim-Tataru \cite{hunter2016two}, we set
 \begin{equation*}
     \W := W_\alpha, \quad R := \frac{Q_\alpha}{1+W_\alpha}, \quad Y := \frac{W_\alpha}{1+W_\alpha}.
 \end{equation*}
 The function $R$ has an intrinsic meaning as it represents the complex velocity on the water surface.
 We also need two other auxiliary functions.
 The first one $\ua$ is the \textit{frequency-shift}, and it is given by
 \begin{equation*}
     \ua := a + \frac{\gamma}{2}a_1, \quad a := i (\bar{\nP}[\bar{R}R_\alpha] - \nP[R\bar{R}_\alpha]), \quad a_1 := R+ \bar{R} -N,
 \end{equation*}
 where the leading part $a$ is also called \textit{Taylor coefficient}, and
 \begin{equation*}
     N: = \nP[W\bar{R}_\alpha - \bar{\W}R]+ \bar{\nP}[\bar{W}R_\alpha - \W\bar{R}].
 \end{equation*}
The other auxiliary function we need is the \textit{advection velocity} $\ub$.
Its leading part $b$ is also called the \textit{transport coefficient}.
\begin{equation*}
    \ub := b-i\frac{\gamma}{2} b_1, \quad b: = \nP \left[\frac{Q_\alpha}{J}\right] + \bar{\nP} \left[\frac{\bar{Q}_\alpha}{J} \right], \quad b_1: = \nP \left[ \frac{W}{1+\bar{\W}}\right]- \bar{\nP} \left[ \frac{\bar{W}}{1+\W}\right].
\end{equation*}
The expression $g +\ua$ represents the normal derivative of the pressure on the free boundary, and plays an essential role in the proof.

Using the notation of $\ua$ and $\ub$,  one can rewrite \eqref{e:CVWW} as
\begin{equation}
 \left\{
             \begin{array}{lr}
            W_t +\underline{b}(W_\alpha+1) +i\frac{\gamma}{2}W = \bar{R} + i\frac{\gamma}{2}\bar{W} &\\
             Q_t + \underline{b}Q_\alpha - igW +i\gamma Q - i\frac{\gamma}{2}\bar{R}W = \bar{\nP}[|R|^2]-i\frac{\gamma}{2}\bar{\nP}[W\bar{R}-\bar{W}R].&
             \end{array}
\right.\label{e:CVWW1}
\end{equation}

Instead of studying \eqref{e:CVWW} directly, in this article, we will mostly consider the following
differentiated system.
The pair $(\W,R)$ is a good variable because it diagonalizes the differentiated system:
 \begin{equation}
\left\{
             \begin{array}{lr}
            \mathbf{W}_t +\underline{b}\mathbf{W}_\alpha + \dfrac{(1+\mathbf{W})R_\alpha}{1+\Bar{\mathbf{W}}} = (1+\mathbf{W})\underline{M}+i\dfrac{\gamma}{2}\mathbf{W}(\mathbf{W}-\Bar{\mathbf{W}})  &\\
             R_t + \underline{b}R_\alpha +i\gamma R - i\dfrac{g\mathbf{W}-a}{1+\mathbf{W}} = i\dfrac{\gamma}{2}\dfrac{R\mathbf{W}+\Bar{R}\mathbf{W}
             +N}{1+\mathbf{W}}.&
             \end{array}
\right.\label{differentiatedEqn}
\end{equation}
Here the auxiliary functions are given by
\begin{align*}
    & M := \frac{R_\alpha}{1+\bar{\W}}+
    \frac{\bar{R}_\alpha}{1+\W} - b_\alpha = \bar{\nP}[\bar{R}Y_\alpha - R_\alpha \bar{Y}] + \nP[R\bar{Y}_\alpha -\bar{R}_\alpha Y], \\
  & M_1 := \W -\bar{\W} - b_{1,\alpha} = \nP[W\bar{Y}]_\alpha - \bar{\nP}[\bar{W}Y]_\alpha, \qquad \underline{M} := M- i \frac{\gamma}{2}M_1.
\end{align*}
Note that $\ub_\alpha$ satisfies the relation
\begin{equation}
    \ub_\alpha = \frac{R_\alpha}{1+\bar{\W}}+ \frac{\bar{R}_\alpha}{1+\W} -i\frac{\gamma}{2}(\W -\bar{\W}) -\underline{M}. \label{ubalpha}
\end{equation}
Using the variable $(Y,R)$, we can rewrite the equation \eqref{differentiatedEqn} using the material derivative $D_t:= \partial_t + \ub \partial_\alpha$ as
 \begin{equation}
\left\{
             \begin{array}{lr}
            D_t Y + |1-Y|^2R_\alpha = (1-Y)\underline{M}+ i\frac{\gamma}{2}[Y^2+ \frac{\bar{Y}}{1-\bar{Y}}Y(1-Y)]  &\\
             D_t R -i(g+\ua)Y = -i\ua-i\frac{\gamma}{2}(R-\bar{R}).&
             \end{array}
\right.\label{EqnYR}
\end{equation}
Again the system \eqref{differentiatedEqn} does not have a complete scaling.
However, we can  still introduce the \textit{order} of  multilinear expressions for \eqref{differentiatedEqn} following the work in \cite{MR3869381}.
For single terms, we assign the following orders
\begin{itemize}
    \item The order of $|D|^s W$ is $s-1$.
    \item The order of $|D|^s R$ is $s-\frac{1}{2}$.
    \item The order of $\gamma$ is $\frac{1}{2}$.
\end{itemize}
For a multilinear form involving products of such terms, the total order is defined as the sum of the orders of all factors.

Before stating the main low regularity result in this article, we first briefly recall some recent results about this problem and  give a short overview of the local well-posedness result in \cite{MR3869381}.
To state them  we need to recall and  define a series of pointwise and $BMO$   control norms
\begin{equation*}
\begin{aligned}
&A : = \|\mathbf{W} \|_{L^\infty}+\|Y\|_{L^\infty} +\||D|^{\frac{1}{2}}R \|_{L^\infty \cap B^0_{\infty, 2}},\\
&B : = \||D|^{\frac{1}{2}}\mathbf{W} \|_{BMO}+\|R_\alpha \|_{BMO},\\
&A_{-\frac{1}{2}} : = \| |D|^{\frac{1}{2}}W\|_{L^\infty} + \|R\|_{L^\infty},\quad A_{-1}:= \|W\|_{L^\infty} + \||D|^{\frac{1}{2}} Q\|_{BMO},\\
&A_{-\frac{3}{2}}:= \||D|^{-\frac{1}{2}}W\|_{L^\infty}, \quad A_{-2}: = \| |D|^{-1}W \|_{L^\infty},  \\
&\uA_{\frac{1}{2}}=\underline{B}: =B+\gamma A +\gamma^2 A_{-\frac{1}{2}}, \quad \underline{A}: = A+\gamma A_{-\frac{1}{2}}+\gamma^2 A_{-1}+ \gamma^3 A_{-\frac{3}{2}} + \gamma^4 A_{-2}.
\end{aligned}
\end{equation*}
These control norms are defined and used in the energy estimate in \cite{MR3869381} as well as in later works; see  for instance Theorem \ref{t:oldcubicenergy} in \cite{MR4462478}.
Here the $A_{-1}$ norm is slightly different from the one defined in \cite{MR3869381}, where  we add an additional control norm $\||D|^\frac{1}{2} Q\|_{BMO}$.  
This new norm helps us get a more refined cubic energy estimate, and does not produce a long energy estimate as Theorem \ref{t:oldcubicenergy}.

We also add here  additional comments describing these norms.
 The control norms $A$ and $\uA$ are invariant with respect to the scaling \eqref{SpacetimeScaling}.
 The norm $A_s$ has $s$ more derivatives compared to $A$.
One should remark that although each term of the underline control norms $\uA_s$ does not have the same amount of derivatives, they have the same order.
The difference of orders between $\uA_s$ and  $\uA$ is also $s$.
The underline control norms $\uA_s$ become $A_s$ in the zero vorticity case.

To complete the description of these norms one should recall that $B$ allows for the propagation of regularity of the solutions; the same is true about   $\uB$ which ultimately determines the well-posedness of the system in \cite{MR3869381}.
In this article, we also define the following control norms
\begin{align*}
    &A_{\frac{1}{4}} : = \| \W\|_{\dot{B}^{\frac{1}{4}}_{\infty, 2}}+ \| R\|_{\dot{B}^{\frac{3}{4}}_{\infty, 2}}, \qquad A_{-\frac{1}{4}} : = \| \W\|_{\dot{B}^{-\frac{1}{4}}_{\infty, 2}}+ \| R\|_{\dot{B}^{\frac{1}{4}}_{\infty, 2}},  \\
    & A_{-\frac{3}{4}} : = \| \W\|_{\dot{B}^{-\frac{3}{4}}_{\infty, 2}}+ \| R\|_{\dot{B}_{\infty, 2}^{-\frac{1}{4}}}, \quad
    \Astar : = A_{\frac{1}{4}}+ \gamma^{\frac{1}{2}}A+ \gamma A_{-\frac{1}{4}}  + \gamma^{\frac{3}{2}} A_{-\frac{1}{2}}+ \gamma^2 A_{-\frac{3}{4}}.
\end{align*}
Here $\Astar$ can be seen as  an intermediate control norm interpolating between $\uA$ and $\uB$.
In our energy estimate below, we no longer use the control norm $\uB$, the product $\uA \uB$ is replaced by $\uA^2_{\frac{1}{4}}$.
For the $L^2$ and $L^4$ based control norms, we define
\begin{align*}
    &\mathbf{N}_s : = \| (\W, R)\|_{\dot{H}^s\times \dot{H}^{s+\frac{1}{2}}}, \quad A^\sharp : = \|\W\|_{\dot{W}^{\frac{1}{4},4}}+ \|R\|_{\dot{W}^{\frac{3}{4},4}}, \quad A^\sharp_{\frac{1}{4}}: = \|\W\|_{\dot{W}^{\frac{1}{2},4}}+ \|R\|_{\dot{W}^{1,4}},\\
    &\uAS : = A^\sharp +\gamma^{\frac{1}{2}}\|R\|_{\dot{W}^{\frac{1}{2},4}} +\gamma \|W \|_{\dot{W}^{\frac{3}{4},4}}+\gamma\|R\|_{\dot{W}^{\frac{1}{4},4}} +\gamma^{\frac{3}{2}}\|R\|_{L^4} + \gamma^2 \|W \|_{\dot{W}^{\frac{1}{4},4}}+ \gamma^2 \|R \|_{\dot{W}^{-\frac{1}{4},4}},\\
    &\uAStar : = A^\sharp_{\frac{1}{4}} +\gamma^{\frac{1}{2}}\|R\|_{\dot{W}^{\frac{3}{4},4}} +\gamma \|W \|_{\dot{W}^{1,4}}+\gamma\|R\|_{\dot{W}^{\frac{1}{2},4}} +\gamma^{\frac{3}{2}}\|R\|_{\dot{W}^{\frac{1}{4},4}} + \gamma^2 \|W \|_{\dot{W}^{\frac{1}{2},4}}+ \gamma^2 \|R \|_{L^4}.
\end{align*}
By the Sobolev embedding $\dot{W}^{\frac{1}{4},4}(\mathbb{R})\hookrightarrow BMO(\mathbb{R})$, the control norm $\uAS$ can be seen roughly as a slightly larger control norm than $\uA$.
We also have the following observation for the control norm $\uAStar$: $\Astar \lesssim \uAStar$.

We remark that in Ai-Ifrim-Tataru \cite{ai2023dimensional} for the zero vorticity case, the control norms $A, A_{\frac{1}{4}}$ and $A^\sharp$ are used for the energy estimate.
In the case of nonzero vorticity, our control norms need to include the vorticity terms with the same order.
This leads to the above choice of control norms.

\subsection{Some previous results}
The constant vorticity water waves system is a special case of free boundary incompressible Euler equations with general vorticity.
Historically a lot of work has been done for the non-constant vorticity case. 
For completeness reasons we include some references though the list is by no means exhaustive.
We refer the reader to the following more contemporary references and the references within: Christodoulou-Lindblad \cite{MR1780703}, Coutand-Shkoller\cite{MR2138139}, Lannes \cite{MR2178961}, Lindblad \cite{MR2291920}, Shatah-Zeng\cite{MR2388661}, Zhang-Zhang \cite{MR2410409}.
It is worth mentioning that most of these works rely on a Lagrangian formulation rather than an Eulerian approach in a compact domain.
This is also the case with  the most recent work pursued for these model problems, for example, the article of Wang-Zhang-Zhao-Zheng \cite{MR4263411}.
However, the work of  Ifrim-Pineau-Tataru-Taylor \cite{ifrim2023sharp} departs from this Lagrangian setting and introduces a full Eulerian approach which allows them to obtain  sharp local well-posedness for the free boundary Euler equations.

If one imposes conditions on the vorticity, then the free boundary Euler equations become the water wave equations.
There are a lot of existence and well-posedness results on the zero vorticity case.
These results go back to the early works on Nalimov \cite{n}, Dyachenko \emph{et.al} \cite{zakharov, zakharov2} and so on.
More relevant to our presentation are the references within the last 20 years and so. 
The first local well-posedness result for 2D gravity water wave equations is due to Wu.
Her result on the well-posedness is established in high regularity Sobolev spaces; see \cite{MR1471885}.
 Alazard, Burq and Zuily use the tools of paradifferential operators and Strichartz estimates in \cite{MR3260858, MR3852259} and effectively lower the regularity required for the initial data.
 Substantial regularity improvements of these results were obtained by Ai in \cite{MR4161284, MR4098033}.
 In getting these results, Ai relied on a very careful analysis which leads to a novel parametrix construction that gives lossless Strichartz estimates.
 Recently, Ai-Ifrim-Tataru devised a new method, called \textit{the balanced energy estimate}  in \cite{ai2023dimensional}.
 As a direct consequence of this new method, they have the state-of-art result in this realm of problems proving the local well-posedness of 2D gravity water waves with only $\frac{1}{4}$ more derivatives above the scaling.
 As a side note, it should be mentioned that this method and consequently the regularity improvement on the initial data obtained in \cite{ai2023dimensional} do not rely on the dispersive character of the problem.

The study of constant vorticity water waves is less common. Ifrim-Tataru in \cite{MR3869381} proved the first local well-posedness for large data, as well as cubic lifespan bounds
for small data solutions in a low regularity setting.
Around the same time
Bieri-Miao-Shahshahani-Wu in \cite{bieri2015motion} proved the local existence of 2D free boundary self-gravitating incompressible fluids with constant vorticity for smooth initial data in bounded domains.
Recently,  it was shown in \cite{MR4658635} by Berti-Maspero-Murgante the almost global in time existence result of small amplitude solutions of the 2D gravity-capillary water wave equations with constant vorticity in $\mathbb{T}$.

The use of conformal formulation for two-dimensional water waves  originates in early work on traveling waves of Levi-Civita \cite{MR1512238}.
It has been widely used since then in order to study a variety of water wave problems, especially the irrotational pure gravity model such as in \cite{hunter2016two, ai2023dimensional, MR3499085,  MR4483135,  MR3625189}.
In this article, we follow the formulation and keep  mostly the same notations as in the work \cite{MR3869381} of  Ifrim and Tataru, though we will introduce new elements/notations.
The authors of \cite{MR3869381} not only establish the cubic energy estimate but also prove that \eqref{differentiatedEqn} is locally well-posed in $\mathcal{H}^1$.
The author, together with Ifrim, Rowan and Tataru worked further with this formulation in \cite{MR4462478} to show that the constant vorticity gravity water waves can be approximated by the Benjamin-Ono equation and established a better cubic energy estimate.
The author and Rowan also use the holomorphic coordinates in \cite{rowan2023dimensional,rowan2024} to prove the existence of 2D deep gravity and gravity-capillary solitary waves with constant vorticity.

Before going to the main results in this article, we recall some important results  in the work of Ifrim-Tataru \cite{MR3869381} and Ifrim-Rowan-Tataru-Wan \cite{MR4462478}, which are considered the basis of our results in this article.
The first result is the main local well-posedness result in \cite{MR3869381}.
\begin{theorem}[\hspace{1sp}\cite{MR3869381}]
 Let $n\geq 1$. The system $\eqref{e:CVWW}$ is locally well-posed in $\mathcal{H}^n$ for initial data $(W_0, Q_0)$ with the following regularity:
 \begin{equation*}
     (W_0, Q_0)\in \mathcal H, \quad (\mathbf{W}_0, R_0) \in \doth{1},
 \end{equation*}
 and satisfying the pointwise constraints
 \begin{align*}
   |\mathbf{W}(\alpha)+1|>\delta>0&  \quad\text{(no interface singularities)}, \\
    g+ \underline{a}(\alpha) >\delta>0& \quad\text{(Taylor sign condition)}.
 \end{align*}
 Furthermore, the solution can be continued for as long as $\uA$, $\uB$ remain bounded and the pointwise conditions above hold uniformly. 
 The same result holds in the periodic setting.
 \label{t:AiResult}
\end{theorem}

Next, it was proved in \cite{MR4462478} the following energy estimates for small initial data.
\begin{theorem}[\hspace{1sp}\cite{MR4462478}]\label{t:oldcubicenergy}
 For any $n\geq 0$ there exists an energy functional $E^{n,(3)}$ which has the following properties as long as $\underline{A}\ll 1$:
 \begin{enumerate} [label=(\roman*)]
     \item Norm equivalence:
       \begin{equation*}
           E^{n,(3)}(\mathbf{W},R) = (1+O(\underline{A}))\mathcal{E}_0 (\partial^n \mathbf{W}, \partial^n R) + O(\gamma^4 \uA^2)\mathcal{E}_0 (\partial^{n-1}\mathbf{W}, \partial^{n-1}R).
       \end{equation*}
     \item Cubic energy estimates:
      \begin{align*}
     \frac{d}{dt}E^{n,(3)}\lesssim_{\uA}& \left(\gamma^2AB+\gamma^3A^2+\gamma^3BA_{-1/2}+\gamma^4AA_{-1/2}+\gamma^4A_{-1}B+\gamma^5 AA_{-1}\right)\nonumber\\
     & \cdot(E^{n,(3)}E^{n-1,(3)})^{\frac 1 2} +\uA\uB E^{n,(3)}.
 \end{align*}
      Here if $n=0$ then $\mathcal{E}_0 (\partial^{-1}\mathbf{W}, \partial^{-1}R)$ is naturally replaced by $\mathcal{E}(W,Q)$.
      $\mathcal{E}_0$ is defined in \eqref{ConservedEnergy}, and $\mathcal{E}$ is the conserved energy \eqref{HamiltonianEnergy}.
 \end{enumerate}
\end{theorem}

We also recall that in \cite{MR3869381} the normal form corrections which eliminate the quadratic parts of  source terms in \eqref{e:CVWW}:

\begin{equation} \label{e:NormalForm}
\begin{aligned}
W^{[2]} = &-(W +\bar{W})W_\alpha - \frac{\gamma}{2g}[(Q+\bar{Q})W_\alpha +(W+\bar{W})Q_\alpha ]\\
& + i\frac{\gamma^2}{g}\left[(\partial^{-1}W - \partial^{-1}\bar{W})W_\alpha +W^2 + \frac{1}{2}|W|^2\right] - \frac{\gamma^2}{4g^2}(Q+\bar{Q})Q_\alpha \\
& + i\frac{i
\gamma^3}{4g^2}[(Q+\bar{Q})W+(\partial^{-1}W - \partial^{-1}\bar{W})Q_\alpha]+ \frac{\gamma^4}{4g^2}(\partial^{-1}W - \partial^{-1}\bar{W})W,\\
Q^{[2]} = & -(W+\bar{W})Q_\alpha - \frac{\gamma}{2g}(Q+\bar{Q})Q_\alpha + i\frac{\gamma}{4}(W^2+2|W|^2)\\
& + i\frac{\gamma^2}{2g}\left[(\partial^{-1}W - \partial^{-1}\bar{W})Q_\alpha + \frac{1}{2}(Q+\bar{Q})W\right]+\frac{\gamma^3}{4g}(\partial^{-1}W - \partial^{-1}\bar{W})W. 
\end{aligned}
\end{equation}
In other words, the new variables $(\tilde W, \tilde Q) : = (W+ \nP W^{[2]}, Q+ \nP Q^{[2]})$ solve the system
\begin{equation*}
\left\{
             \begin{array}{lr}
             \tilde{W}_t + \tilde{Q}_\alpha = G^{\geq 2}(W, Q) &  \\
             \tilde{Q}_t - ig\tilde{W}  +i\gamma \tilde{Q}  = K^{\geq 2}(W, Q),&
             \end{array}
\right.
\end{equation*}
where the right-hand side only contains the cubic and higher order nonlinearity of \eqref{e:CVWW}.

We remark that \eqref{e:NormalForm} is the normal form transformation for the original water wave system \eqref{e:CVWW}, but not for the differentiated system \eqref{differentiatedEqn}.
In order to obtain a low regularity energy estimate for the system \eqref{differentiatedEqn}, we need to build  new paradifferential normal forms  that are similar but different from \eqref{e:NormalForm}.
Inspired by the above normal forms, the authors in  \cite{MR3869381} construct the modified cubic energy that enables them to establish the local well-posedness result.
Later in Section \ref{s:LinearEstimate}, we will also make use of the above normal form transformation, but rather at the paradifferential level, to construct the paradifferential corrections for the linearized system.

\subsection{The main results}
Our balanced energy estimate is the following:
\begin{theorem}
 For any $s\geq 0$ there exists an energy functional $E_s$ which has the following properties as long as $\max \{\uAS, \uA\} \ll 1$:
 \begin{enumerate} [label=(\roman*)]
     \item Norm equivalence:
       \begin{equation}
           E_s(\mathbf{W},R) = (1+O(\underline{A}))\|(\mathbf{W},  R)\|^2_{\doth{s}} + O(\gamma^4 \underline{A}^2)\|(\mathbf{W}, R)\|^2_{\doth{s-1}}.\label{normEquivalence}
       \end{equation}
     \item Balanced energy estimates:
      \begin{equation}
     \frac{d}{dt}E_s\lesssim_{\uAS} \Astar(\gamma^{\frac{1}{4}}+\Astar) E_s.\label{strongcubicest}
 \end{equation}
 \end{enumerate} \label{t:MainEnergyEstimate}
\end{theorem}

Compared to the previous energy estimates Theorem \ref{t:oldcubicenergy}, our energy estimates get improved in the following way:
\begin{enumerate}
\item We take advantage of the paradifferential structure of the water waves system.
Our energy estimates no longer depend on the control norm $\uB$, they now only depend on the control norm $\Astar$, which has $\frac{1}{4}$ less derivatives compared to $\uB$.
This effectively lowers the regularity for well-posedness.
Our normal form corrections are also  at the paradifferential level in the same manner as done in the breakthrough work of Ai-Ifrim-Tataru \cite{ai2023dimensional}.
\item Compared to the cubic energy estimates in Theorem \ref{t:oldcubicenergy}, the right-hand side of our energy estimates no longer depends on $E_{s-1}$.
As a result, we only need to apply Gronwall's inequality for once and don't need to use the induction as in \cite{MR3869381}.
\item We allow the index $s$ to be non-integer, which is more flexible compared to the previous one.
\end{enumerate}

Having established the energy estimate for the full equations, we combine it with the energy estimate for the linearized equations Theorem \ref{t:LinearizedWellposed} to get the following low regularity well-posedness result.
\begin{theorem}
 Let $s \geq s_0 = \frac{3}{4}$. The system \eqref{differentiatedEqn} is locally well-posed for  initial data $(\W_0, R_0 )$ in $\mathcal{H}^s(\mathbb{R})$ (or $\mathbb{T}$)  with the following regularity:
\begin{equation*}
(\W_0, R_0)\in \doth{\frac{3}{4}}, \qquad \gamma^2 (\W_0, R_0)\in \doth{-\frac{1}{4}},
\end{equation*}
and such that $\underline{A}^\sharp_{0} = \uAS(\W_0, R_0)$, $\uA_{0} = \uA(\W_0, R_0)$ are small.
Furthermore, the solution can be continued for as long as $\max \{\uAS, \uA\}$ remains bounded and $\Astar(\gamma^{\frac{1}{4}}+\Astar) \in L^1_t$.
\label{t:MainWellPosedness}
\end{theorem}
Although for the second initial condition one can just write $(\W_0, R_0)\in \doth{\frac{3}{4}}\cap\doth{-\frac{1}{4}}$, we keep the coefficient $\gamma^2$ here for book-keeping, and also at time $t=0$, the normal form energy $E_{\frac{3}{4}}(\W_0, R_0)$ constructed in \eqref{normEquivalence} is bounded.
The control norms at time zero $\underline{A}^\sharp_{0}$ and $\uA_{0}$ need to be chosen small so that  $O(\underline{A})$ and $O(\gamma^4 \underline{A})$ in \eqref{normEquivalence} are sufficiently small compared to $1$ and the implicit constant that depends on $\underline{A}^\sharp$ in \eqref{strongcubicest} is bounded at time $t=0$.

\begin{remark}
For local well-posedness, we mean that for all initial data $(\W_0, R_0)$  satisfying the above conditions,  there exists a positive time $T>0$ depending only on the size of initial data, such that the following holds:
\begin{enumerate}
\item \textit{Unique existence for the regular solutions}: If the initial data is in $\mathcal{H}^s$ that satisfies the above conditions for some integer $s\geq 1$, then there exist a unique solution $(\W, R)\in \mathcal{H}^s$ in $[0,T]$, with the property that
\begin{equation*}
    \|(\W, R)\|_{C[0,T; \mathcal{H}^{s^{'}}]} \lesssim \|(\W_0, R_0)\|_{C[0,T; \mathcal{H}^{s^{'}}]}, \quad 1\leq s^{'}\leq s.
\end{equation*}
\item \textit{Existence for rough solutions}: For initial data in $\mathcal{H}^{\frac{3}{4}}$, the solution is in $C[0,T;\mathcal{H}^{\frac{3}{4}}]$, and it can be treated as the unique strong limit of smooth solutions.
\item \textit{Continuous dependence on the  data for rough solutions}: If a sequence $(\W_j, R_j)(0)$ that converges to $(\W_0, R_0)$ in $\mathcal{H}^{\frac{3}{4}}$ topology, then $(\W_j, R_j)(t)$ also converges to $(\W, R)(t)$ in $\mathcal{H}^{\frac{3}{4}}$, for $t\in [0,T]$.
\end{enumerate}
\end{remark}
\begin{remark}
In \cite{MR3869381}, in order to prove the local well-posedness, the initial data needs to satisfy the no interface singularities condition: $|\W(\alpha)+1|>\delta >0$, and the Taylor sign condition: $g+\ua(\alpha)>0$.
Later in \eqref{uaBMO}, we will prove that $\|\ua\|_{L^\infty}\lesssim \uA(1+\uA)$, so that these two conditions are automatically satisfied when $\uA$ is small in our setting.

Compared to the previous well-posedness result Theorem \ref{t:AiResult}, our result lowers $\frac{1}{4}$ Sobolev regularity.
For simplicity, we will focus on the case of the real line.
For the periodic setting, we refer the reader to the discussion in Appendix $A$ of \cite{hunter2016two} for the minor changes.
\end{remark}

\subsection{The structure of the article} In this article, we use plenty of estimates in paradifferential calculus.
The full system and the linearized system will be rewritten in terms of paradifferential equations.
To get ready for this, we review in Section \ref{s:Def} the definitions and notations of paraproducts and some of the classical Coifman-Meyer paraproduct estimates.
We also introduce the necessary definitions of norms and function spaces we will be using.

Then in Section \ref{s:WaterBound} we apply the paraproduct estimates to auxiliary functions for both Sobolev and $BMO$ bounds.
We consider the corresponding bounds for frequency-shift $\ua$, advection velocity $\ub$ and auxiliary functions $Y$ and $\underline{M}$.
In addition, we compute the leading terms of  para-material derivatives of $\W, R, W, Y, X: = T_{1-Y}W, Z:= T_{1-Y}Q,$ and $U:= T_{1-Y}\partial_\alpha^{-1}W$.

Having established the necessary bounds, we consider in Section \ref{s:LinearEstimate} the linearized system and its energy estimate in $\doth{\frac{1}{4}}$.
In order to apply the paradifferential calculus, we first rewrite the linearized system as a system of  paradifferential equations.
Then we consider the estimates for both linear paradifferential flow and the source terms.
The detailed ideas will be described within this section.

Next, we turn our attention to the energy estimate for the full system in Section \ref{s:FullEqn}.
Again we first reduce the full system to another system of paradifferential equations.
In the following, we show that the paradifferential equations can be further reduced to the linearized paradifferential equations with unbalanced source terms.
In the end, we use paradifferential corrections to eliminate these unfavourable source terms.

Finally, collecting the energy estimates for both linearized and the full systems, we give an outline of proof for the local well-posedness of the water wave system in Section \ref{s:Proof}.\\

\textbf{Acknowledgments.} The author would like to thank Mihaela Ifrim and Albert Ai for introduction and discussion of many important details in their work \cite{ai2023dimensional}.

\section{Definition of norms and estimates} \label{s:Def}
In this section we review the definition of norms and estimates we will use later in the article.
Our analysis primarily relies on the paradifferential calculus, especially the paraproduct type estimates.
\subsection{Norms and function spaces}
We begin with  Littlewood-Paley frequency decomposition
\begin{equation*}
    I = \sum_{k\in \mathbb{Z}} P_k,
\end{equation*}
where $P_k$ are smooth symbols with frequency localized at $2^k$.
Most of our analysis is at the level of homogeneous Sobolev spaces $\dot{H}^s$, whose norms are given by
\begin{equation*}
  \|f\|_{\dot{H}^s}  \sim \|(\sum_{k \in \mathbb{Z}} |2^{ks} P_k f|^{2})^{\frac{1}{2}} \|_{L^2}  = \|2^{ks} P_k f\|_{L^2_\alpha l^2_k}.
\end{equation*}
We recall the Littlewood-Paley square function and its restricted version
\begin{equation*}
    S(f)(\alpha) : = \left(\sum_{k\in \mathbb{Z}}|P_k f(\alpha)|^2\right)^\frac{1}{2}, \quad S_{>k}(f)(\alpha) : = \left(\sum_{j>k}|P_j f(\alpha)|^2\right)^\frac{1}{2}.
\end{equation*}
The $BMO$ space can be defined using the following  square function characterization:
\begin{equation*}
    \|f\|_{BMO}^2 : = \sup_k \sup_{|Q| = 2^{-k}}2^k\int_Q |S_{>k}(f)|^2 \,d\alpha.
\end{equation*}
For real number $s$ we can define the homogeneous spaces  $BMO^s$ with norm
\begin{equation*}
    \|f\|_{BMO^s} : = \||D|^s f \|_{BMO}.
\end{equation*}
We will also need the following homogeneous Besov spaces $\dot{B}^s_{\infty, 2}$, whose norms are defined by
\begin{equation*}
    \|f\|_{\dot{B}^s_{\infty, 2}} \sim \left(\sum_{k\in \mathbb{Z}}  2^{2sk} \| P_k f\|^2_{L^\infty} \right)^{\frac{1}{2}}.
\end{equation*}
We have the embedding property $\dot{B}^s_{\infty, 2} \hookrightarrow BMO^s$.

\subsection{ Paraproduct and  Moser type estimates} For the product of two functions $fg$, we  use the Littlewood-Paley paraproduct type decomposition to decompose it as
\begin{equation*}
    fg  = \sum_{k\in \mathbb{Z}} f_{<k-4}g_k + \sum_{k\in \mathbb{Z}} f_{k}g_{<k-4} + \sum_{|k-l|<4}f_k g_l : = T_f g+ T_g f + \Pi(f,g).
\end{equation*}
For the paraproducts, we have the H\"{o}lder type  estimates
\begin{align*}
    &\| T_f g\|_{L^r} \leq \|f\|_{L^p} \|g\|_{L^q}, \quad \frac{1}{r} = \frac{1}{p}+ \frac{1}{q}, \quad 1<p\leq \infty,\quad 1<q,r<\infty, \\
     &\| \Pi (f, g)\|_{L^r}\leq \|f\|_{L^p} \|g\|_{L^q}, \quad \frac{1}{r} = \frac{1}{p}+ \frac{1}{q}, \quad 1<p,q,r<\infty.
\end{align*}
In the case of $q= \infty$, the right-hand side is replaced by  $BMO$ norms:
\begin{equation}
    \| T_f g\|_{L^p} + \| \Pi (f, g)\|_{L^p}\leq \||D|^{-s}f\|_{L^p} \|g\|_{BMO^s}, \quad 1<p<\infty, s\geq 0. \label{LpBalance}
\end{equation}
For the estimate of $T_f g$ term we have the weaker bound
\begin{equation}
    \|T_f g\|_{L^p} \leq \|g\|_{\dot{W}^{s,p}}\|f\|_{BMO^{-s}}, \quad 1<p<\infty, s>0. \label{LpLowhigh}
\end{equation}
We also have the following commutator estimates:
\begin{lemma}[\hspace{1sp}\cite{ai2023dimensional}] \label{t:CommutatorL2}
 The following commutator estimates hold for $1<p<\infty$:
 \begin{align*}
     &\||D|^s [\nP,g]|D|^\sigma f \|_{L^p} \lesssim \| |D|^{s+\sigma}g\|_{BMO}\| f\|_{L^p}, \qquad \sigma\geq 0, \quad s\geq 0,\\
      &\||D|^s [\nP,g]|D|^\sigma f \|_{L^p} \lesssim \| |D|^{s+\sigma}g\|_{L^p}\| f\|_{BMO}, \qquad \sigma> 0, \quad s\geq 0.
 \end{align*}
\end{lemma}
 Later these commutator estimates will be applied to functions which are either holomorphic or anti-holomorphic.

Next we consider some  product type estimates involving $BMO$ or $L^\infty$ norms.
\begin{lemma}[\hspace{1sp}\cite{ai2023dimensional}]
\begin{enumerate}
    \item The following estimates hold:
    \begin{align}
       & \|\Pi(u,v) \|_{BMO}\lesssim \|u\|_{BMO}\|v\|_{BMO}, \label{BMOBalance}\\
        &\|P_{\leq k} \Pi_{\geq k}(u,v)\|_{L^\infty} \lesssim \|u\|_{BMO}\|v\|_{BMO},\\
       & \|T_u v\|_{BMO}\lesssim \|u\|_{L^\infty}\|v\|_{BMO},\\
       & \| T_u v\|_{BMO}\lesssim \|u\|_{BMO^{-s}}\|v\|_{BMO^{s}}, \qquad s>0. \label{BMOSigma}
    \end{align}
    \item For $s>0$, the space $L^\infty\cap BMO^s$ is an algebra and satisfies the estimate
    \begin{equation}
        \|uv\|_{BMO^s}\lesssim \|u\|_{L^\infty} \| v\|_{BMO^s}+ \|v\|_{L^\infty} \| u\|_{BMO^s}.
    \end{equation}
    \item Furthermore, for any smooth function $F$ that vanishes at $0$, then the following Moser type  estimates hold:
    \begin{equation}
       \|F(u)\|_{BMO^s}\lesssim_{\|u\|_{L^\infty}} \|u\|_{BMO^s}. \label{MoserOne}
    \end{equation}
\end{enumerate}
\end{lemma}
As for the estimates in $\dot{H}^s$, we have that
\begin{lemma}[\hspace{1sp}\cite{ai2023dimensional}]
    Let $s>0$, then $\dot{H}^s \cap L^\infty$ is an algebra, with the estimate
\begin{equation}
\|uv\|_{\dot{H}^s} \lesssim \|u\|_{\dot{H}^s} \|v\|_{L^\infty} + \|u\|_{L^\infty}\|v\|_{\dot{H}^s}.
\end{equation}
 Furthermore, for any smooth function $F$ that vanishes at $0$, then the following Moser type  estimates hold:
 \begin{equation}
\|F(u)\|_{\dot{H}^s}\lesssim_{\|u\|_{L^\infty}} \|u\|_{\dot{H}^s}. \label{MoserTwo}
    \end{equation}
\end{lemma}

Below we record the following para-commutators, para-products, and para-associativity lemmas.
\begin{lemma}[\hspace{1sp}\cite{ai2023dimensional}]
\begin{enumerate}
    \item (Para-commutators)
   Assume that $s_1, s_2<1$, then we have that
   \begin{equation}
       \|T_fT_g -T_gT_f\|_{\dot{H}^s \rightarrow \dot{H}^{s+s_1+s_2}} \lesssim \||D|^{s_1}f\|_{BMO}\||D|^{s_2}g\|_{BMO}.\label{ParaCommutator}
   \end{equation}
    \item (Para-products)
   Assume that $s_1, s_2<1$, and $s_1+s_2\geq 0$, then
   \begin{align}
       &\|T_fT_g -T_{fg}\|_{\dot{H}^s \rightarrow \dot{H}^{s+s_1+s_2}} \lesssim \||D|^{s_1}f\|_{BMO}\||D|^{s_2}g\|_{BMO},\label{ParaProducts} \\
        &\|T_fT_g -T_{fg}\|_{\dot{W}^{s,4} \rightarrow \dot{W}^{s+s_1+s_2,4}} \lesssim \||D|^{s_1}f\|_{BMO}\||D|^{s_2}g\|_{BMO}.\label{ParaProductsTwo}
   \end{align}
   \item (Para-associativity) For $s+s_2\geq 0$, $s+s_1+s_2\geq 0$, and $s_1<1$, we have
   \begin{align}
       &\|T_f\Pi(v,u) - \Pi(v, T_f u)\|_{\dot{H}^{s+s_1+s_2}} \lesssim \| |D|^{s_1}f\|_{BMO}\||D|^{s_2}v\|_{BMO}\| u\|_{\dot{H}^{s}},\label{ParaAssociateOne}\\
       &\|T_f\Pi(v,u) - \Pi(v, T_f u)\|_{\dot{W}^{s+s_1+s_2,4}} \lesssim \| |D|^{s_1}f\|_{BMO}\||D|^{s_2}v\|_{BMO}\| u\|_{\dot{W}^{s,4}}. \label{ParaAssociateThree}
   \end{align}
    \item For $s_1, s_2< 1$, $s+s_1+s_2\geq 0$, $s+s_2 \geq 0$ and $\bar{v} = \bar{P}\bar{v}$, we have
   \begin{equation}
       \|T_fP(\bar{v}u) - P(\bar{v}T_f u)\|_{\dot{H}^{s+s_1+s_2}} \lesssim \| |D|^{s_1}f\|_{BMO}\||D|^{s_2}v\|_{BMO}\| u\|_{\dot{H}^{s}}. \label{ParaAssociateTwo}
   \end{equation}
\end{enumerate}
\end{lemma}
Finally, we record here the para-Leibniz rule.
Define the \textit{para-material derivative} to be
\begin{equation*}
    T_{D_t} := \partial_t + T_{\ub}\partial_\alpha.
\end{equation*}
We then consider the following four versions of para-Leibniz errors.
The first two are unbalanced para-Leibniz errors
\begin{align*}
    &E^p_L(u,v) = T_{D_t}T_u v - T_{T_{D_t}u}v -T_uT_{D_t}v,\\
    & \tilde{E}^p_L(u,v) = T_{D_t}T_u v - T_{D_{t} u}v -T_uT_{D_t}v.
\end{align*}
The other two are the balanced para-Leibniz errors
\begin{align*}
    E^\pi_L(u,v) &= T_{D_t}\Pi(u,v) -\Pi(T_{D_t}u,v) -\Pi(u, T_{D_t}v), \\
    \tilde{E}^\pi_L(u,v) &= T_{D_t}\Pi(u,v) -\Pi(D_t u,v) -\Pi(u, T_{D_t}v).
\end{align*}
With above notations, the Leibniz error can be bounded according to the following lemma.
\begin{lemma}[\hspace{1sp}\cite{ai2023dimensional}]
\begin{enumerate}
\item For the unbalanced para-Leibniz error $E^p_L(u,v)$ we have the bounds
\begin{align}
    &\|E^p_{L}(u,v)\|_{\dot{H}^s} \lesssim \uA_{\frac{1}{4}} \|u\|_{BMO^{\frac{1}{4}-\sigma}}\|v\|_{\dot{H}^{s+\sigma}}, \quad \sigma >0, \label{UnbalancedParaLeibnizOne}\\
    &\|E^p_{L}(u,v)\|_{\dot{H}^s} \lesssim \uA_{\frac{1}{4}} \|u\|_{\dot{H}^{\frac{1}{4}-\sigma}}\|v\|_{BMO^{s+\sigma}}, \quad \sigma >0. \label{UnbalancedParaLeibnizTwo}
\end{align}
In the case $\sigma=0$ the same bounds \eqref{UnbalancedParaLeibnizOne} and \eqref{UnbalancedParaLeibnizTwo} hold for $\tilde{E}^p_L(u,v)$ with $\sigma =0$.
\item For the balanced para-Leibniz error  $E^\pi_L(u,v)$ we have the estimate
\begin{equation}
    \|E^\pi_L(u,v) \|_{\dot{H}^s} \lesssim \uA_{\frac{1}{4}} \|u\|_{BMO^{\frac{1}{4}-\sigma}}\|v\|_{\dot{H}^{s+\sigma}}, \qquad \sigma \in \mathbb{R},\quad s\geq 0. \label{BalancedLeibniz}
\end{equation}
In the case $\sigma =0$, we also have the same bound for $\tilde{E}^\pi_L(u,v)$.
\end{enumerate} \label{t:Leibniz}
\end{lemma}
The proof of this lemma is almost identical to Lemma $3.6$ in \cite{ai2023dimensional}, we ask the interested reader to check the proof there.
The only difference is that we replace the bound for $b$ by the corresponding estimate for $\ub$.

\section{Water waves related bounds} \label{s:WaterBound}
In this section, we first consider the Sobolev and $BMO$ bounds of auxiliary functions $Y$, $\ua$, $\ub$, and $\underline{M}$, then we compute the leading terms of  para-material derivatives of $\W, R, Y, X, Z, U$ and $\ua$.
These estimates will play a role in the construction of normal form energies in later sections.
\subsection{Sobolev and $BMO$ bounds}
We begin with the estimates for the auxiliary function $Y: = \frac{\W}{1+\W}$.
Applying Moser type estimates \eqref{MoserOne} and \eqref{MoserTwo}, one get
\begin{lemma}[\hspace{1sp}\cite{ai2023dimensional}]
 The function $Y$ satisfies the  BMO bound
 \begin{equation}
     \||D|^{\frac{1}{4}}Y\|_{BMO} \lesssim_A A_{\frac{1}{4}}, \label{YBMO}
 \end{equation}
as well as the Sobolev bounds
\begin{equation*}
    \|Y\|_{\dot{H}^{\sigma}} \lesssim_A \|\W\|_{\dot{H}^{\sigma}}, \quad \||D|^{\sigma}Y\|_{L^4}\lesssim \|\W\|_{\dot{W}^{\sigma,4}}, \qquad \sigma \geq 0.
\end{equation*}
\end{lemma}

We continue with the bounds for the Taylor coefficient $a$.
\begin{lemma}[\hspace{1sp}\cite{ai2023dimensional}]
The Taylor coefficient $a$ is nonnegative and satisfies the $BMO$ bound and the uniform bound
\begin{equation*}
    \|a\|_{BMO} \lesssim \|R\|^2_{BMO^{\frac{1}{2}}}, \qquad \|a\|_{L^\infty} \lesssim \|R\|^2_{\dot{B}^{\frac{1}{2}}_{\infty, 2}}.
\end{equation*}
In addition, it satisfies
\begin{equation*}
\||D|^{\frac{1}{4}}a \|_{BMO}\lesssim AA_{\frac{1}{4}}, \quad  \||D|^{\frac{1}{2}}a \|_{BMO}\lesssim A^2_{\frac{1}{4}},
\end{equation*}
and the $L^4$ based bounds
\begin{equation*}
    \||D|^{\frac{1}{2}}a\|_{L^4}\lesssim A^\sharp A_{\frac{1}{4}}, \qquad \| a\|_{L^4}\lesssim A^\sharp A_{-\frac{1}{4}}.
\end{equation*}
\end{lemma}
For the $BMO$ bound of $a_1$, we can rewrite $N$ as
\begin{equation*}
    N = T_{\bar{R}_\alpha}W +\Pi(W, \bar{R}_\alpha) - T_{\bar{\W}}R - \Pi(\bar{\W}, R)+ T_{R_\alpha}\bar{W}+ \Pi(\bar{W}, R_\alpha)-T_{\W}\bar{R} - \Pi(\W, R_\alpha).
\end{equation*}
We use \eqref{BMOSigma} to estimate the low-high portion of $N$ and \eqref{BMOBalance} to estimate the high-high portion of $N$:
\begin{align*}
    &\|N\|_{BMO} \lesssim A_{-\frac{1}{4}}^2, \quad \gamma\|N\|_{BMO} \lesssim \gamma AA_{-\frac{1}{2}} \lesssim \uA^2, \quad \||D|^{\frac{1}{4}}N\|_{BMO} \lesssim AA_{-\frac{1}{4}}, \\
    &  \gamma\||D|^{\frac{1}{4}}N\|_{BMO} \lesssim \gamma AA_{-\frac{1}{4}} \lesssim A\uA_{\frac{1}{4}}.
\end{align*}
For the $L^\infty$ bound of $a_1$, it is  shown in the Proposition  $A.4$ of \cite{MR3869381} that
\begin{equation*}
    \|a_1\|_{L^\infty} \lesssim_A A_{-\frac{1}{2}}(1+A).
\end{equation*}
As for the $L^4$ based bound of $N$, we write $N$ as
\begin{equation*}
    N = [\bar{\nP}, \bar{W}]R_\alpha - [\bar{\nP}, \bar{R}]\W + [\nP, W]\bar{R}_\alpha - [\nP, R]\bar{\W}.
\end{equation*}
Using Lemma \ref{t:CommutatorL2}, we have
\begin{align*}
    &\gamma\||D|^{\frac{1}{2}}[\bar{\nP}, \bar{W}]R_\alpha \|_{L^4}+\gamma\||D|^{\frac{1}{2}}[\nP, W]\bar{R}_\alpha \|_{L^4} \lesssim \gamma \||D|^{\frac{3}{4}}W\|_{BMO}\||D|^{\frac{3}{4}}R\|_{L^4} \lesssim A^\sharp \Astar, \\
     &\gamma\||D|^{\frac{1}{2}}[\bar{\nP}, \bar{R}]\W \|_{L^4}+\gamma\||D|^{\frac{1}{2}}[\nP, R]\bar{\W} \|_{L^4} \lesssim \gamma \||D|^{\frac{3}{4}}R\|_{BMO}\||D|^{\frac{3}{4}}W\|_{L^4} \lesssim \uAS \Astar,
\end{align*}
and the same estimates hold for complex conjugates.
As a consequence,
\begin{equation*}
\gamma\||D|^{\frac{1}{2}}N \|_{L^4}\lesssim \uAS\Astar, \quad \gamma\|N\|_{L^4} \lesssim A^\sharp A_{-\frac{3}{4}}.
\end{equation*}

Combining the estimates above, we get:
\begin{lemma}
 The  frequency-shift $\ua$ satisfies the $BMO(L^\infty)$ bounds
 \begin{equation}
       \||D|^{\frac{1}{4}}\ua \|_{BMO} \lesssim \uA_{\frac{1}{4}}(1+A), \quad \|\ua\|_{L^\infty}\lesssim_{A} \uA(1+A),\label{uaBMO}
 \end{equation}
 and the $L^4$ based estimates
 \begin{equation}
     \||D|^{\frac{1}{2}}\ua\|_{L^4} \lesssim (\Astar + \gamma^{\frac{1}{2}})\uAS, \qquad \gamma\|\ua\|_{L^4} \lesssim (\Astar + \gamma^{\frac{1}{2}})\uAS. \label{uaAsharp}
 \end{equation}
\end{lemma}

Next, we recall the bounds for the transport coefficient $b$.
\begin{lemma}[\hspace{1sp}\cite{ai2023dimensional}]
Let $s>0$, then the transport coefficient $b$ satisfies
\begin{equation*}
    \||D|^sb\|_{BMO}\lesssim_A \||D|^s R\|_{BMO}, \quad \|b\|_{\dot{H}^s} \lesssim_A \|R\|_{\dot{H}^s}, \quad \||D|^s b\|_{L^4}\lesssim_A \||D|^s R\|_{L^4}.
\end{equation*}
In particular, we have
\begin{equation*}
     \||D|^{\frac{1}{2}}b\|_{BMO}\lesssim_A A, \quad  \||D|^{\frac{3}{4}}b\|_{BMO}\lesssim_A A_{\frac{1}{4}}, \quad \||D|^{\frac{3}{4}}b\|_{L^4}\lesssim_A A^\sharp.
\end{equation*}
\end{lemma}
For the bound of $b_1$, we  rewrite $b_1$ as
\begin{align*}
    b_1 &= \nP[W(1-\bar{Y})] - \bar{P}[\bar{W}(1-Y)] = W-\bar{W}+\bar{\nP}[\bar{W}Y]-\nP[W\bar{Y}]\\
    & = W-\bar{W}+T_Y \bar{W} + \bar{\nP}\Pi(\bar{W}, Y)- T_{\bar{Y}}W -\nP\Pi(W,\bar{Y}).
\end{align*}
Again we use \eqref{BMOBalance} and \eqref{BMOSigma} to estimate for $s> 0$,
\begin{equation*}
    \||D|^s b_1\|_{BMO} \lesssim_A \||D|^s W\|_{BMO}, \qquad  \| b_1\|_{\dot{H}^s} \lesssim_A \| W\|_{\dot{H}^s}, \quad \|b_1\|_{\dot{W}^{s,4}}\lesssim_A \|W\|_{\dot{W}^{s,4}}.
\end{equation*}
Gathering together the bounds for $b$ and $b_1$, we obtain the following result for $b_1$:
\begin{lemma}
   The advection velocity $\ub$ satisfies the estimates
   \begin{equation}
       \||D|^{\frac{1}{2}}\ub \|_{BMO} \lesssim_A \uA, \quad \||D|^{\frac{3}{4}}\ub \|_{BMO} \lesssim_A \uA_{\frac{1}{4}}, \quad \||D|^{\frac{3}{4}}\ub\|_{L^4}\lesssim_A \uAS. \label{ubBMO}
   \end{equation}
\end{lemma}

As for the auxiliary functions $M$ and $M_1$, we have
\begin{lemma}[\hspace{1sp}\cite{ai2023dimensional, MR3869381}]
\begin{enumerate}
\item  The function $M$ satisfies the $L^\infty$ bound
\begin{equation}
    \|M\|_{L^\infty} \lesssim_A A^2_{\frac{1}{4}},\label{MBMOOne}
\end{equation}
as well as the Sobolev bounds
\begin{equation*}
   \|M\|_{\dot{H}^{s-\frac{1}{2}}} \lesssim_A A \|(\W, R)\|_{\doth{s}}, \quad s\geq 0.
\end{equation*}
\item The auxiliary function $M_1$ satisfies the $L^\infty$ bound
  \begin{equation}
      \|M_1\|_{L^\infty} \lesssim_A A^2. \label{MoneLinfinity}
  \end{equation}
\end{enumerate}
\end{lemma}
Combining the estimates \eqref{MBMOOne} and \eqref{MoneLinfinity}, we immediately get the $L^\infty$ bound for $\underline{M}$
\begin{equation}
    \|\underline{M}\|_{L^\infty} \lesssim_A \uA_{\frac{1}{4}}^2. \label{uMBound}
\end{equation}

\subsection{Leading terms of para-material derivatives} \label{s:TDtBound}
The material derivative $D_t = \partial_t + \ub \partial_\alpha$ is very important in the water waves system.
At the paradifferential level it is replaced by the para-material derivative $T_{D_t} = \partial_t + T_\ub \partial_\alpha$.
In this subsection, we compute the leading term of para-material derivatives of various functions.

\begin{lemma}
We have the following results on para-material derivatives of $W, \W$ and $R$:
\begin{enumerate}
    \item Para-material derivative of $W$:
\begin{equation}
    T_{D_t}W = -T_{1+W_\alpha} \nP[(1-\bar{Y})R]+ i\frac{\gamma}{2}T_{1+W_\alpha} \nP[W(1-\bar{Y})]-\nP\Pi(W_\alpha, \ub)-i\frac{\gamma}{2}W. \label{ParaW}
\end{equation}
    \item Para-material derivatives of $(\W,R)$:
    \begin{align*}
   &T_{D_t}\W = -T_{\ub_\alpha}\W -i\frac{\gamma}{2}\W-\nP\partial_\alpha \left[T_{1+W_\alpha} [(1-\bar{Y})(R- i\frac{\gamma}{2} W)]+\Pi(\W, \ub)\right],\\
   &T_{D_t}R = -i\gamma R-\nP T_{R_\alpha}\ub -\nP\Pi(R_\alpha, \ub)+i\nP[(g+\ua)Y] -\nP[R\bar{R}_\alpha]-i\frac{\gamma}{2}\nP[W\bar{R}_\alpha-\bar{W}_\alpha R].
    \end{align*}
    \item Leading terms of  para-material derivatives of $(\mathbf{W}, R)$:
     \begin{equation}
\left\{
             \begin{array}{lr}
            T_{D_t} \W +T_{1+\W}T_{1-\bar{Y}}R_\alpha = G  &\\
             T_{D_t} R -iT_{g+\ua}Y +i\gamma R = K,&
             \end{array} \label{ParaWalphaR}
\right.
\end{equation}
where the source terms $(G,K)$ in \eqref{ParaWalphaR} satisfy the BMO bound
\begin{equation*}
    \|G\|_{BMO}+\|K\|_{BMO^{\frac{1}{2}}} \lesssim_\uA \uA^2_{\frac{1}{4}},
\end{equation*}
and the Sobolev bound
\begin{equation*}
    \|G\|_{\dot{W}^{\frac{1}{4},4}}+\gamma\|K\|_{\dot{W}^{\frac{1}{4},4}} \lesssim_\uAS \uA_{\frac{1}{4}}\uAStar.
\end{equation*}
\end{enumerate}
\end{lemma}
\begin{proof}
 For the para-material derivative of $W$, we use the first equation of \eqref{e:CVWW1} to write
 \begin{equation*}
     T_{D_t}W = -T_{1+W_\alpha} \ub -\Pi(W_\alpha,\ub)- i\frac{\gamma}{2}W + \bar{R}+i\frac{\gamma}{2}\bar{W}.
 \end{equation*}
 After plugging in the expression of $\ub$ and applying the Littlewood-Paley projection $\nP$, we can eliminate the anti-holomorphic portion and obtain the result in $(1)$.
Note that in $T_{D_t}$, the advection velocity $\ub$ has relative low frequency, so that  the Littlewood-Paley projection $\nP$  freely passes over it without causing any trouble.

 By differentiating \eqref{ParaW}, we obtain the para-material derivative of $\W$.
 Using \eqref{EqnYR}, we have
 \begin{equation*}
    T_{D_t}R = -T_{R_\alpha}\ub -\Pi(R_\alpha, \ub)+i(g+\ua)Y -i\ua-i\frac{\gamma}{2}(R-\bar{R}).
 \end{equation*}
 Applying the projection $\nP$ and using the definition of $\ua$, we  obtain the para-material derivative of $R$.

 For the leading term of $T_{D_t}\W$,  all terms in the expression of $T_{D_t}\W$ have a good balance of derivatives, with a derivative falling on a low frequency variable, except  when the derivative $\partial_\alpha$ falls on $R$.
 Hence, we can rewrite
 \begin{equation*}
    T_{D_t}\W = -T_{1+\W}T_{1-\bar{Y}}R_\alpha +G,
 \end{equation*}
 where the source term $G$ is given by
 \begin{align*}
     G = & \nP T_{1+\W}\Pi(\bar{Y}, R_\alpha)-T_{\ub_\alpha}W_\alpha -\nP T_{\W_\alpha}[(1-\bar{Y})R]+ \nP T_{1+\W} [\bar{Y}_\alpha R] + \nP\partial_\alpha \Pi(\W, \ub)\\
     &+ i\frac{\gamma}{2}\nP T_{\W_\alpha} [(1-\bar{Y})W] -i\frac{\gamma}{2}\nP T_{1+\W} [\bar{Y}_\alpha W] +i\frac{\gamma}{2}\nP T_\W \W -i\frac{\gamma}{2}\nP T_{1+\W} [\bar{Y}\W].
 \end{align*}
 Each term in $G$ has a good balanced of derivatives, and satisfies the desired estimate.
 The computation of leading term of $T_{D_t}R$ in part $(3)$ is straightforward.
 We use the estimates \eqref{uaBMO}, \eqref{uaAsharp}, \eqref{ubBMO} and apply inequalities \eqref{LpBalance}, \eqref{LpLowhigh}, \eqref{BMOBalance} and \eqref{BMOSigma} to get the result.
\end{proof}

Next, we compute the para-material derivative of $Y: = \frac{\W}{1+\W}$.
\begin{lemma}
The leading term of the para-material derivative of $Y$ is given by
\begin{equation}
    T_{D_t} Y = -T_{|1-Y|^2} R_\alpha + G. \label{YParaDerivative}
\end{equation}
The source term $G$ satisfies the $BMO$ bound
\begin{equation*}
    \|G\|_{BMO}\lesssim_A \uA^2_{\frac{1}{4}}.
\end{equation*}
We also have
\begin{equation*}
    \| |D|^{-\frac{1}{4}}T_{|1-Y|^2}R_\alpha\|_{BMO}\lesssim_A \uA_{\frac{1}{4}}.
\end{equation*}
\end{lemma}
\begin{proof}
We expand and rewrite the $Y$ equation of \eqref{EqnYR}.
\begin{align*}
    T_{D_t}Y + T_{|1-Y|^2} R_\alpha =& -T_{R_\alpha} |1-Y|^2 - \Pi(R_\alpha, |1-Y|^2) - T_{Y_\alpha} \ub - \Pi(Y_\alpha, \ub) \\
    &+ (1-Y)\underline{M}+ i\frac{\gamma}{2}\left[Y^2+ \frac{\bar{Y}}{1-\bar{Y}}Y(1-Y)\right].
\end{align*}
The first four terms on the right-hand side have a good balance of derivatives, and they satisfy the $BMO$ bound using \eqref{BMOBalance}, \eqref{BMOSigma} and \eqref{YBMO} inequalities.
For the last two terms, we use \eqref{MoneLinfinity} to get the bound.
Hence the right-hand side can be put into the source term $G$.
As for the estimate of the $T_{|1-Y|^2} R_\alpha$, applying \eqref{BMOSigma} yields the estimate.
\end{proof}

We continue to compute the para-material derivatives of the auxiliary functions $X = T_{1-Y}W$, $Z = T_{1-Y}Q$ and $U = T_{1-Y}\partial^{-1}_\alpha W$.
\begin{lemma}
We have the following results on the para-material derivative of
 $X$:
\begin{enumerate}
\item Para-material derivative of $X$:
\begin{equation*}
    T_{D_t}X = T_{T_{1-\bar{Y}}R_\alpha}X -\nP[(1-\bar{Y})R] +i\frac{\gamma}{2}\nP[(1-\bar{Y})W] -\nP\Pi(X_\alpha, \ub)- i \frac{\gamma}{2}X+ E_1,
\end{equation*}
where for $s+\frac{3}{4}\geq 0$, the error $E_1$ satisfies
\begin{equation*}
    \| E_1\|_{\dot{W}^{\frac{5}{4},4}}+ \gamma^2 \| E_1\|_{\dot{W}^{\frac{1}{4},4}}\lesssim_{\uAS} \underline{A}^2_{\frac{1}{4}}, \qquad \|E_1\|_{\dot{H}^{s+\frac{3}{4}}}\lesssim_{\uA}\| \W\|_{\dot{H}^s}.
\end{equation*}
\item Leading term of the para-material derivative of $X$:
\begin{equation*}
    T_{D_t}X +T_{1-\bar{Y}}R  = E_2,
\end{equation*}
where the error $E_2$ satisfies the $BMO$ bound
\begin{equation*}
\||D|E_2\|_{BMO}+ \gamma^2 \|E_2\|_{BMO}\lesssim_{\uAS} \underline{A}^2_{\frac{1}{4}},
\end{equation*}
and the Sobolev bound
\begin{equation*}
    \| E_2\|_{\dot{W}^{\frac{5}{4},4}}+ \gamma^2 \| E_2\|_{\dot{W}^{\frac{1}{4},4}}\lesssim_{\uAS} \underline{A}_{\frac{1}{4}}\uAStar.
\end{equation*}
\item Leading terms of material and para-material derivatives of $X_\alpha$:
\begin{equation*}
    T_{D_t}X_\alpha +T_{1-\bar{Y}}R_\alpha  = E_3,\qquad  D_t X_\alpha +T_{1-\bar{Y}}R_\alpha = E_3,
\end{equation*}
where the error $E_3$ satisfies
\begin{equation*}
\|E_3\|_{BMO}\lesssim_{\uAS}\underline{A}^2_{\frac{1}{4}}, \qquad \|E_3\|_{\dot{W}^{\frac{1}{4},4}}\lesssim_\uAS \Astar \uAStar.
\end{equation*}
\item  Paradifferential identities relating $\W, X_\alpha$, and $Y$:
\begin{equation}
    X_\alpha = T_{1-Y}\W+E_4  = T_{1+\W}Y +E_4, \qquad Y = T_{(1-Y)^2}\W +E_4, \label{XAlphaY}
\end{equation}
where the error $E_4$ satisfies
\begin{equation*}
    \|E_4\|_{\dot{W}^{\frac{1}{2},4}}\lesssim_{A^\sharp} \Astar, \qquad \||D|^{\frac{1}{2}}E_4\|_{BMO}\lesssim_A \underline{A}^2_{\frac{1}{4}}.
\end{equation*}
\end{enumerate} \label{t:XParaMaterial}
\end{lemma}
\begin{proof}
 We apply the para-Leibniz rule Lemma \ref{t:Leibniz} to $X$,
 \begin{equation*}
     T_{D_t}X = -T_{D_t Y}W + T_{1-Y}T_{D_t}W +E_1.
 \end{equation*}
For the first term on the right-hand side, we use \eqref{EqnYR} to write
\begin{equation*}
   T_{D_t Y}W = -T_{|1-Y|^2R_\alpha }W + T_{(1-Y)\underline{M}+ i\frac{\gamma}{2}[Y^2+ \frac{\bar{Y}}{1-\bar{Y}}Y(1-Y)]}W.
\end{equation*}
Here for the  term $-T_{|1-Y|^2R_\alpha }W$, we use \eqref{ParaProducts} to separate $X =T_{1-Y}W$, and then \eqref{BMOBalance}, \eqref{BMOSigma} to  peel off perturbative components where  $R_\alpha$ have lower or comparative frequencies than $Y$.
The second term belongs to $E_1$ after applying the bound \eqref{uMBound} directly.

For the second term $T_{1-Y}T_{D_t}W$, we rewrite using \eqref{ParaW},
\begin{equation*}
 T_{1-Y}T_{D_t}W = T_{1-Y}\left(-T_{1+W_\alpha} \nP[(1-\bar{Y})R]+ i\frac{\gamma}{2}T_{1+W_\alpha} \nP[W(1-\bar{Y})]-\nP\Pi(W_\alpha, \ub)-i\frac{\gamma}{2}W\right).
\end{equation*}
After applying \eqref{ParaProducts}, \eqref{ParaProductsTwo}, \eqref{ParaAssociateOne}, \eqref{ParaAssociateThree} on the right-hand side to distribute and simplify the para-coefficient $T_{1-Y}$, we obtain the para-material derivative of $X$.

For the leading term of $D_t X$, we notice that $T_{T_{1-\bar{Y}}R_\alpha}X$ and $\nP\Pi(X_\alpha, \ub)$ have a good balance of derivatives and may be absorbed into $E_2$.
The error $E_1$ is part of $E_2$ due to the embedding $\dot{W}^{\frac{5}{4},4}(\mathbb{R})\hookrightarrow BMO^1(\mathbb{R})$.
For the other terms, we peel off balanced components to write
\begin{equation*}
    -\nP[(1-\bar{Y})R] +i\frac{\gamma}{2}\nP[(1-\bar{Y})W]-i\frac{\gamma}{2}X = -T_{1-\bar{Y}}R + i\frac{\gamma}{2}T_{1-\bar{Y}}W  +E_2.
\end{equation*}
The remaining two vorticity terms can be combined togther
\begin{equation*}
    i\frac{\gamma}{2}T_{1-\bar{Y}}W - i\frac{\gamma}{2}X = i\frac{\gamma}{2}T_{Y-\bar{Y}}W = -\gamma T_{\Im Y}W,
\end{equation*}
which again may be absorbed into $E_2$.

$(3)$ is a direct consequence of $(1)$ and $(2)$, and $(4)$ is proved in Lemma $3.2$ of \cite{ai2023dimensional}.
\end{proof}

\begin{lemma}
We have the following results for para-material derivatives of $Z$ and $U$:
\begin{enumerate}
\item Leading term of the para-material derivative of $Z$:
\begin{equation*}
    T_{D_t}Z - igX +i\gamma Z = E_1,
\end{equation*}
where the error $E_1$ satisfies
\begin{equation*}
   \gamma^2\||D|^{\frac{1}{2}}E_1\|_{BMO}+ \gamma^3\|E_1\|_{BMO}\lesssim_{\uAS} \underline{A}^2_{\frac{1}{4}}.
\end{equation*}
\item Leading term of the para-material derivative of $U$:
\begin{equation*}
    T_{D_t}U +T_{1-\bar{Y}}Z= E_2,
\end{equation*}
where the error $E_2$ satisfies
\begin{equation*}
   \gamma^2\||D|E_2\|_{BMO}\lesssim_{\uAS} \underline{A}^2_{\frac{1}{4}}.
\end{equation*}
\item We have the relations:
\begin{align}
&Z_\alpha = R + E_3, \label{ZAlphaR}\\
&U_\alpha = X + E_4, \label{UAlphaX}
\end{align}
where $E_3$ and $E_4$ satisfy
\begin{equation*}
\gamma^3\|E_3\|_{BMO}\lesssim_A \underline{A}^2_{\frac{1}{4}}, \quad \gamma\||D|E_4\|_{BMO}+\gamma^3\|E_4\|_{BMO}\lesssim_A \underline{A}^2_{\frac{1}{4}}.
\end{equation*}
\item Leading terms of the para-material derivatives of $Z_\alpha$ and $U_\alpha$:
\begin{align*}
&T_{D_t}Z_\alpha -igT_{1-Y}W_\alpha + i\gamma R = E_5,\\
&T_{D_t}U_\alpha + T_{1-\bar{Y}}R = E_6,
\end{align*}
where the errors $E_5$ and $E_6$ satisfy
\begin{align*}
&\gamma\|E_5\|_{BMO}+\||D|^{\frac{1}{2}}E_5\|_{BMO}\lesssim_\uAS \uA^2_{\frac{1}{4}}, \\ &\gamma^2\|E_6\|_{BMO}+\||D|E_6\|_{BMO}\lesssim_\uAS \uA^2_{\frac{1}{4}}.
\end{align*}
\end{enumerate}
\end{lemma}
\begin{proof}
\begin{enumerate}
\item We apply the para-Leibniz rule Lemma \ref{t:Leibniz} to $Z$,
 \begin{equation*}
     T_{D_t}Z = -T_{D_t Y}Q + T_{1-Y}T_{D_t}Q +E_1,
 \end{equation*}
 where for $E_1$, we have
 \begin{equation*}
     \gamma^2\|E_1\|_{\dot{W}^{\frac{3}{4},4}}+ \gamma^3\|E_1\|_{\dot{W}^{\frac{1}{4},4}}\lesssim_{A^\sharp}\underline{A}^2_{\frac{1}{4}},
 \end{equation*}
so that it satisfies the above error estimate due to the embedding $\dot{W}^{\frac{1}{4},4}(\mathbb{R})\hookrightarrow BMO(\mathbb{R})$.
The term $T_{D_t Y}Q$ is perturbative, similar to the estimate in Lemma \ref{t:XParaMaterial}.
For $T_{D_t}Q$, we use the second equation of \eqref{e:CVWW1} and apply the porjection $\nP$ to write
\begin{equation*}
    T_{D_t}Q = igW -i\gamma Q + i\frac{\gamma}{2}\nP[\bar{R}W]- T_{Q_\alpha}\nP\ub -\nP\Pi(Q_\alpha, \ub) .
\end{equation*}
Applying $T_{1-Y}$ to $T_{D_t}Q$, using  \eqref{ParaProductsTwo},  \eqref{ParaAssociateThree} and the fact that $R = (1-Y)Q_\alpha$, we have
\begin{equation*}
T_{1-Y}T_{D_t}Q = igX-i\gamma Z  + i\frac{\gamma}{2}\nP[\bar{R}X]- T_{R}\nP\ub -\nP\Pi(R, \ub)+E_1.
\end{equation*}
The last three terms of the right-hand side may be absorbed into $E$ using \eqref{BMOBalance} and \eqref{BMOSigma}.
\item  We apply the para-Leibniz rule Lemma \ref{t:Leibniz} to $U$,
 \begin{equation*}
     T_{D_t}U = -T_{D_t Y}\partial^{-1}_\alpha W + T_{1-Y}T_{D_t}\partial^{-1}_\alpha W +E_2.
 \end{equation*}
The first term on the right $T_{D_t Y}\partial^{-1}_\alpha W$ is perturbative, similar to the estimate in Lemma \ref{t:XParaMaterial}.
Using the computation in Lemma \ref{t:XParaMaterial}, one can write
\begin{equation*}
    T_{D_t}W + T_{1-\bar{Y}}Q_\alpha = E_7,
\end{equation*}
where  the error $E_7$ satisfies
\begin{equation*}
    \gamma^2 \|E_7\|_{BMO}\lesssim_{\uAS} \underline{A}^2_{\frac{1}{4}}.
\end{equation*}
Applying the anti-derivative $\partial^{-1}_\alpha$ to $T_{D_t}W$, one can rewrite using the commutator
\begin{equation*}
 T_{D_t} \partial^{-1}_\alpha W  + T_{1-\bar{Y}}Q = [T_{D_t}, \partial^{-1}_\alpha]W - [T_{\bar{Y}},\partial^{-1}_\alpha]Q_\alpha + E_2.
\end{equation*}
For the two commutators, they are
\begin{align*}
  [T_{D_t}, \partial^{-1}_\alpha]W &= T_{\ub} W -\partial^{-1}_\alpha(T_{\ub} W_\alpha) = \partial^{-1}_\alpha(T_{\ub_\alpha} W), \\
  [T_{\bar{Y}},\partial^{-1}_\alpha]Q_\alpha &= T_{\bar{Y}}Q - \partial^{-1}_\alpha(T_{\bar{Y}} Q_\alpha) = \partial^{-1}_\alpha(T_{\bar{Y}_\alpha}Q),
\end{align*}
so that they may be absorbed into the error $E_2$ using \eqref{BMOSigma}.
Applying $T_{1-Y}$ to $T_{D_t}\partial^{-1}_\alpha W$, we obtain the leading term of $T_{D_t}U$.
\item Taking the derivative of $Z$, we have
\begin{equation*}
    Z_\alpha = \partial_\alpha T_{1-Y}Q = T_{1-Y}Q_\alpha - T_{Y_\alpha}Q = R + T_{Q_\alpha}Y +\Pi(Y,Q_\alpha) -T_{Y_\alpha} Q.
\end{equation*}
Then the last three terms on the right-hand side go to the error $E_3$.

Again applying derivative to $U$, we have
\begin{equation*}
    U_\alpha = \partial_\alpha T_{1-Y}\partial_\alpha^{-1}W = T_{1-Y}W - T_{Y_\alpha}\partial_\alpha^{-1}W = X+ E_4.
\end{equation*}
\item Taking the derivative of $T_{D_t}Z$,
\begin{equation*}
    T_{D_t}Z_\alpha +T_{\ub_\alpha}X_\alpha -igX_\alpha + i\gamma Z_\alpha =E_5.
\end{equation*}
The term $T_{\ub_\alpha}X_\alpha$ can be absorbed into $E_5$.
We change $X_\alpha$ to $T_{1-Y}\W$ by \eqref{XAlphaY}, and change $U_\alpha$ to $X$ by \eqref{UAlphaX}.

Similarly, taking the derivative of $T_{D_t}U$, we obtain
\begin{equation*}
    T_{D_t}U_\alpha +T_{\ub_\alpha}U_\alpha -T_{\bar{Y}_\alpha}Z + T_{1-\bar{Y}}Z_\alpha = E_6.
\end{equation*}
The terms $T_{\ub_\alpha}U_\alpha$ and $T_{\bar{Y}_\alpha}Z$ can be moved into $E_6$ and we replace $Z_\alpha$ by $R$ using \eqref{ZAlphaR}.
\end{enumerate}
\end{proof}

Lastly, we compute the material and para-material derivatives of the frequency-shift $\ua$.
\begin{lemma}
The leading  terms of the material and para-material derivatives of $\ua$ are given by
\begin{align}
    D_t \ua &= -(g+\ua)M+ \gamma^2\Im R - g\gamma \Im \W+E, \label{MaterialA}\\
    T_{D_t} \ua & =-T_{g+\ua}M+ \gamma^2\Im R -g\gamma \Im \W+E, \label{ParaMaterialA}
\end{align}
where the error term $E$ satisfies the estimate
\begin{equation}
    \|E\|_{BMO}\lesssim_\uA \uA^2_{\frac{1}{4}}.\label{ErrorE}
\end{equation}
\end{lemma}
\begin{proof}
First, the leading terms of material derivative \eqref{MaterialA} is a direct consequence of the leading terms of the para-material derivative \eqref{ParaMaterialA} due to the estimate \eqref{uaBMO}, \eqref{ubBMO} and \eqref{MBMOOne}:
\begin{equation*}
  \|\Pi(\ub, \ua_\alpha)\|_{BMO} + \| T_{\ua_\alpha} \ub\|_{BMO} \lesssim_\uA \uA^2_{\frac{1}{4}},\quad
   \|\Pi(M, \ua)\|_{BMO} + \| T_{M} \ua\|_{BMO} \lesssim_\uA \uA^2_{\frac{1}{4}}.
\end{equation*}
Therefore, it only suffices to compute the leading term of the para-material derivative of $\ua$ \eqref{ParaMaterialA}.

Next, we compute the para-material derivative of $a$.
Using the para-associativity \eqref{ParaAssociateTwo}, the para-Leibniz estimates \eqref{UnbalancedParaLeibnizOne} \eqref{BalancedLeibniz} and also \eqref{ParaWalphaR}, we compute
\begin{align*}
    T_{D_t}\nP(\bar{R}_\alpha R) &= \nP[\partial_\alpha T_{D_t}\bar{R}\cdot R] + \nP[\bar{R}_\alpha T_{D_t}R] + E\\
    & = \nP[\partial_\alpha (-i T_{g+\ua}\bar{Y}+i\gamma \bar{R})\cdot R]+ \nP[\bar{R}_\alpha \cdot(iT_{g+\ua}Y -i\gamma R)]+E,\\
    & = iT_{g+\ua}\nP[\bar{R}_\alpha Y -\bar{Y}_\alpha R] +E
\end{align*}
with error $E$ satisfying \eqref{ErrorE}.
Recall the definition of $M$,
\begin{equation*}
    M = \nP[R\bar{Y}_\alpha - \bar{R}_\alpha Y] + \bar{\nP}[\bar{R}Y_\alpha -R_\alpha\bar{Y}],
\end{equation*}
we have that
\begin{align*}
    T_{D_t} a &= -i(T_{D_t}\nP(\bar{R}_\alpha R) - T_{D_t}\bar{\nP}(R_\alpha \bar{R})) \\
    & =  T_{g+\ua}(\nP[\bar{R}_\alpha Y -\bar{Y}_\alpha R]+ \bar{\nP}[R_\alpha\bar{Y}-\bar{R}Y_\alpha ])+E \\
    & = T_{g+\ua} M+E,
\end{align*}
with error $E$ satisfying \eqref{ErrorE}.

We continue to compute the para-material derivatives of the vorticity terms,
\begin{equation*}
\frac{\gamma}{2}T_{D_t}(R+\bar{R}) = \frac{\gamma}{2}(iT_{g+\ua}Y-iT_{g+\ua}\bar{Y}-i\gamma R+i\gamma\bar{R})+E = \gamma^2 \Im R - g\gamma \Im \W +E.
\end{equation*}

Finally, we show that $\frac{\gamma}{2}T_{D_t}N$ can be put into the error $E$.
We compute
\begin{align*}
    &T_{D_t} \nP[W\bar{R}_\alpha -\bar{\W} R] = \nP[ T_{D_t}W\bar{R}_\alpha + WT_{D_t}\bar{R}_\alpha -T_{D_t}\bar{\W} R-\bar{\W}T_{D_t} R ]\\
    =& \nP[-T_{1+\W}T_{1-\bar{Y}}R\cdot \bar{R}_\alpha -iT_{g+\ua}\bar{Y}_\alpha \cdot W +i\gamma \bar{R}_\alpha W+ T_{1+\bar{\W}}T_{1-Y}\bar{R}_\alpha\cdot R -i\bar{\W}T_{g+\ua}Y+i\gamma R\bar{\W}].
\end{align*}
Each term on the right-hand side can be placed into the error term $E$.
For instance, we can rewrite
\begin{equation*}
    \nP[\bar{R}_\alpha W +R\bar{\W}] = \nP T_{\bar{R}_\alpha} W+ \nP\Pi(\bar{R}_\alpha, W) + \nP T_{\bar{\W}}R +\nP\Pi(R,\bar{\W}),
\end{equation*}
so that by using \eqref{BMOBalance} and \eqref{BMOSigma},
\begin{equation*}
   \gamma^2\|\nP[\bar{R}_\alpha W +R\bar{\W}] \|_{BMO}\lesssim_\uA \gamma^2A_{-\frac{3}{4}}A_{\frac{1}{4}}\lesssim \uA^2_{\frac{1}{4}}.
\end{equation*}
Combining the estimates at each step, we obtain the para-material derivative of $\ua$ \eqref{ParaMaterialA}.
\end{proof}

\section{Estimates for the linearized equations} \label{s:LinearEstimate}
In this section, we derive the balanced energy estimates for the linearized system.
Let the solutions for the linearized
water waves  around a solution $(W, Q)$ to the system \eqref{e:CVWW} by $(w, q)$ and $r: = q-Rw$.
Then it is computed in Section $3$ of \cite{MR3869381} that the linearized variables $(w,r)$ solve the system
\begin{equation}
\left\{
             \begin{array}{lr}
             (\partial_t + \mathfrak{M}_{\underline{b}}\partial_\alpha)w +\mathbf{P}\left[\dfrac{1}{1+\Bar{\mathbf{W}}}r_\alpha\right]+\mathbf{P}\left[\dfrac{R_\alpha}{1+\Bar{\mathbf{W}}}w\right] +\gamma\mathbf{P}\left[\dfrac{\Im \W}{1+\Bar{\mathbf{W}}}w\right] = \mathbf{P}\underline{\mathcal{G}_0}(w,r)&\\
            (\partial_t + \mathfrak{M}_{\underline{b}}\partial_\alpha)r +i\gamma r-i\mathbf{P}\left[\dfrac{g+\underline{a}}{1+\mathbf{W}}w\right]=\mathbf{P}\underline{\mathcal{K}_0}(w,r),&  \label{linearizedeqn}
             \end{array}
\right.
\end{equation}
where $\mathfrak{M}_\ub f = \nP[\ub f]$,
and the source terms $\underline{\mathcal{G}}_0(w,r),  \underline{\mathcal{K}}_0(w,r)$ are given by
\begin{equation}  \label{GKZeroDef}
\begin{aligned}
&\underline{\mathcal{G}}_0(w,r) = \mathcal{G}(w,r)-i\frac{\gamma}{2}\mathcal{G}_1(w,r), \quad \underline{\mathcal{K}}_0(w,r) = \mathcal{K}(w,r)-i\frac{\gamma}{2}\mathcal{K}_1(w,r), \\
&\mathcal{G}(w,r) = (1+\mathbf{W})(\mathbf{P}\Bar{m}+\Bar{\mathbf{P}}m), \quad \mathcal{G}_1(w,r) = -(1+\mathbf{W})(\mathbf{P}\Bar{m}_1-\Bar{\mathbf{P}}m_1), \\
&\mathcal{K}(w,r) = \Bar{\mathbf{P}}n - \mathbf{P}\Bar{n}, \quad \mathcal{K}_1(w,r) = \Bar{\mathbf{P}}m_2 + \mathbf{P}\Bar{m}_2, \quad n : = \frac{\Bar{R}(r_\alpha +R_\alpha w)}{1+\mathbf{W}}, 
\end{aligned}    
\end{equation}
as well as
\begin{equation*}
    \begin{aligned}
    &m : = \frac{q_\alpha - Rw_\alpha}{J}+\frac{\Bar{R}w_\alpha}{(1+\mathbf{W})^2} = \frac{r_\alpha + R_\alpha w}{J}+\frac{\Bar{R}w_\alpha}{(1+\mathbf{W})^2},\\
    &m_1: = \frac{1}{1+\Bar{\mathbf{W}}}w - \frac{\Bar{W}}{(1+\W)^2}w_\alpha, \quad m_2 := \Bar{R}w - \frac{\Bar{W}r_\alpha +\Bar{W}R_\alpha w}{1+\mathbf{W}}.
    \end{aligned}
\end{equation*}

We  define the associated  linear energy
 \begin{equation*}
    E^{(2)}_{lin}(w,r)=\int (g+\underline a)|w|^2+ \Im(r\bar{r}_\alpha)\, d\alpha.
\end{equation*}
Then it is shown in \cite{MR3869381} the following quadratic energy estimate for large data:
\begin{proposition}[\hspace{1sp}\cite{MR3869381}]
 The linearized system \eqref{linearizedeqn} is locally well-posed in $L^2\times \dot{H}^{\frac{1}{2}}$, and the following properties hold:
 \begin{enumerate}[label=(\roman*)]
     \item Norm equivalence:
       \begin{equation*}
           E^{(2)}_{lin}(w,r) \approx_{\underline{A}} \mathcal{E}_0(w,r).
       \end{equation*}
     \item Energy estimate:
      \begin{equation*}
     \frac{d}{dt}E^{(2)}_{lin}(w,r)\lesssim_A (\uB + \gamma \uA)E^{(2)}_{lin}(w,r).
 \end{equation*}
  \end{enumerate} \label{t:QuadraticEnergy}
\end{proposition}
In this section, we prove the following  energy estimate for small data of \eqref{linearizedeqn}.

\begin{theorem}
Assume that $\max\{\uA, \uAS \} \lesssim 1$.
Then the linearized system \eqref{linearizedeqn} is locally well-posed in $\doth{\frac{1}{4}}$.
Moreover, there exists an energy functional $E^{\frac{1}{4}}_{lin}(w,r)$ with the following properties:
\begin{enumerate}
    \item Norm equivalence:
    \begin{equation*}
        E^{\frac{1}{4}}_{lin}(w,r) \approx_{\uAS} \| (w,r)\|^2_{\doth{\frac{1}{4}}} + O(\gamma^4 \uAS^2) \| (w,r)\|^2_{\doth{-\frac{3}{4}}}.
    \end{equation*}
    \item Energy estimate:
    \begin{equation*}
        \frac{d}{dt} E^{\frac{1}{4}}_{lin}(w,r) \lesssim_{\uAS} \Astar(\gamma^{\frac{1}{2}}+\uAStar)E^{\frac{1}{4}}_{lin}(w,r).
    \end{equation*}
\end{enumerate} \label{t:LinearizedWellposed}
\end{theorem}
Compared to the previous result Proposition \ref{t:QuadraticEnergy}, in our theorem, the coefficient of the energy estimate does not depend on the pointwise control norm $\uB$.
It merely depends on $\Astar$, which can be seen as an intermediate control norm between $\uA$ and $\uB$, and also the $L^4$ based control norm $\uAStar$.
$\doth{\frac{1}{4}}$ is the minimal Sobolev regularity one can expect for $(w, r)$ to have the above balanced energy estimate.

From the linearized equations \eqref{linearizedeqn} we obtain the corresponding paradifferential flow
\begin{equation}
\left\{
             \begin{array}{lr}
             T_{D_t}w+T_{1-\bar{Y}}\partial_\alpha r + T_{(1-\bar{Y})R_\alpha }w +\gamma T_{\Im \W}w = 0 &\\
           T_{D_t}r +i\gamma r -iT_{1-Y}T_{g+\ua}w=0.&  \label{ParadifferentialFlow}
             \end{array}
\right.
\end{equation}
In the following, we will fix a self-adjoint quantization for $T$.
To achieve this, we may use the Weyl quantization, or simply the average $\frac{1}{2}(T+T^{*})$.
Using the self-adjoint quantization, later for computations of the integrals such as \eqref{TimeDerivativeE0Para}, for any real-valued function $f$, one can distribute the para-coefficient $T_f$ so that
\begin{equation*}
\int T_f g\cdot h \,d\alpha = \int  g\cdot T_f h \,d\alpha, \quad \forall f,g\in L^2.
\end{equation*}
This will make our computation easier by avoiding the estimates for $(T_{f})^{*}$.

The linearized equations \eqref{linearizedeqn} can be rewritten in the paradifferential form
\begin{equation}
\left\{
             \begin{array}{lr}
             T_{D_t}w+T_{1-\bar{Y}}\partial_\alpha r + T_{(1-\bar{Y})R_\alpha }w +\gamma T_{\Im \W}w = \mathcal{G}^{\sharp}(w,r) &\\
           T_{D_t}r +i\gamma r -iT_{1-Y}T_{g+\ua}w=\mathcal{K}^{\sharp}(w,r),&  \label{ParadifferentialLinearEqn}
             \end{array}
\right.
\end{equation}
where the source terms $(\mathcal{G}^{\sharp}, \mathcal{K}^{\sharp})$ are given by
\begin{equation*}
   \mathcal{G}^{\sharp} = \nP(\underline{\mathcal{G}}_0+\underline{\mathcal{G}}_1), \qquad   \mathcal{K}^{\sharp} = \nP(\underline{\mathcal{K}}_0+\underline{\mathcal{K}}_1),
\end{equation*}
with $(\underline{\mathcal{G}}_0, \underline{\mathcal{K}}_0)$ are as \eqref{GKZeroDef} and
\begin{align*}
  \underline{\mathcal{G}}_1 = &(T_{r_\alpha}\bar{Y}+ \Pi(r_\alpha, \bar{Y})) -(T_{w_\alpha}\ub + \Pi(w_\alpha, \ub)) -(T_w ((1-\bar{Y})R_\alpha+ \gamma \Im \W)\\
  & + \Pi(w, (1-\bar{Y})R_\alpha + \gamma \Im\W)),\\
  \underline{\mathcal{K}}_1 = & -(T_{r_\alpha}\ub + \Pi(r_\alpha, \ub))+ i(T_{1-Y}T_w\ua+T_{1-Y}\Pi(w,\ua)-T_{(g+\ua)w}Y-\Pi((g+\ua)w, Y))
\end{align*}
are the paradifferential truncations.

The proof of Theorem \ref{t:LinearizedWellposed} is divided into the following steps.
First, we consider a variant of \eqref{ParadifferentialLinearEqn} with more general right-hand side $(G,K)$,
\begin{equation}
\left\{
             \begin{array}{lr}
             T_{D_t}w+T_{1-\bar{Y}}\partial_\alpha r + T_{(1-\bar{Y})R_\alpha }w +\gamma T_{\Im \W}w = G &\\
           T_{D_t}r +i\gamma r -iT_{1-Y}T_{g+\ua}w=K.&  \label{SourceParadifferential}
             \end{array}
\right.
\end{equation}
Under this setting, one can prove the following result.
\begin{proposition}
Assume that $\max\{\uA, \uAS\}$ is small,  then the homogeneous paradifferential system \eqref{ParadifferentialFlow} is locally well-posed in $\doth{s}$ for any $s\in \mathbb{R}$.
Furthermore, for each $s$, there exists an  energy functional $E^{s,para}_{lin}(w, r)$ such that we have
\begin{enumerate}
\item The norm equivalence:
\begin{equation*}
    E^{s,para}_{lin}(w, r) \approx_{A^\sharp} \|(w,r)\|_{\doth{s}}^2.
\end{equation*}
\item The time derivative of $E^{s,para}_{lin}(w, r)$ is bounded by
\begin{equation*}
    \frac{d}{dt}  E^{s,para}_{lin}(w, r) \lesssim \Astar(\gamma^{\frac{1}{2}}+\uAStar)  \|(w,r)\|_{\doth{s}}^2.
\end{equation*}
\end{enumerate}  \label{t:wellposedflow}
\end{proposition}
We will first prove the easier case for $s=0$, and then consider the case for more general $s$.
Clearly, Theorem \ref{t:LinearizedWellposed} will follow directly from Proposition \ref{t:wellposedflow} as long as the source terms $(\mathcal{G}^{\sharp}, \mathcal{K}^{\sharp})$ satisfy for $s = \frac{1}{4}$,
\begin{equation}
    \|(\mathcal{G}^{\sharp}, \mathcal{K}^{\sharp})\|_{\doth{s}}\lesssim \Astar(\gamma^{\frac{1}{2}}+\uAStar) \| (w,r)\|_{\doth{s}}. \label{GoodSourceTermBound}
\end{equation}

Next, we take into account the nonlinear source terms $(\mathcal{G}^{\sharp}, \mathcal{K}^{\sharp})$ on the right-hand side.
Unfortunately, the source terms $(\mathcal{G}^{\sharp}, \mathcal{K}^{\sharp})$ do not satisfy the bound \eqref{GoodSourceTermBound} for any $s$, because of both the quadratic contributions and  unbalanced cubic contributions.

In order to deal with these unfavourable source terms, we will use the paradifferential normal form analysis to construct the modified normal form linear variables $(w_{NF}, r_{NF})$.
We work at a specific regularity level, namely $s=\frac{1}{4}$, as this is the minimal Sobolev index that allows us to obtain the perturbative bounds in \eqref{GoodSourceTermBound}.
 With these new modified normal form linear variables $(w_{NF}, r_{NF})$, the source terms become perturbative and satisfy the bound of \eqref{GoodSourceTermBound} type.
\begin{proposition}
  Assuming that $(w,r)$ solve the linearized paradifferential system \eqref{ParadifferentialLinearEqn}, then there exist modified normal form linear variables $(w_{NF}, r_{NF})$ satisfying \eqref{SourceParadifferential} and that we have
\begin{enumerate}
\item Invertibility:
\begin{equation*}
    \| (w_{NF}, r_{NF})-(w,r)\|_{\doth{\frac{1}{4}}}\lesssim_{\uA} \uAS\left(\|(w,r) \|_{\doth{\frac{1}{4}}}+ \gamma^2 \| (w,r)\|_{\doth{-\frac{3}{4}}}\right).
\end{equation*}
\item  Perturbative source terms:
\begin{equation*}
 \| (G,K)\|_{\doth{\frac{1}{4}}}\lesssim_{\uAS}  \Astar \uAStar \left(\| (w,r)\|_{\doth{\frac{1}{4}}}+ \gamma^2 \| (w,r)\|_{\doth{-\frac{3}{4}}}\right).
\end{equation*}
\end{enumerate}
\end{proposition}
\begin{remark}
It is possible to improve the above source term bound, replacing $\Astar \uAStar$ by a slightly smaller constant $\uA^2_{\frac{1}{4}}$ as in \cite{ai2023dimensional}.
However, this not only requires more delicate estimates for the para-material derivatives and more steps of paradifferential normal form corrections, but also does not help in the proof of the local well-posedness.
We will not prove this sharper version of estimates here.
\end{remark}
The rest of this section is devoted to the proof of the above results.
In Section \ref{s:H0Bound}, we compute the time derivative for the linear paradifferential energy $E^{0,para}_{lin}$ when $s=0$.
Then in Section \ref{s:HsBound}, we consider for general $s\in \mathbb{R}$.
$E^{0,para}_{lin}(|D|^s w, |D|^s r)$ does not satisfy our need for paradifferential energy because it brings additional nonperturbative source terms.
In order to eliminate these bad terms, we use the paradifferential conjugation to construct new variables $(\tilde{w}^s, \tilde{r}^s)$.
This change of variables reduces the source terms of paradifferential equations to balanced ones, thus proving Proposition \ref{t:wellposedflow}.
In the rest of Section \ref{s:LinearEstimate} we take into account the effect of source terms $(\mathcal{G}^{\sharp}, \mathcal{K}^{\sharp})$.
In Section \ref{s:SourceBound}, we compute the $\mathcal{H}$ bound of $(\underline{\mathcal{G}}_0, \underline{\mathcal{K}}_0)$ and its leading parts.
Then we compute the material and para-material derivatives of $(w,r)$, $x: = T_{1-Y}w$, and $u: = T_{1-Y}\partial^{-1}_\alpha w$.
Next, in Section \ref{s:GKOne}, we construct the normal form corrections that remove $(\underline{\mathcal{G}}_1, \underline{\mathcal{K}}_1)$ up to balanced cubic terms.
Finally in Section \ref{s:GKZero}, we construct the normal form corrections that remove the leading term of $(\underline{\mathcal{G}}_0, \underline{\mathcal{K}}_0)$.
After these paradifferential normal form transformations, the system is finally reduced to the desired form, and this finishes the proof of Theorem \ref{t:LinearizedWellposed}.

\subsection{$\mathcal{H}^0$ bound for the paradifferential equation}
\label{s:H0Bound}
In this subsection, we prove Proposition \ref{t:wellposedflow} in the case $s=0$.
Following the setup in \cite{ai2023dimensional}, we consider the  paradifferential energy
\begin{equation*}
    E^{0,para}_{lin}(w,r) = \int_{\mathbb{R}} T_{g+\ua}w \cdot \bar{w} + \Im(r\bar{r}_\alpha)\,d\alpha.
\end{equation*}
Clearly, we have the norm equivalence since
\begin{equation*}
    \| T_\ua w\|_{L^2} \lesssim \|\ua\|_{L^\infty}\|w\|_{L^2} \lesssim_A O(\uA) \|w\|_{L^2}.
\end{equation*}
The assumption  $\uA \lesssim 1$ ensures that $ E^{0,para}_{lin}(w,r)$ is a positive energy.
As for the time derivative of the energy, we have the following computation.
\begin{proposition}
 Suppose that $(w,r)$ solve the  \eqref{SourceParadifferential}, and $(G,K)\in L^2\times \dot{H}^{\frac{1}{2}}$, then the time derivative of the paradifferential energy is:
\begin{equation}
  \frac{d}{dt}  E^{0,para}_{lin} = 2\Re\int T_{g+\ua}\bar{w}G -i\bar{r}_\alpha K \,d\alpha + \gamma \int  T_{\gamma\Im R -2g\Im \W}\bar{w}\cdot w \,d\alpha + O_{\uA} (\underline{A}^2_{\frac{1}{4}})E^{0,para}_{lin}. \label{TimeDerivativeE0Para}
\end{equation}
\end{proposition}
\begin{proof}
By direct computation, and the self-adjointness of $T$,
\begin{equation*}
\frac{d}{dt} E^{0,para}_{lin}(w,r) = 2\Re\int T_{g+\ua} \bar{w}\cdot w_t\, d\alpha + 2\Im \int \bar{r}_\alpha r_t \,d\alpha + \int T_{\ua_t} \bar{w}\cdot w \,d\alpha.
\end{equation*}
The strategy here is to replace the time derivatives $w_t$ and $r_t$ by para-material derivatives $T_{D_t} w$ and $T_{D_t} r$, so that we can use the system \eqref{SourceParadifferential}.
Using integration by parts, we write
\begin{align*}
 2\Re \int& T_{g+\ua}\bar{w}\cdot T_{\ub}w_\alpha \,d\alpha    = -\int T_{((g+\ua)\ub)_\alpha}\bar{w}\cdot w \,d\alpha \\
 &- \int (T_{\ub_\alpha}T_{\ua} + T_{\ub}T_{\ua_\alpha}-T_{(\ua\ub)_\alpha})\bar{w}\cdot w \,d\alpha -\int(T_\ub T_\ua -T_\ua T_\ub )\bar{w}_\alpha \cdot w \,d\alpha, \\
 2\Im \int& \bar{r}_\alpha\cdot T_\ub \partial_\alpha r \,d\alpha  =0.
\end{align*}
The  above commutator integrals satisfy
\begin{equation*}
  \int (T_{\ub_\alpha}T_{\ua} + T_{\ub}T_{\ua_\alpha}-T_{(\ua\ub)_\alpha})\bar{w}\cdot w \,d\alpha + \int(T_\ub T_\ua -T_\ua T_\ub )\bar{w}_\alpha  \cdot w \,d\alpha  = O_\uA(\underline{A}^2_{\frac{1}{4}}) \|w\|_{L^2}^2
\end{equation*}
due to the para-product estimate \eqref{ParaProducts}, the para-commutator estimate \eqref{ParaCommutator} and also the $BMO$ bounds \eqref{uaBMO} and \eqref{ubBMO}.

Adding the $T_\ub$ integral to the energy estimate, and using \eqref{SourceParadifferential}, we get that
\begin{align*}
    \frac{d}{dt}  E^{0,para}_{lin}(w,r)
     = & 2\Re\int T_{g+\ua} \bar{w}\cdot T_{D_t}w \,d\alpha + 2\Im \int \bar{r}_\alpha T_{D_t}r \,d\alpha \\
     &+ \int T_{\ua_t + \ub \ua_\alpha + (g+\ua)\ub_\alpha}\bar{w}\cdot w d\alpha + O_\uA(\underline{A}^2_{\frac{1}{4}}) \|w\|_{L^2}^2\\
    = & 2\Re\int T_{g+\ua} \bar{w}\cdot G -i\bar{r}_\alpha K -T_{g+\ua}\bar{w} \cdot T_{(1-\bar{Y})R_\alpha +\gamma \Im\W}w \,d\alpha \\
    & + \int T_{\ua_t + \ub \ua_\alpha + (g+\ua)\ub_\alpha}\bar{w}\cdot w d\alpha + O_\uA(\underline{A}^2_{\frac{1}{4}}) \|w\|_{L^2}^2 \\
      = & 2\Re\int T_{g+\ua} \bar{w}\cdot G -i\bar{r}_\alpha K \,d\alpha  + O_\uA(\underline{A}^2_{\frac{1}{4}})\|w\|_{L^2}^2\\
     &+ \int T_{D_t \ua + (g+\ua)(\ub_\alpha -2\Re \frac{R_\alpha}{1+\bar{\W}}-2\gamma \Im \W)}\bar{w}\cdot w \,d\alpha.
\end{align*}
Here, using \eqref{MaterialA} and the definition of $\underline{M}$, we have
\begin{align*}
  D_t \ua &+ (g+\ua)\left(\ub_\alpha -2\Re \frac{R_\alpha}{1+\bar{\W}}-2\gamma \Im \W\right) = D_t \ua -(g+\ua)(\underline{M}+\gamma \Im \W)\\
  &=-2(g+\ua)M+ \gamma^2\Im R - 2g\gamma \Im \W+i\frac{\gamma}{2}(g+\ua)M_1-\ua \gamma \Im \W+O_\uA(\underline{A}^2_{\frac{1}{4}})\\
  & = \gamma^2 \Im R -2g\gamma \Im \W + O_\uA(\underline{A}^2_{\frac{1}{4}}),
\end{align*}
where we use the $L^\infty$ bounds  \eqref{uaBMO}, \eqref{MBMOOne}, and \eqref{MoneLinfinity}.
Gathering all the terms, we obtain the estimate \eqref{TimeDerivativeE0Para}.
\end{proof}
The local well-posedness of the homogeneous paradifferential flow in $\mathcal{H}$ follows by a direct fixed point argument.

\subsection{$\doth{s}$ bounds for the linear paradifferential flow.}
\label{s:HsBound}
We now consider more general $s\in \mathbb{R}$ for Proposition \ref{t:wellposedflow}, and prove the $\doth{s}$ well-posedness of the linear paradifferential flow.
We will construct the following variables $(\tilde{w}^s, \tilde{r}^s)$:
\begin{proposition}
   Let $s\in \mathbb{R}$, and $(w,r)$ solve the linear paradifferential flow \eqref{ParadifferentialFlow}.
   Then there exist  linearized, normalized variables $(\tilde{w}^s, \tilde{r}^s)$ solving
   \begin{equation}
\left\{
             \begin{array}{lr}
             T_{D_t}\tilde{w}^s+T_{1-\bar{Y}}\partial_\alpha \tilde{r}^s + T_{(1-\bar{Y})R_\alpha+ \gamma \Im\W }\tilde{w}^s = G_s &\\
           T_{D_t}\tilde{r}^s +i\gamma \tilde{r}^s -iT_{1-Y}T_{g+\ua}\tilde{w}^s=K_s,&  \label{RSourceParadifferential}
             \end{array}
\right.
\end{equation}
and such that
\begin{equation*}
    \| (\tilde{w}^s, \tilde{r}^s)-|D|^s(w,r)\|_{\mathcal{H}^0} \lesssim_\uA \uA \| (w,r)\|_{\doth{s}},
\end{equation*}
\begin{equation}
    \| (G_s, K_s)\|_{\mathcal{H}^0} \lesssim_\uA \underline{A}^2_{\frac{1}{4}}\|(w,r)\|_{\doth{s}}. \label{GsHsSource}
\end{equation}
\end{proposition}
Then Proposition \ref{t:wellposedflow} follows by the fixed point argument and choosing
\begin{equation*}
    E^{s,para}_{lin}(w,r) = E^{0,para}_{lin}(\tilde{w}^s,\tilde{r}^s).
\end{equation*}
\begin{proof}
A first idea to this problem is to consider the variables
  \begin{equation*}
      (\tilde{w}^s, \tilde{r}^s) := (|D|^s w, |D|^s r).
  \end{equation*}
The new variables $(\tilde{w}^s, \tilde{r}^s)$ solve the system
\begin{equation}
\left\{
             \begin{array}{lr}
             T_{D_t}w^s+T_{1-\bar{Y}}\partial_\alpha r^s + T_{(1-\bar{Y})R_\alpha + \gamma \Im\W}w^s = \mathcal{G}^s_0 &\\
           T_{D_t}r^s +i\gamma r^s -iT_{1-Y}T_{g+\ua}w^s=\mathcal{K}^s_0,&
             \end{array}
\right. \label{TPhiwrsSystem}
\end{equation}
where the source terms $(\mathcal{G}^s_0, \mathcal{K}^s_0)$ are given by
\begin{align*}
    &\mathcal{G}^s_0 = L(\ub_\alpha, w^s)-L(\bar{Y}_\alpha, r^s) + L([(1-\bar{Y})R_\alpha]_\alpha + \gamma \Im\W_\alpha, \partial^{-1}_\alpha w^s),\\
    &\mathcal{K}^s_0 = L(\ub_\alpha, r^s)+ iL(Y_\alpha, \partial_\alpha^{-1}T_{g+\ua}w^s).
\end{align*}
Here $L$ denotes the order zero paradifferential commutator
\begin{equation*}
L(f_\alpha, u) = -[|D|^s, T_f]\partial_\alpha |D|^{-s}u \approx -sT_{f_\alpha} u + \text{lower order terms}.
\end{equation*}
Unfortunately, the source terms $(\mathcal{G}^s_0, \mathcal{K}^s_0)$ do not satisfy the bounds \eqref{GsHsSource}, so that they cannot be treated perturbatively.
Even after applying the normal form transformation, there are still some unbalanced quadratic and higher order terms left.

To eliminate these unbalanced terms, we first apply the paradifferential conjugation, which is similar to the  renormalization approach in \cite{MR3948114} and \cite{wan2023l2}, but is performed at the paradifferential level.
Precisely, we define the new variables
  \begin{equation*}
     (\tilde{w}_{1}^s, \tilde{r}_{1}^s) : = (T_{\Phi}|D|^s w, T_{\Phi}|D|^s r),
  \end{equation*}
  where  the conjugation function $\Phi$ satisfies the differential equation
  \begin{equation}
      D_t \Phi = s\Phi \ub_\alpha + O_\uA(\underline{A}^2_{\frac{1}{4}}). \label{PhaseEquation}
  \end{equation}
We choose the real-valued conjugation function $\Phi$ by
  \begin{equation*}
      \Phi :=  |1+\W|^{-2s}e^{-s\frac{\gamma}{g}\Re R -s \frac{\gamma^2}{g}\Im W}.
  \end{equation*}
Indeed, one can compute
\begin{align*}
D_t \Phi &= \partial_{\W}\Phi D_t \W+\partial_{\bar{\W}}\Phi D_t \bar{\W} \\
&= s\Phi\left(R_\alpha(1-\bar{Y})+\bar{R}_\alpha(1-Y)-2\underline{M}+\gamma \Im \W (Y+\bar{Y})  -\frac{\gamma}{g}D_t \Re R - \frac{\gamma^2}{g}D_t \Im W\right) \\
&= s\Phi\left(\ub_\alpha -\underline{M}+\gamma \Im \W (Y+\bar{Y}-1) + \gamma \Im \W\right) +  O_\uA(\underline{A}^2_{\frac{1}{4}})\\
&= s\Phi \ub_\alpha + O_\uA(\underline{A}^2_{\frac{1}{4}}),
\end{align*}
where we use  the fact that
\begin{equation*}
  \gamma D_t \Re R = -g\gamma \Im \W + \gamma^2 \Im R +  O_\uA(\underline{A}^2_{\frac{1}{4}}), \quad
  \gamma^2 D_t \Im W = -\gamma^2 \Im R +  O_\uA(\underline{A}^2_{\frac{1}{4}}),
\end{equation*}
and also \eqref{ubalpha}, \eqref{uMBound}.
We then apply the para-product rule \eqref{ParaProducts} for $\Phi \ub_\alpha$ to get
\begin{equation*}
    T_{D_t}\tilde{w}_{1}^s = T_\Phi T_{D_t}w^s + sT_{\ub_\alpha}\tilde{w}_{1}^s +G_s.
\end{equation*}
Similarly, we have
\begin{equation*}
    T_{D_t}\tilde{r}_{1}^s = T_\Phi T_{D_t}r^s + sT_{\ub_\alpha}\tilde{r}_{1}^s +K_s.
\end{equation*}
For the other terms on the left-hand side of \eqref{TPhiwrsSystem}, we use \eqref{ParaCommutator} and \eqref{ParaProducts} to write
\begin{align*}
    &T_{1-\bar{Y}}\partial_\alpha \tilde{r}^s_1 = T_\Phi T_{1-\bar{Y}}\partial_\alpha r^s -s(T_{1-\bar{Y}}T_{(1+\W)Y_\alpha +\frac{\gamma}{g}\Re R_\alpha + \frac{\gamma^2}{g}\Im \W}+T_{\bar{Y}_\alpha})\tilde{r}^s_1 +G_s, \\
    &T_{(1-\bar{Y})R_\alpha + \gamma \Im\W }\tilde{w}^s_1 = T_\Phi T_{(1-\bar{Y})R_\alpha + \gamma \Im\W}w^s_1 + G_s
\end{align*}
in the first equation, and
\begin{equation*}
    T_{1-Y}T_{g+\ua} \tilde{w}^s_1 =T_\Phi T_{1-Y}T_{g+\ua} w^s_1+K_s
\end{equation*}
in the second equation.
$T_\Phi$ commutes with the coefficients on the right side of \eqref{TPhiwrsSystem}  modulo acceptable errors.
Therefore, for $(\tilde{w}^s_1, \tilde{r}^s_1)$ we write
\begin{equation}
\left\{
             \begin{array}{lr}
             T_{D_t}\tilde{w}^s_1+T_{1-\bar{Y}}\partial_\alpha \tilde{r}^s_1 + T_{(1-\bar{Y})R_\alpha + \gamma \Im\W}\tilde{w}^s_1 = \mathcal{G}^s_1 +G_s &\\
           T_{D_t}\tilde{r}^s_1 +i\gamma \tilde{r}^s_1 -iT_{1-Y}T_{g+\ua}\tilde{w}^s_1 =\mathcal{K}^s_1+K_s,&
             \end{array}
\right. \label{TPhiTildewrsSystem}
\end{equation}
with the nonperturbative source terms
\begin{align*}
 \mathcal{G}^s_1 = &L(\ub_\alpha, \tilde{w}^s_1) +sT_{\ub_\alpha}\tilde{w}^s_1 - L(\bar{Y}_\alpha, \tilde{r}^s_1) -sT_{\bar{Y}_\alpha} \tilde{r}^s_1 \\
 &+ L([(1-\bar{Y})R_\alpha]_\alpha+ \gamma \Im\W_\alpha, \partial_\alpha^{-1}\tilde{w}^s_1)
 -sT_{1-\bar{Y}}T_{(1+\W)Y_\alpha+\frac{\gamma}{g}\Re R_\alpha + \frac{\gamma^2}{g}\Im \W}\tilde{r}^s_1,\\
 \mathcal{K}^s_1 = & L(\ub_\alpha, \tilde{r}^s_1) +sT_{\ub_\alpha}\tilde{r}^s_1 +iL(Y_\alpha, \partial_\alpha^{-1}T_{g+\ua}\tilde{w}^s_1).
\end{align*}
The first two leading terms in both $\mathcal{G}^s_1$ and $\mathcal{K}^s_1$ cancel.
More precisely, define the bilinear term
\begin{equation*}
    L_1(f_{\alpha\alpha}, \partial_\alpha^{-1}u) = L(f_\alpha, u) +sT_{f_\alpha} u\approx -s(s-1)T_{f_{\alpha\alpha}} \partial_\alpha^{-1}u + \text{lower order terms},
\end{equation*}
we can rewrite
\begin{align*}
 \mathcal{G}^s_1 = &L_1(\ub_{\alpha\alpha}, \partial^{-1}_\alpha\tilde{w}^s_1)  - L_1(\bar{Y}_{\alpha\alpha}, \partial_\alpha^{-1}\tilde{r}^s_1) + L([(1-\bar{Y})R_\alpha]_\alpha + \gamma \Im\W_\alpha, \partial_\alpha^{-1}\tilde{w}^s_1) \\
 &-sT_{1-\bar{Y}}T_{(1+\W)Y_\alpha+\frac{\gamma}{g}\Re R_\alpha + \frac{\gamma^2}{g}\Im \W}\tilde{r}^s_1,\\
 \mathcal{K}^s_1 = & L_1(\ub_{\alpha\alpha}, \partial_\alpha^{-1}\tilde{r}^s_1)  +iL(Y_\alpha, \partial_\alpha^{-1}T_{g+\ua}\tilde{w}^s_1).
\end{align*}
Here the source terms $(\mathcal{G}^s_1, \mathcal{K}^s_1)$ satisfy the estimate
\begin{equation*}
    \||D|^{-\frac{1}{4}}(\mathcal{G}^s_1, \mathcal{K}^s_1) \|_{\mathcal{H}^0}\lesssim_\uA \Astar \| (w,r)\|_{\doth{s}}.
\end{equation*}
We claim that
\begin{align*}
    &\|(\tilde{w}_{1}^s, \tilde{r}_{1}^s)\|_{\mathcal{H}^0}\approx_{\uA} \|(w,r)\|_{\doth{s}}, \\
     &\| (\tilde{w}_{1}^s, \tilde{r}_{1}^s)-|D|^s(w,r)\|_{\mathcal{H}^0} \lesssim_\uA \uA \| (w,r)\|_{\doth{s}}.
\end{align*}
In fact, it is shown in \cite{hunter2016two} and \cite{ai2023dimensional} that
\begin{equation*}
    \|(T_{|1+\W|^{-2s}}\tilde{w}^s, T_{|1+\W|^{-2s}}\tilde{r}^s)\|_{\mathcal{H}^0}\approx_{\uA} \|(w,r)\|_{\doth{s}},
\end{equation*}
and also the corresponding difference bound.
For the other exponential factor,
\begin{equation*}
   \Big|-s\frac{\gamma}{g}\Re R - s\frac{\gamma^2}{g}\Im W \Big| \leq \frac{s}{g} (\gamma |\Re R|+ \gamma^2 |\Im W|)\lesssim_{s,g} \uA.
\end{equation*}
As long as $\uA \ll 1$, the exponential factor is harmless, and does not contribute too much to the $\mathcal{H}^0$ norm.

Next, we proceed with additional normal form corrections.
The normal form corrections not only  replace bilinear source terms by trilinear terms, but also turn trilinear and higher unbalanced interactions into balanced ones.
Our normal form corrections will consist of the $L_1$ bilinear corrections which eliminate the $L_1$ terms, and also the secondary corrections which eliminate the other terms.

We begin with the corrections consisting of $L_1$ bilinear forms.
Heuristically, at the paradifferential level, we use \eqref{ubalpha} and the para-material derivatives \eqref{ParaWalphaR} to write
\begin{align*}
    \ub_{\alpha\alpha} &\approx 2\Re T_{1-\bar{Y}}R_{\alpha\alpha} + \gamma \Im T_{1-\bar{Y}}\W_\alpha \approx -2\Re T_{1-Y}T_{D_t}\W_\alpha - \gamma \Im T_{1-\bar{Y}}(iT_{D_t}R_\alpha -\gamma R_\alpha)\\
    & = -2\Re T_{1-Y}T_{D_t}\W_\alpha - \frac{\gamma}{g}\Re T_{1-\bar{Y}}T_{D_t}R_\alpha -\frac{\gamma^2}{g}\Im T_{1-Y}T_{D_t}\W.
\end{align*}
This motivates us to set
\begin{equation*}
\left\{
             \begin{array}{lr}
             \tilde{w}^s_2 = L_1(2\Re T_{1-Y}\W_\alpha + \frac{\gamma}{g}\Re T_{1-\bar{Y}}R_\alpha + \frac{\gamma^2}{g}\Im T_{1-Y}\W , \partial_\alpha^{-1}\tilde{w}^s_1) &\\
           \tilde{r}^s_2 = L_1(2\Re T_{1-Y}\W_\alpha + \frac{\gamma}{g}\Re T_{1-\bar{Y}}R_\alpha + \frac{\gamma^2}{g}\Im T_{1-Y}\W , \partial_\alpha^{-1}\tilde{r}^s_1).&
             \end{array}
\right.
\end{equation*}
We claim that this correction has the following effect
\begin{equation}
\left\{
             \begin{array}{lr}
        T_{D_t}(\tilde{w}^s_1+\tilde{w}^s_2)+T_{1-\bar{Y}}\partial_\alpha (\tilde{r}^s_1+\tilde{r}^s_2) + T_{(1-\bar{Y})R_\alpha + \gamma \W}(\tilde{w}^s_1+\tilde{w}^s_2) = \mathcal{G}^s_2  +G_s &\\
           T_{D_t}(\tilde{r}^s_1+\tilde{r}^s_2) +i\gamma (\tilde{r}^s_1+\tilde{r}^s_2) -iT_{1-Y}T_{g+\ua}(\tilde{w}^s_1+\tilde{w}^s_2) =\mathcal{K}^s_2 +K_s,&
             \end{array}
\right. \label{TTildewr2sSystem}
\end{equation}
where
\begin{align*}
    &\mathcal{G}^s_2 =  L\left([(1-\bar{Y})R_\alpha]_\alpha +  \gamma \Im\W_\alpha, \partial_\alpha^{-1}\tilde{w}^s_1\right)  + L\left((1+\W)Y_\alpha +\frac{\gamma}{g}\Re R_\alpha+\frac{\gamma^2}{g}\Im \W , T_{1-\bar{Y}}\tilde{r}^s_1\right),\\
    &\mathcal{K}^s_2 = iL(Y_\alpha, \partial_\alpha^{-1}T_{g+\ua}\tilde{w}^s_1).
\end{align*}
To prove this we begin with the first equation of the system.
When applying the para-material derivatives $T_{D_t}$, we use either \eqref{UnbalancedParaLeibnizOne} or \eqref{UnbalancedParaLeibnizTwo} to  distribute it to the paraproducts.
 The commutators with $\partial_\alpha$ involve $T_{\ub_\alpha}$, and we use \eqref{ubBMO} to place it into the error term $G_s$.
 We compute
 \begin{align*}
     T_{D_t}\tilde{w}^s_2 = &-  L_1\left(2\Re T_{T_{D_t}Y}\W_\alpha + \frac{\gamma}{g}\Re T_{T_{D_t}\bar{Y}}R_\alpha + \frac{\gamma^2}{g}\Im T_{T_{D_t}Y}\W , \partial_\alpha^{-1}\tilde{w}^s_1\right) \\
     &+ L_1\left(2\Re T_{1-Y}T_{D_t}\W_\alpha + \frac{\gamma}{g}\Re T_{1-\bar{Y}}T_{D_t}R_\alpha + \frac{\gamma^2}{g}\Im T_{1-Y}T_{D_t}\W , \partial_\alpha^{-1} \tilde{w}^s_1\right) \\
     &+ L_1\left(2\Re T_{1-Y}\W_\alpha + \frac{\gamma}{g}\Re T_{1-\bar{Y}}R_\alpha + \frac{\gamma^2}{g}\Im T_{1-Y}\W , \partial_\alpha^{-1} T_{D_t}\tilde{w}^s_1\right) + G_s.
 \end{align*}
 The first term can be added into $G_s$ in view of \eqref{YParaDerivative}.
 The second term becomes $L_1(\ub_{\alpha\alpha}, \partial_\alpha^{-1}\tilde{w}^s_1) + G_s$.
 For the third term, we use \eqref{TPhiTildewrsSystem} together with the source term bounds to handle $T_{D_t}\tilde{w}_1^s$.
These bounds allow us to estimate the corresponding $L_1$ contributions by taking advantage of the fact that $L_1$ has a paraproduct structure.
 As a consequence,
 \begin{equation*}
   T_{D_t}\tilde{w}^s_2 = -L_1(\ub_{\alpha\alpha}, \partial_\alpha^{-1}\tilde{w}^s_1) - L_1\left(2\Re T_{1-Y}\W_\alpha + \frac{\gamma}{g}\Re T_{1-\bar{Y}}R_\alpha - \frac{\gamma^2}{g}\Im T_{1-Y}\W , T_{1-\bar{Y}}\tilde{r}^s_1\right) + G_s.
 \end{equation*}
 For the remaining two terms on the left-hand side of the first equation,
 we repeatedly apply  \eqref{ParaCommutator} and \eqref{ParaProducts} so that the error terms can be absorbed into $G_s$.
 Also,  when derivatives fall on the coefficients, we have a good balance of derivatives.
 We get
 \begin{align*}
             T_{1-\bar{Y}}\partial_\alpha\tilde{r}^s_2 = & L_1\left(2\Re T_{1-Y}\W_\alpha + \frac{\gamma}{g}\Re T_{1-\bar{Y}}R_\alpha - \frac{\gamma^2}{g}\Im T_{1-Y}\W , T_{1-\bar{Y}}\tilde{r}^s_1\right)\\
             + & L_1(\bar{Y}_{\alpha\alpha}, \partial_\alpha^{-1}\tilde{w}^s_1)+L_1\left( \W_{\alpha\alpha} + \frac{\gamma}{g}R_{\alpha\alpha} - \frac{\gamma^2}{g}\W_\alpha,T_{1-\bar{Y}}\partial_\alpha^{-1}\tilde{r}^s_1\right)+G_s, \\
           T_{(1-\bar{Y})R_\alpha+ \gamma \Im\W}\tilde{w}^s_2 = & G_s.
\end{align*}
We repeat above computation for the second equation of the system.
After simplification,
\begin{align*}
   &T_{D_t}\tilde{r}^s_2 =   L_1\left(2\Re T_{1-Y}\W_\alpha + \frac{\gamma}{g}\Re T_{1-\bar{Y}}R_\alpha - \frac{\gamma^2}{g}\Im T_{1-Y}\W , \partial_\alpha^{-1}T_{D_t}\tilde{r}^s_1\right)-L_1(\ub_{\alpha\alpha}, \partial_\alpha^{-1}\tilde{r}^s_1) + K_s, \\
   &-iT_{1-Y}T_{g+\ua} \tilde{w}^s_2 = L_1\left(2\Re T_{1-Y}\W_\alpha + \frac{\gamma}{g}\Re T_{1-\bar{Y}}R_\alpha - \frac{\gamma^2}{g}\Im T_{1-Y}\W , -iT_{1-Y}T_{g+\ua}\partial_\alpha^{-1}\tilde{w}^s_1\right) + K_s, \\
   & i\gamma \tilde{r}^s_2 = L_1\left(2\Re T_{1-Y}\W_\alpha + \frac{\gamma}{g}\Re T_{1-\bar{Y}}R_\alpha - \frac{\gamma^2}{g}\Im T_{1-Y}\W , i\gamma\partial_\alpha^{-1}\tilde{r}^s_1\right) + K_s.
 \end{align*}
 Collecting and cancelling all these contributions,  and extracting perturbative paradifferential terms into $(G_s, K_s)$, all $L_1$ terms in $(\mathcal{G}^s_1, \mathcal{K}^s_1)$ cancel, and we prove the claim for \eqref{TTildewr2sSystem}.

 Then we consider the corrections consisting of the remaining $L$ bilinear forms in the secondary source term $(\mathcal{G}^s_2, \mathcal{K}^s_2)$.
 For $(\tilde{w}^s_3,\tilde{r}^s_3)$ defined by
 \begin{align*}
     &\tilde{w}^s_3 : = L\left((1+\W)Y_\alpha +\frac{\gamma}{g}\Re R_\alpha +\frac{\gamma^2}{g}\Im \W , \partial_\alpha^{-1}\tilde{w}^s_1\right), \\
     &\tilde{r}^s_3 : = 0,
 \end{align*}
 they are easily checked to cancel the secondary source terms $(\mathcal{G}^s_2,\mathcal{K}^s_2)$.

 Finally, by using the paraproduct estimate \eqref{LpLowhigh} to redistribute the derivatives if necessary, the normal form corrections satisfy the bound
 \begin{equation*}
     \|(\tilde{w}^s_2 +\tilde{w}^s_3, \tilde{r}^s_2 +\tilde{r}^s_3) \|_{\mathcal{H}^0} \lesssim_\uA \uA\|(w,r)\|_{\doth{s}}.
 \end{equation*}
 Let
 \begin{align*}
   &\tilde{w}^s =   \tilde{w}^s_1 +\tilde{w}^s_2 +\tilde{w}^s_3, \\
   &\tilde{r}^s =   \tilde{r}^s_1 +\tilde{r}^s_2 +\tilde{r}^s_3,
 \end{align*}
then $(\tilde{w}^s, \tilde{r}^s)$ solve \eqref{SourceParadifferential} with the desired bounds, and this finishes the proof of the proposition.
\end{proof}

\subsection{The bounds for the paradifferential source terms} \label{s:SourceBound}
In this subsection, we first compute the paradifferential representation of the source terms $(\underline{\mathcal{G}}_0, \underline{\mathcal{K}}_0)$ and $(\underline{\mathcal{G}}_1, \underline{\mathcal{K}}_1)$.
Then we compute the leading terms of para-material derivatives of $w,r$ and auxiliary functions $x = T_{1-Y}w$, $u = T_{1-Y}\partial^{-1}_\alpha w$.
\begin{lemma} \label{t:G0K0Estimate}
The source terms $(\underline{\mathcal{G}}_0, \underline{\mathcal{K}}_0)$ satisfy the bound
\begin{equation}
    \|(\underline{\mathcal{G}}_0, \underline{\mathcal{K}}_0) \|_{\mathcal{H}^0} \lesssim_{A^\sharp} \Astar \left(\|(w,r) \|_{\doth{\frac{1}{4}}}+ \gamma \|(w,r) \|_{\doth{-\frac{1}{4}}} \right). \label{MathcalG0K0}
\end{equation}
Moreover, they have the representation
\begin{equation*}
 \nP\underline{\mathcal{G}}_0 = \mathcal{G}_{0,0} +G, \quad  \nP\underline{\mathcal{K}}_0 = \mathcal{K}_{0,0} +K,
\end{equation*}
where $(G,K)$ are perturbative in the sense that they satisfy the quadratic bound
\begin{equation*}
 \| (G,K)\|_{\doth{\frac{1}{4}}}\lesssim_{\uAS}    \Astar \uAStar\left(\| (w,r)\|_{\doth{\frac{1}{4}}}+ \gamma \| (w,r)\|_{\doth{-\frac{1}{4}}}\right).
\end{equation*}
The main parts of the source terms $(\mathcal{G}_{0,0}, \mathcal{K}_{0,0})$ are given by
\begin{align*}
   &\mathcal{G}_{0,0} = -\nP[T_{1-\bar{Y}}\bar{r}_\alpha T_{1+\W}Y]+\nP[T_{1-\bar{Y}}\bar{x}_\alpha T_{1+\W}R] -i\frac{\gamma}{2}\nP[\bar{w}T_{1+\W}Y +T_{1-\bar{Y}}\bar{x}_\alpha T_{1+\W}W],\\
   &\mathcal{K}_{0,0} = -\nP[T_{1-\bar{Y}}\bar{r}_\alpha R] -i\frac{\gamma}{2}\nP[\bar{w}R - T_{1-\bar{Y}}\bar{r}_\alpha W].
\end{align*}
\end{lemma}
\begin{proof}
  Recall that
  \begin{align*}
     &\underline{\mathcal{G}}_0(w,r) = \mathcal{G}(w,r)+i\frac{\gamma}{2}(1+\mathbf{W})(\mathbf{P}\Bar{m}_1-\Bar{\mathbf{P}}m_1), \\
     &\underline{\mathcal{K}}_0(w,r) = \mathcal{K}(w,r)-i\frac{\gamma}{2}(\Bar{\mathbf{P}}m_2 + \mathbf{P}\Bar{m}_2),
  \end{align*}
  where
\begin{equation*}
m_1 = (1-\bar{Y})w - (1-Y)^2\bar{W}w_\alpha, \quad m_2 = \bar{R}w -(1-Y)\bar{W}(r_\alpha +R_\alpha w).
\end{equation*}
The corresponding analysis for $(\mathcal{G}, \mathcal{K})$ is already carried out in Lemma $5.3$
of \cite{ai2023dimensional}, the only difference compared to \cite{ai2023dimensional} is that we can put $-\nP(T_{\bar{x}}\bar{R}_\alpha T_{1+\W}R)$ into $G$, and $-\nP(T_{\bar{x}}\bar{R}_\alpha R)$ into $K$ since we use the control norm $\uAStar$.
It  suffices to compute the vorticity terms.

We consider $-i\frac{\gamma}{2}\mathcal{G}_1(w,r)$ term first.
We claim that the term
\begin{equation*}
    -i\frac{\gamma}{2}\nP((1+\W)\bar{\nP}m_1) = -i\frac{\gamma}{2}\nP(\W \bar{\nP}m_1)
\end{equation*}
is perturbative.
Since by using the commutator structure and Lemma \ref{t:CommutatorL2},
\begin{equation*}
    \|\nP(\W \bar{\nP}m_1) \|_{\dot{H}^{\frac{1}{4}}}= \||D|^{\frac{1}{4}}[\nP, \W]\bar{\nP} m_1 \|_{L^2}\lesssim \||D|^{\frac{1}{4}}\W \|_{BMO}\| \bar{\nP}m_1\|_{L^2},
\end{equation*}
 it only remains to show that
 \begin{equation}
     \gamma \| \bar{\nP}m_1\|_{L^2}\lesssim_\uA \Astar \left(\|w\|_{\dot{H}^{\frac{1}{4}}} + \gamma\|w\|_{\dot{H}^{-\frac{1}{4}}}\right). \label{PMOneLTwo}
 \end{equation}
 After expanding the $m_1$ expression, we apply the Lemma \ref{t:CommutatorL2},
 \begin{align*}
    & \gamma\| \bar{\nP}(\bar{Y}w)\|_{L^2} \lesssim \gamma \| [\bar{\nP}, \bar{Y}]w\|_{L^2} \lesssim \gamma\||D|^{\frac{1}{4}}Y\|_{BMO} \|w\|_{\dot{H}^{-\frac{1}{4}}} \lesssim_\uA \gamma \Astar  \|w\|_{\dot{H}^{-\frac{1}{4}}} ,\\
    & \gamma \|\bar{\nP}(\bar{W}w_\alpha)\|_{L^2} \lesssim \|[\bar{\nP}, \bar{W}]w_\alpha \|_{L^2} \lesssim \gamma \||D|^{\frac{3}{4}}W \|_{BMO}\| w\|_{\dot{H}^{\frac{1}{4}}} \lesssim_\uA \Astar \|w\|_{\dot{H}^{\frac{1}{4}}},\\
    & \gamma \|\bar{\nP}(Y\bar{W}w_\alpha)\|_{L^2} \lesssim \|Y\|_{L^\infty} \gamma \|\bar{\nP}(\bar{W}w_\alpha)\|_{L^2} \lesssim_\uA \Astar \|w\|_{\dot{H}^{\frac{1}{4}}}.
 \end{align*}

 Next, for the $i\frac{\gamma}{2}\nP((1+\W)\nP\bar{m}_1)$ term, the high-low and high-high components are balanced owing to \eqref{LpBalance} and \eqref{PMOneLTwo},
 so that
 \begin{equation*}
   i\frac{\gamma}{2}\nP((1+\W)\nP\bar{m}_1) = i\frac{\gamma}{2}T_{1+\W}\nP\bar{m}_1 +G,
 \end{equation*}
 where $G$ is perturbative.
 We then write
 \begin{equation*}
  \nP\bar{m}_1 =- \nP[Y\bar{w}+(1-\bar{Y})^2W\bar{w}_\alpha].
 \end{equation*}
For each cubic term, when the lowest frequency variable is differentiated, it may be absorbed into $G$.
In particular, $Y$ and $W$ cannot have the lowest frequency due to the Littlewood-Paley projection $\nP$.
Also, when $\bar{w}_\alpha$ has the lowest frequency, the term is perturbative.
Therefore, after applying \eqref{ParaAssociateTwo} we obtain
\begin{equation*}
      i\frac{\gamma}{2}\nP((1+\W)\nP\bar{m}_1) = -i\frac{\gamma}{2}\nP(\bar{w}T_{1+\W}Y +T_{1-\bar{Y}}\bar{x}_\alpha T_{1+\W}W) +G.
\end{equation*}

Then we consider the $-i\frac{\gamma}{2}\mathcal{K}_1(w,r)$ term.
Applying the Lemma \ref{t:CommutatorL2} to $\gamma\bar{\nP}m_2$ in $\dot{H}^{\frac{1}{2}}$, we similarly compute
\begin{align*}
   &\gamma\| \bar{\nP}(\bar{R}w)\|_{\dot{H}^{\frac{1}{2}}} \lesssim \gamma\| |D|^{\frac{1}{2}}[\bar{\nP},\bar{R}]w\|_{L^2} \lesssim \gamma \||D|^{\frac{3}{4}}R\|_{BMO}\|w\|_{\dot{H}^{-\frac{1}{4}}} \lesssim_\uA \Astar  \gamma\|w\|_{\dot{H}^{-\frac{1}{4}}} ,\\
   &  \gamma \|\bar{\nP}(\bar{W}r_\alpha)\|_{\dot{H}^{\frac{1}{2}}} \lesssim \||D|^{\frac{1}{2}}[\bar{\nP}, \bar{W}]r_\alpha \|_{L^2} \lesssim \gamma \||D|^{\frac{3}{4}}W \|_{BMO}\| r\|_{\dot{H}^{\frac{3}{4}}} \lesssim_\uA \Astar \|r\|_{\dot{H}^{\frac{3}{4}}},\\
    & \gamma \|\bar{\nP}(Y\bar{W}r_\alpha)\|_{\dot{H}^{\frac{1}{2}}} \lesssim  \gamma \| \bar{\nP}(\bar{W}Y)\|_{\dot{W}^{\frac{1}{2},4}} \|r_\alpha\|_{L^4} \lesssim \gamma \||D|^{\frac{1}{2}}[\bar{P},\bar{W}]Y \|_{L^4}\|r_\alpha\|_{\dot{H}^{\frac{3}{4}}}\lesssim_\uA \uAS\Astar \|r\|_{\dot{H}^{\frac{3}{4}}}, \\
    &\gamma \|\bar{\nP}[\bar{W}R_\alpha w] \|_{\dot{H}^{\frac{1}{2}}} +\gamma \|\bar{\nP}[Y\bar{W}R_\alpha w] \|_{\dot{H}^{\frac{1}{2}}} \lesssim_\uA \uAS \Astar  \left(\|w\|_{\dot{H}^{\frac{1}{4}}} + \gamma\|w\|_{\dot{H}^{-\frac{1}{4}}} \right).
\end{align*}
Putting these together, we get
\begin{equation}
    \gamma \|\bar{\nP}m_2\|_{\dot{H}^{\frac{1}{2}}} \lesssim_\uA \Astar \left(\|(w,r) \|_{\doth{\frac{1}{4}}} + \gamma \|(w,r)\|_{\doth{-\frac{1}{4}}}\right). \label{PMTwoLTwo}
\end{equation}
$\bar{\nP}m_2$ becomes zero after applying the Littlewood-Paley projection $\nP$.
We argue as before to get
\begin{equation*}
    -i\frac{\gamma}{2}\nP \bar{m}_2 = -i\frac{\gamma}{2}\nP(\bar{w}R -(T_{\bar{x}}\bar{R}_\alpha + T_{1-\bar{Y}}\bar{r}_\alpha)W) +K =-i\frac{\gamma}{2}\nP(\bar{w}R-T_{1-\bar{Y}}\bar{r}_\alpha W) +K,
\end{equation*}
because we have
\begin{equation*}
    \gamma\|\nP(T_{\bar{x}}\bar{R}_\alpha W) \|_{\dot{H}^{\frac{3}{4}}}\lesssim \|\bar{x}\|_{L^4}\|\bar{R}_\alpha\|_{L^4}\gamma\||D|^{\frac{3}{4}}W\|_{BMO}\lesssim \Astar \uAStar \|w\|_{\dot{H}^{\frac{1}{4}}},
\end{equation*}
using Sobolev embedding.
Finally, the source terms estimate \eqref{MathcalG0K0} follows directly from \eqref{PMOneLTwo}, \eqref{PMTwoLTwo} and the corresponding estimates in  \cite{ai2023dimensional}.
\end{proof}

As for the source terms  $(\underline{\mathcal{G}}_1, \underline{\mathcal{K}}_1)$, we have the following result:
\begin{lemma}
    The  source terms $(\underline{\mathcal{G}}_1, \underline{\mathcal{K}}_1)$ have the decomposition
    \begin{equation*}
    \begin{aligned}
        &\nP \underline{\mathcal{G}}_1 =  -T_w T_{1-\bar{Y}}R_\alpha +G, \\
        &\nP \underline{\mathcal{K}}_1 = -iT_wT_{g+\ua}Y +K,
    \end{aligned}
    \end{equation*}
 where $(G,K)$ satisfy the  quadratic bound \eqref{MathcalG0K0}. \label{t:G1K1Estimate}
\end{lemma}
\begin{proof}
In the source term $\underline{\mathcal{G}}_1$, the anti-holomorphic term $T_{r_\alpha}\bar{Y}$ is eliminated after applying the Littlewood-Paley projection $\nP$.
For terms having $r_\alpha$ or $w_\alpha$, they satisfy \eqref{MathcalG0K0}.
For instance, we use \eqref{LpBalance} to estimate
\begin{equation*}
    \|T_{r_\alpha}\bar{Y}\|_{L^2}+ \|\Pi(r_\alpha, \bar{Y})\|_{L^2} \lesssim \||D|^{\frac{3}{4}}r\|_{L^2}\| |D|^{\frac{1}{4}}Y\|_{BMO}\lesssim \Astar \|r\|_{\dot{H}^{\frac{3}{4}}}.
\end{equation*}
For the rest of the terms in $\underline{\mathcal{G}}_1$, we write
\begin{equation*}
-(T_w ((1-\bar{Y})R_\alpha+ \gamma \Im \W)
   + \Pi(w, (1-\bar{Y})R_\alpha + \gamma \Im\W)) = -T_w T_{1-\bar{Y}}R_\alpha  +G,
\end{equation*}
where terms except $-T_w T_{1-\bar{Y}}R_\alpha$ are balanced, and we can put them into $G$.
For instance, using the fact that $\dot{H}^{\frac{1}{4}}(\mathbb{R})\hookrightarrow L^4(\mathbb{R})$,
we can bound
\begin{equation*}
\| T_w \Pi(\bar{Y}, R_\alpha)\|_{L^2}\lesssim \|w\|_{L^4}\| \Pi(\bar{Y}, R_\alpha)\|_{L^4} \lesssim \||D|^{\frac{1}{4}}Y\|_{BMO}\||D|^{\frac{3}{4}} R\|_{L^4}\|w\|_{\dot{H}^{\frac{1}{4}}}\lesssim_{A^\sharp}\Astar \|w\|_{\dot{H}^{\frac{1}{4}}}.
\end{equation*}
We also use \eqref{LpBalance} to estimate
\begin{equation*}
    \gamma\|T_w \Im \W\|_{L^2}\lesssim \gamma\||D|^{-\frac{1}{4}}w\|_{L^2} \||D|^{\frac{1}{4}}\W\|_{BMO} \lesssim \Astar \gamma\|w\|_{\dot{H}^{-\frac{1}{4}}}.
\end{equation*}
For other terms in $\underline{\mathcal{K}}_1$, we use \eqref{ParaProducts} to write
\begin{equation*}
i(T_{1-Y}T_w\ua+T_{1-Y}\Pi(w,\ua)-T_{(g+\ua)w}Y-\Pi((g+\ua)w, Y)) = -iT_wT_{g+\ua}Y +K,
\end{equation*}
where  terms except $-iT_wT_{g+\ua}Y$ are balanced, and they can be absorbed into $K$.
For instance, using \eqref{uaAsharp} for $\ua$ as well as the Sobolev embedding $\dot{H}^{\frac{1}{4}}(\mathbb{R})\hookrightarrow L^4(\mathbb{R})$, we bound
\begin{equation*}
    \|T_{1-Y}T_w\ua \|_{\dot{H}^{\frac{1}{2}}}\lesssim (1+\|Y\|_{L^\infty})\|w\|_{L^4}\||D|^{\frac{1}{2}}\ua \|_{L^4} \lesssim_{A^\sharp}\Astar \|w\|_{\dot{H}^{\frac{1}{4}}}.
\end{equation*}
These give the decomposition for   $(\nP\underline{\mathcal{G}}_1, \nP\underline{\mathcal{K}}_1)$.
\end{proof}
Putting together Lemma \ref{t:G0K0Estimate}, Lemma \ref{t:G1K1Estimate} and  equations \eqref{ParadifferentialLinearEqn}, one can compute the following para-material derivatives for $(w,r)$:
\begin{lemma}
  For $(w,r)$ that satisfy \eqref{ParadifferentialLinearEqn}, they have the representation
  \begin{equation*}
\left\{
             \begin{array}{lr}
             T_{D_t}w= -T_{1-\bar{Y}}(r_\alpha + T_w R_\alpha ) + G_2=: G_1 +G_2 &\\
           T_{D_t}r =-i\gamma r +iT_{g+\ua} (x -T_w Y)+ K_2 =: K_1+ K_2,&
             \end{array}
\right.
\end{equation*}
and likewise with $(D_t w, D_t r)$ in place of $(T_{D_t}w, T_{D_t}r)$.
$(G_1, K_1)$ satisfy the linear bound
\begin{equation*}
    \||D|^{-\frac{1}{4}}(G_1, K_1)\|_{\mathcal{H}^0} \lesssim_{\uAS} \| (w,r)\|_{\doth{\frac{1}{4}}}+ \gamma \|(w,r) \|_{\doth{-\frac{1}{4}}},
\end{equation*}
and $(G_2, K_2)$ satisfy the quadratic bound \eqref{MathcalG0K0}.  \label{t:wrParaMaterialD}
\end{lemma}
\begin{proof}
From Lemma \ref{t:G0K0Estimate}, we know that the source terms $(\nP\underline{\mathcal{G}}_0, \nP\underline{\mathcal{K}}_0)$ can be absorbed into $(G_2, K_2)$.
The source terms $(\nP\underline{\mathcal{G}}_1, \nP\underline{\mathcal{K}}_1)$ and the terms on the left-hand side of \eqref{ParadifferentialLinearEqn} belong to $(G_1, K_1)$.
The only exceptions are $T_{(1-\bar{Y})R_\alpha }w$ and $\gamma T_{\Im \W}w$.
For these two terms. we have
\begin{align*}
&\|T_{(1-\bar{Y})R_\alpha }w\|_{L^2} \lesssim \|\bar{Y}\|_{L^\infty}\||D|^{\frac{3}{4}}R\|_{BMO}\|w\|_{\dot{H}^{\frac{1}{4}}} \lesssim_A \Astar \|w\|_{\dot{H}^{\frac{1}{4}}},\\
&\gamma\|T_{\Im \W}w\|_{L^2}\lesssim \gamma \||D|^{\frac{3}{4}}W\|_{BMO}\|w\|_{\dot{H}^{\frac{1}{4}}} \lesssim \Astar \|w\|_{\dot{H}^{\frac{1}{4}}},
\end{align*}
so that they can be put into $G_2$.
We can use para-commutator \eqref{ParaCommutator} to reorder paraproducts freely.
For the difference between $(T_{D_t}w, T_{D_t}r)$ and $(D_t w, D_t r)$,
\begin{align*}
    &\|T_{D_t}w - D_t w\|_{L^2} \leq \|T_{w_
    \alpha}\ub\|_{L^2} + \|\Pi(w_\alpha, \ub) \|_{L^2} \lesssim \|w\|_{\dot{H}^{\frac{1}{4}}}\||D|^{\frac{3}{4}}\ub\|_{BMO} \lesssim \Astar \|w\|_{\dot{H}^{\frac{1}{4}}}, \\
    &\|T_{D_t}r - D_t r\|_{\dot{H}^{\frac{1}{2}}} \leq \|T_{r_\alpha}\ub\|_{\dot{H}^{\frac{1}{2}}} + \|\Pi(r_\alpha, \ub) \|_{\dot{H}^{\frac{1}{2}}} \lesssim \|r\|_{\dot{H}^{\frac{3}{4}}}\||D|^{\frac{3}{4}}\ub\|_{BMO} \lesssim \Astar \|r\|_{\dot{H}^{\frac{3}{4}}},
\end{align*}
so that they can be added to $(G_2,K_2)$, and  $(D_t w, D_t r)$ have the same representation as $(T_{D_t}w, T_{D_t}r)$.

To estimate $(G_1, K_1)$, we notice that the terms
\begin{equation*}
    -T_{1-\bar{Y}}r_\alpha, \quad -i\gamma r, \quad iT_{g+\ua}x
\end{equation*}
are already balanced, and their estimates are straightforward.
For the two other terms, we use the $L^4$ Sobolev embedding
\begin{align*}
    &\||D|^{-\frac{1}{4}}T_{1-\bar{Y}}r_\alpha\|_{L^2} \lesssim \|1-\bar{Y}\|_{L^\infty}\|w\|_{L^4}\||D|^{\frac{3}{4}}R\|_{L^4}\lesssim_\uA A^\sharp \|w\|_{\dot{H}^{\frac{1}{4}}},\\
     &\|i|D|^{-\frac{1}{4}}T_{g+\ua}T_w Y \|_{\dot{H}^{\frac{1}{2}}}\lesssim \|g+\ua\|_{L^\infty}\|w\|_{L^4}\||D|^{\frac{1}{4}}Y\|_{L^4}\lesssim_\uA A^\sharp \|w\|_{\dot{H}^{\frac{1}{4}}},
\end{align*}
where we use the fact that
\begin{equation*}
    \||D|^{\frac{1}{4}}Y\|_{L^4}\lesssim \||D|^{\frac{1}{4}}\W\|_{L^4} \lesssim_\uA A^\sharp.
\end{equation*}
Therefore, we have shown that $(G_1, K_1)$ satisfy the linear bound.
\end{proof}

Finally, we compute the para-material derivatives of $x  = T_{1-Y}w$, $\partial^{-1}_\alpha r$, and $u= T_{1-Y}\partial^{-1}_\alpha w$.
\begin{lemma}
   For $(w,r)$ that satisfy \eqref{ParadifferentialLinearEqn}, $T_{D_t}x$ has the following representation
   \begin{equation*}
     T_{D_t}x = -T_{1-\bar{Y}}(T_{1-Y}r_\alpha + T_x R_\alpha)+G_2 = :G_1 +G_2,
   \end{equation*}
   where $G_1$ satisfies the linear bound
   \begin{equation*}
       \||D|^{-\frac{1}{4}}G_1\|_{L^2}\lesssim_{A^\sharp} \|(w,r)\|_{\doth{\frac{1}{4}}},
   \end{equation*}
  and $G_2$ satisfies the quadratic bound
  \begin{equation*}
      \|G_2\|_{L^2}\lesssim_{A^\sharp}\Astar \| (w,r)\|_{\doth{\frac{1}{4}}}.
  \end{equation*}
  Similar estimates hold when $T_{D_t}x$ is replaced by $D_t x$.
\end{lemma}
\begin{proof}
Applying the para-Leibniz rule \eqref{UnbalancedParaLeibnizOne}, we write
\begin{equation*}
    T_{D_t}x= T_{-T_{D_t}Y}w + T_{1-Y}T_{D_t}w + G_2,
\end{equation*}
where the para-Leibniz error can be placed into $G_2$.
For $T_{D_t}Y$, we  use \eqref{YParaDerivative}, which shows that the $T_{-T_{D_t}Y}w$ term on the right belongs to $G_2$.
For the $T_{1-Y}T_{D_t}w$ term, we apply the formula of $T_{D_t}w$ in Lemma \ref{t:wrParaMaterialD}  and para-commutator \eqref{ParaCommutator} to write
\begin{equation*}
    -T_{1-Y}T_{1-\bar{Y}}(r_\alpha + T_w R_\alpha )+ G_2 = -T_{1-\bar{Y}}(T_{1-Y}r_\alpha + T_x R_\alpha) + G_2.
\end{equation*}
The bounds for $G_1$, $G_2$ as well as the representation of $D_t w$ follow directly from Lemma \ref{t:wrParaMaterialD}.
\end{proof}
\begin{lemma}
    For $(w,r)$ that satisfy \eqref{ParadifferentialLinearEqn}, we have the following results for para-material derivatives of $\partial^{-1}_\alpha r$, $\partial_\alpha^{-1} w$ and $u = T_{1-Y}\partial^{-1}_\alpha w$:
    \begin{enumerate}
        \item Leading term of the para-material derivative of $\partial^{-1}_\alpha r$:
        \begin{equation*}
            T_{D_t}\partial^{-1}_\alpha r = -i\gamma \partial^{-1}_\alpha r + iT_{g+\ua}(u -T_w X) +K_2 =: K_1 +K_2,
        \end{equation*}
        and likewise with $D_t \partial^{-1}_\alpha r$ in place of $T_{D_t} \partial^{-1}_\alpha r$.
        $K_1$ satisfies the linear bound
        \begin{equation*}
            \gamma \|K_1 \|_{\dot{H}^{\frac{3}{4}}} + \|K_1 \|_{\dot{H}^{\frac{5}{4}}}\lesssim_{\uAS}\|(w,r)\|_{\doth{\frac{1}{4}}}+\gamma^2 \|(w,r)\|_{\doth{-\frac{3}{4}}},
        \end{equation*}
        and $K_2$ satisfies the quadratic bound
        \begin{equation*}
           \gamma^2\|K_2\|_{\dot{H}^{\frac{1}{2}}}  \lesssim_{A^\sharp} \Astar \left(\|(w,r) \|_{\doth{\frac{1}{4}}}+ \gamma^2 \|(w,r) \|_{\doth{-\frac{3}{4}}} \right).
        \end{equation*}
        \item Leading terms of the para-material derivatives of $\partial_\alpha^{-1} w$ and $ u$:
         \begin{align*}
         &T_{D_t}\partial_\alpha^{-1} w = -T_{1-\bar{Y}}(r+ T_w R) +G_2 =: G_1 +G_2,\\
        &T_{D_t}u = -T_{1-\bar{Y}}(T_{1-Y}r+ T_x R) +G_2 =: G_1 +G_2,
        \end{align*}
        and likewise with $(D_t\partial_\alpha^{-1} w, D_t u)$ in place of $(T_{D_t}\partial_\alpha^{-1} w, T_{D_t} u)$.
        $G_1$ satisfies the linear bound
        \begin{equation*}
            \gamma^2\|G_1 \|_{\dot{H}^{-\frac{1}{4}}}+ \|G_1 \|_{\dot{H}^{\frac{3}{4}}}\lesssim_{\uAS}\|(w,r)\|_{\doth{\frac{1}{4}}}+\gamma^2 \|(w,r)\|_{\doth{-\frac{3}{4}}},
        \end{equation*}
        and $G_2$ satisfies the quadratic bound
        \begin{equation*}
           \gamma^2 \|G_2\|_{L^2}  \lesssim_{A^\sharp} \Astar \left(\|(w,r) \|_{\doth{\frac{1}{4}}}+ \gamma^2 \|(w,r) \|_{\doth{-\frac{3}{4}}} \right).
        \end{equation*}
    \end{enumerate}
\end{lemma}
\begin{proof}
\begin{enumerate}
\item  We apply the anti-derivative to $T_{D_t}r$ in Lemma \ref{t:wrParaMaterialD} and use \eqref{ParaCommutator}, \eqref{ParaProducts} to write:
\begin{align*}
    T_{D_t}\partial^{-1}_\alpha r =& -i\gamma \partial^{-1}_\alpha r + iT_{g+\ua}(T_{1-Y}\partial^{-1}_\alpha w -T_w \partial^{-1}_\alpha Y)+[T_{D_t}, \partial^{-1}_\alpha]r\\
    & -i[T_{(g+\ua)(1-Y)},\partial^{-1}_\alpha]w +i[T_{(g+\ua)w}, \partial^{-1}_\alpha]Y + K_2.
\end{align*}
For the three commutator terms, we write
\begin{align*}
 [T_{D_t}, \partial^{-1}_\alpha]r & = T_\ub r - \partial^{-1}_\alpha(T_\ub r_\alpha)= \partial^{-1}_\alpha(T_{\ub_\alpha}r), \\
 [T_{(g+\ua)(1-Y)},\partial^{-1}_\alpha]w & = T_{(g+\ua)(1-Y)}\partial^{-1}_\alpha w - \partial^{-1}_\alpha(T_{(g+\ua)(1-Y)}w) = \partial^{-1}_\alpha(T_{[(g+\ua)(1-Y)]_\alpha}\partial^{-1}_\alpha w),\\
[T_{(g+\ua)w}, \partial^{-1}_\alpha]Y  & = T_{(g+\ua)w}\partial^{-1}_\alpha Y - \partial^{-1}_\alpha(T_{(g+\ua)w} Y)= \partial^{-1}_\alpha(T_{[(g+\ua)w]_\alpha}\partial^{-1}_\alpha Y).
\end{align*}
Using \eqref{LpBalance} and \eqref{LpLowhigh}, these commutator terms may be absorbed into the error $K_2$.
By $(4)$ of Lemma \ref{t:XParaMaterial}, one can write
\begin{equation*}
    \partial^{-1}_\alpha Y = \partial^{-1}_\alpha(T_{1-Y}X_\alpha) + K_3 = X - \partial^{-1}_\alpha T_Y X_\alpha + K_3,
\end{equation*}
where the error $K_3$ satisfies
\begin{equation*}
     \gamma^2\|K_3\|_{\dot{W}^{\frac{1}{2},4}}\lesssim_{A^\sharp} \Astar.
\end{equation*}
We can bound
\begin{align*}
   &\gamma^2 \|T_{g+\ua}T_w  \partial^{-1}_\alpha T_Y X_\alpha\|_{\dot{H}^{\frac{1}{2}}} \lesssim \gamma^2 \| g+\ua\|_{L^\infty}\|w\|_{\dot{H}^{-\frac{3}{4}}}\|T_Y X_\alpha\|_{BMO^{\frac{1}{4}}}\lesssim_\uA \gamma^2\Astar\|w\|_{\dot{H}^{-\frac{3}{4}}},\\
   &\gamma^2 \|T_{g+\ua}T_w  K_3\|_{\dot{H}^{\frac{1}{2}}} \lesssim \gamma^2\| g+\ua\|_{L^\infty} \|w\|_{L^4}\|K_3\|_{\dot{W}^{\frac{1}{2},4}}\lesssim_{A^\sharp} \Astar\|w\|_{\dot{H}^{\frac{1}{4}}},
\end{align*}
so that  $\partial^{-1}_\alpha Y$  can be replaced by $X$.

For the difference between $T_{D_t}\partial^{-1}_\alpha r$ and $D_t \partial^{-1}_\alpha r$,
\begin{equation*}
    \gamma^2\|T_{D_t}\partial^{-1}_\alpha r- D_t \partial^{-1}_\alpha  r\|_{\dot{H}^{\frac{1}{2}}} \leq \gamma^2\|T_{r}\ub\|_{\dot{H}^{\frac{1}{2}}} + \gamma^2\|\Pi(r, \ub) \|_{\dot{H}^{\frac{1}{2}}} \lesssim \gamma^2 \|r\|_{\dot{H}^{-\frac{1}{4}}}\| |D|^{\frac{3}{4}}\ub\|_{BMO},
\end{equation*}
and the difference may be absorbed into $K_2$.
For the estimate of $K_1$, we have
\begin{align*}
    &\gamma^2\|\partial^{-1}_\alpha r \|_{\dot{H}^{\frac{3}{4}}}+ \gamma\|\partial^{-1}_\alpha r \|_{\dot{H}^{\frac{5}{4}}} \lesssim \|r\|_{\dot{H}^{\frac{3}{4}}} + \gamma^2\|r\|_{\dot{H}^{-\frac{1}{4}}}, \\
    &\gamma\|T_{g+\ua}u\|_{\dot{H}^{\frac{3}{4}}}+\|T_{g+\ua}u\|_{\dot{H}^{\frac{5}{4}}}\lesssim \|g+\ua\|_{L^\infty}\|1-Y\|_{L^\infty}\left(\|w\|_{\dot{H}^{\frac{1}{4}}}+ \gamma^2\|w\|_{\dot{H}^{\frac{3}{4}}}\right)\\
    &\gamma\|T_{g+\ua}T_w X\|_{\dot{H}^{\frac{3}{4}}} + \|T_{g+\ua}T_w X\|_{\dot{H}^{\frac{5}{4}}}\lesssim_\uAS \|g+\ua\|_{L^\infty}\|w\|_{L^4}\left(\|X\|_{\dot{W}^{\frac{5}{4},4}}+ \gamma^2\|X\|_{\dot{W}^{\frac{1}{4},4}}\right),
\end{align*}
and they satisfy the linear bound.

\item  We apply the para-Leibniz rule Lemma \ref{t:Leibniz} to $u$:
\begin{equation*}
    T_{D_t}u = - T_{T_{D_t}Y}\partial^{-1}_\alpha w + T_{1-Y}T_{D_t}\partial^{-1}_\alpha w + G_2,
\end{equation*}
where the para-Leibniz error can be placed into $G_2$.
For $T_{D_t}Y$, we  use \eqref{YParaDerivative}, which shows that the $T_{T_{D_t}Y}\partial^{-1}_\alpha w$ term on the right belongs to $G_2$.
For the $T_{1-Y}T_{D_t}\partial^{-1}_\alpha w$ term, we apply the anti-derivative to $T_{D_t}w$ in Lemma \ref{t:wrParaMaterialD}, and para-products \eqref{ParaProducts} to write
\begin{equation*}
   T_{D_t}\partial^{-1}_\alpha w = -T_{1-\bar{Y}}(r+T_w R)+[T_{D_t}, \partial^{-1}_\alpha]w+ [T_{\bar{Y}}, \partial^{-1}_\alpha]r_\alpha - [T_{(1-\bar{Y})w}, \partial^{-1}_\alpha ]R_\alpha + G_2.
\end{equation*}
The three commutator terms can be rewritten as,
\begin{align*}
  &[T_{D_t}, \partial^{-1}_\alpha ]w  = T_\ub w - \partial^{-1}_\alpha (T_\ub w_\alpha)= \partial^{-1}_\alpha(T_{\ub_\alpha}w),   \\
  &[T_{\bar{Y}}, \partial^{-1}_\alpha]r_\alpha = T_{\bar{Y}}r -\partial^{-1}_\alpha(T_{\bar{Y}}r_\alpha) = \partial^{-1}_\alpha(T_{\bar{Y}_\alpha} r),\\
&[T_{(1-\bar{Y})w}, \partial^{-1}_\alpha]R_\alpha = T_{(1-\bar{Y})w}R -\partial^{-1}_\alpha(T_{(1-\bar{Y})w}R_\alpha) = \partial^{-1}_\alpha(T_{[(1-\bar{Y})w]_\alpha}R).
\end{align*}
Using \eqref{LpBalance} and \eqref{LpLowhigh}, these commutator terms may be absorbed into the error $G_2$.
Applying $T_{1-Y}$ to $T_{D_t}\partial^{-1}_\alpha w$ and using \eqref{ParaProducts}, we obtain the leading term of $T_{D_t}u$.

For the difference between $T_{D_t}u$ and $D_t u$,
we have
\begin{equation*}
\gamma^2\|T_{D_t}u-D_t u\|_{L^2}\leq \gamma^2\|T_{u_\alpha}\ub\|_{L^2}+\gamma^2\|\Pi(u_\alpha, \ub)\|_{L^2}\lesssim \gamma^2\|w\|_{\dot{H}^{-\frac{3}{4}}}\||D|^{\frac{3}{4}}\ub\|_{BMO},
\end{equation*}
so that it may be absorbed into $G_2$.
Finally, to estimate $G_1$, we have
\begin{align*}
&\gamma^2\|T_{1-\bar{Y}}T_{1-Y}r\|_{\dot{H}^{-\frac{1}{4}}}+ \|T_{1-\bar{Y}}T_{1-Y}r\|_{\dot{H}^{\frac{3}{4}}} \lesssim_\uA \|r\|_{\dot{H}^{\frac{3}{4}}} + \gamma^2\|r\|_{\dot{H}^{-\frac{1}{4}}}, \\
&\gamma^2\| T_{1-\bar{Y}}T_x R\|_{\dot{H}^{-\frac{1}{4}}}+ \| T_{1-\bar{Y}}T_x R\|_{\dot{H}^{\frac{3}{4}}} \lesssim \|x\|_{L^4}\left(\gamma^2\|R\|_{\dot{W}^{-\frac{1}{4},4}}+\|R\|_{\dot{W}^{\frac{3}{4},4}}\right)\lesssim \uAS \|w\|_{\dot{H}^{\frac{1}{4}}},
\end{align*}
and they satisfy the linear bound.
\end{enumerate}
\end{proof}

\subsection{Normal form analysis for $(\underline{\mathcal{G}}_1, \underline{\mathcal{K}}_1)$} \label{s:GKOne}
In this subsection, we use the paradifferential normal form corrections to remove the source terms $(\nP\underline{\mathcal{G}}_1, \nP\underline{\mathcal{K}}_1)$ modulo balanced cubic terms in the linearized equations.
Recall that we have
\begin{align*}
  \nP\underline{\mathcal{G}}_1 = & \Pi(r_\alpha, \bar{Y}) -\nP T_{w_\alpha}\ub - \Pi(w_\alpha, \ub) - T_w \nP((1-\bar{Y})R_\alpha - i\frac{\gamma}{2}\W) - \Pi(w, (1-\bar{Y})R_\alpha +\gamma \Im \W),\\
  \nP\underline{\mathcal{K}}_1 = & -\nP T_{r_\alpha}\ub - \Pi(r_\alpha, \ub)+ i(\nP T_{1-Y}T_w\ua+T_{1-Y}\Pi(w,\ua)-T_{(g+\ua)w}Y-\Pi((g+\ua)w, Y)).
\end{align*}
For simplicity regarding the balanced cases, we let $\Pi = \nP\Pi$ to include the Littlewood-Paley projection $\nP$.
We will  then prove the following result:
\begin{proposition}
Assume that $(w,r)$ solve the paradifferential equation \eqref{ParadifferentialLinearEqn}, then there exists a linear paradifferential normal form correction
\begin{equation*}
    (\tilde{w}, \tilde{r}) = NF(w,r)
\end{equation*}
that solves the paradifferential equations
\begin{equation}
\left\{
             \begin{array}{lr}
             T_{D_t}\tilde{w}+T_{1-\bar{Y}}\partial_\alpha \tilde{r} + T_{(1-\bar{Y})R_\alpha }\tilde{w} +\gamma T_{\Im \W}\tilde{w} = -\nP \underline{\mathcal{G}}_1(w,r) +G &\\
           T_{D_t}\tilde{r} +i\gamma \tilde{r} -iT_{1-Y}T_{g+\ua}\tilde{w}=-\nP \underline{\mathcal{K}}_1(w,r) +K,&  \label{NormalFormGKOne}
             \end{array}
\right.
\end{equation}
with the following properties:
\begin{enumerate}
\item Quadratic correction bound:
\begin{equation}
    \|NF(w,r)\|_{\doth{\frac{1}{4}}} \lesssim_{\uA} \uAS\left(\| (w,r)\|_{\doth{\frac{1}{4}}}+ \gamma^2 \|(w,r) \|_{\doth{-\frac{3}{4}}}\right). \label{NFOneBound}
\end{equation}
\item Secondary correction bound:
\begin{equation}
\|NF(w,r)\|_{\doth{\frac{1}{4}}}\lesssim_\uA \Astar\left(\|(w,r)\|_{\mathcal{H}^0}+ \gamma^2 \|(w,r) \|_{\doth{-1}}\right). \label{NFTwoBound}
\end{equation}
\item Cubic error bound:
\begin{equation}
    \|(G, K)\|_{\doth{\frac{1}{4}}} \lesssim_{\uAS} \underline{A}_{\frac{1}{4}}\uAStar\left(\| (w,r)\|_{\doth{\frac{1}{4}}}+ \gamma^2 \|(w,r) \|_{\doth{-\frac{3}{4}}}\right). \label{NFCubicError}
\end{equation}
\end{enumerate}
\end{proposition}
Here, the secondary correction bound will be needed in order to estimate the effect of source terms for the inhomogeneous problem \eqref{ParadifferentialLinearEqn}.

As a general guideline, our paradifferential normal form corrections will be either of the two cases:
\begin{enumerate}
\item High-low interactions.
\item Balanced interactions with low frequency output.
\end{enumerate}

Motivated by the normal form transformation \eqref{e:NormalForm}, we define the corrections by
\begin{align*}
\tilde{w} : =& -\partial_\alpha(T_w X +\Pi(w, X))-\Pi(w_\alpha, \bar{X})\\
&-\frac{\gamma}{2g}\left(T_{w_\alpha}Z + \Pi(w_\alpha, Z) + T_{r}X_\alpha +\Pi(r, X_\alpha)+ \Pi(w_\alpha, \bar{Z})\right)\\
&-\frac{\gamma}{2g}\left(T_{r_\alpha}X +\Pi(r_\alpha, X) +T_w R+ \Pi(w, R)+\Pi(r_\alpha, \bar{X}) \right) \\
&+i\frac{\gamma^2}{2g}\left(T_{w_\alpha}U +\Pi(w_\alpha, U) + T_{\partial^{-1}_\alpha w}X_\alpha + \Pi(\partial_\alpha^{-1}w, X_\alpha)+\Pi(w_\alpha, \bar{U})\right)\\
&+i\frac{\gamma^2}{2g}\left(2T_w X + 2\Pi(w,X)+\frac{1}{2}\Pi(w, \bar{X})\right) \\
&-\frac{\gamma^2}{4g^2}\left(\partial_\alpha(T_r Z +\Pi(r,Z))+\Pi(r_\alpha, \bar{Z})\right)\\
&+i\frac{\gamma^3}{4g^2}\left(T_r X +\Pi(r,X) +T_w Z + \Pi(w, Z)+ \Pi(w, \bar{Z})\right)\\
&+i\frac{\gamma^3}{4g^2}\left(T_{r_\alpha}U + \Pi(r_\alpha, U)+T_{\partial^{-1}_\alpha w}R + \Pi(\partial^{-1}_\alpha w, R)-\Pi(r_\alpha, \bar{U})\right)\\
&+\frac{\gamma^4}{4g^2} \left(T_w U+T_{\partial_\alpha^{-1}w}X + \partial_\alpha\Pi(\partial_\alpha^{-1}w, U)-\Pi( w,\bar{U})\right),\\
\tilde{r} :  =& -T_{r_\alpha}X - \Pi(r_\alpha, X)-\Pi(r_\alpha, \bar{X})\\
&-\frac{\gamma}{2g}\left(T_{r_\alpha} Z +T_r R +\partial_\alpha\Pi(r,Z)+ \Pi(r_\alpha, \bar{Z})\right)\\
&+i\frac{\gamma}{2}\left(T_w X +\Pi(w, X)+\Pi(w, \bar{X})\right) \\
&+ i\frac{\gamma^2}{2g}\left(T_{r_\alpha}U +\Pi(r_\alpha, U) + T_{\partial_\alpha^{-1}w}R +\Pi(\partial_\alpha^{-1}w, R)-\Pi(r_\alpha, \bar{U})\right) \\
&+i\frac{\gamma^2}{4g}\left(T_w Z +\Pi(w, Z)+T_r X +\Pi(r,X)+\Pi(w, \bar{Z})\right) \\
&+\frac{\gamma^3}{4g}\left(T_w U +\Pi(w,U)+T_{\partial^{-1}_\alpha w}X +\Pi(\partial^{-1}_\alpha w, X)-\Pi(w, \bar{U})\right).
\end{align*}

We remark that this paradifferential correction  is  roughly the holomorphic part of the paralinearization of the normal form transformation \eqref{e:NormalForm} at the quadratic level.

We first verify that $(\tilde{w}, \tilde{r})$ satisfy the quadratic correction bound \eqref{NFOneBound}.
The balanced terms  are straightforward; for instance,
\begin{equation*}
    \gamma^2\| \Pi(w,X)\|_{\dot{H}^{\frac{1}{4}}}\lesssim_\uA \|w\|_{\dot{H}^{\frac{1}{4}}}\gamma^2\|X\|_{L^\infty} \lesssim \uA \|w\|_{\dot{H}^{\frac{1}{4}}}.
\end{equation*}
For low-high terms with $T_w$ or $\gamma^2 T_{\partial^{-1}_\alpha w} $, we use
\begin{equation*}
    \|T_f G\|_{\dot{H}^{\frac{1}{4}}} \lesssim \|f\|_{L^4}\||D|^{\frac{1}{4}}G\|_{L^4}\lesssim \|f\|_{\dot{H}^{\frac{1}{4}}}\|G\|_{\dot{W}^{\frac{1}{4},4}}
\end{equation*}
type of estimates with Sobolev embedding, which yields an $\uAS$ coefficient.
For other low high terms, we use \eqref{LpBalance}.

As for the secondary correction bound \eqref{NFTwoBound}, these estimates are also straightforward; for instance,
\begin{equation*}
    \|T_{r_\alpha}X \|_{\dot{H}^{\frac{3}{4}}}\lesssim \|r\|_{\dot{H}^{\frac{1}{2}}}\|D|^{\frac{1}{4}}\W\|_{BMO}\lesssim \Astar \|r\|_{\dot{H}^{\frac{1}{2}}}.
\end{equation*}

Next, we plug in the corrections to the system and compute the source terms.
In order to simplify the computation below, we consider the following cases where terms may be absorbed into $(G,K)$:
\begin{enumerate}
\item Cubic and higher order terms such that the lowest frequency variable is ``fully differentiated" or ``over differentiated" when combined with the vorticity $\gamma$.
The ``fully differentiated" cases include for instance
\begin{equation*}
    (|D|^{\frac{5}{4}}W, |D|^{\frac{3}{4}}R)\in BMO, \quad (|D|^{\frac{1}{4}}w, |D|^{\frac{3}{4}}r)\in L^2, \quad \gamma^2( |D|^{-\frac{3}{4}}w, |D|^{-\frac{1}{4}}r)\in L^2.
\end{equation*}
When there are even more derivatives on these terms, they are ``over differentiated".
\item  Cubic and higher order terms such that the lowest frequency variable is $\gamma \ua$.
For instance, we can estimate a typical term $\gamma^3 T_{T_\ua x}X$ in $\dot{H}^{\frac{1}{4}}$ using Sobolev embedding by
\begin{equation*}
    \gamma^3\| T_{T_\ua x}X \|_{\dot{H}^{\frac{1}{4}}}\lesssim \gamma\|\ua \|_{L^4}\|w\|_{L^4} \gamma^2\||D|^{\frac{1}{4}}X\|_{BMO} \lesssim_{\uAS} \uA^2_{\frac{1}{4}}\|w\|_{\dot{H}^{\frac{1}{4}}}.
\end{equation*}
\item  Cubic and higher order terms such that the lowest frequency variable is either $w$, $\gamma r$ or $\gamma^2 \partial_\alpha^{-1} w$.
Then we need to bound $w$, $\gamma r$ or $\gamma^2 \partial_\alpha^{-1} w$ in $L^4$ norm, and they are further bounded by
$\|w\|_{\dot{H}^{\frac{1}{4}}}$, $\gamma \|r\|_{\dot{H}^{\frac{1}{4}}}$ or $\gamma^2\|w\|_{\dot{H}^{-\frac{3}{4}}}$ using Sobolev embedding.
This will also bring an $\uAStar$ constant.
For instance, we can estimate a typical term $-\frac{\gamma}{2g}T_{R_\alpha}T_r X_\alpha$ in $\dot{H}^{\frac{1}{4}}$ by
\begin{equation*}
   -\frac{\gamma}{2g}\|T_{R_\alpha}T_r X_\alpha\|_{\dot{H}^{\frac{1}{4}}}\lesssim \|R_\alpha\|_{L^4}\gamma\|r\|_{L^4} \||D|^{\frac{1}{4}}X_\alpha\|_{BMO}\lesssim \gamma\Astar \uAStar \|r\|_{\dot{H}^{\frac{1}{4}}}
   \end{equation*}
   using Sobolev embedding.
\end{enumerate}
\begin{proof}
We insert $(\tilde{w}, \tilde{r})$ into the left-hand side of \eqref{NormalFormGKOne}.
We consider separately the high-low corrections, frequency-balanced corrections with holomorphic variables, and the frequency-balanced corrections with anti-holomorphic variables for both equations.
For each case, we will compute the contributions of the corrections according to the power of the vorticity $\gamma$.

When computing the para-material derivatives of paraproducts, we apply the para-Leibniz rule Lemma \ref{t:Leibniz} to distribute para-material derivatives to each variable.
In particular, when $w$ is at low frequency in the first equation or $\gamma r$ is at low frequency in the second equation, we apply Lemma \ref{t:Leibniz} in the case $\sigma = 0.$
For the expressions of material derivatives and para-material derivatives of each variable, we use the leading terms plus errors in $BMO$ or $L^2$ based spaces in corresponding lemmas in Section \ref{s:TDtBound} and Section \ref{s:SourceBound}.
The only exceptions are $-\partial_\alpha T_w T_{D_t}X$, $-\frac{\gamma}{2g}T_r T_{D_t}X_\alpha$,  $i\frac{\gamma^2}{2g}T_{\partial_\alpha^{-1} w}T_{D_t}X_\alpha$, $-\frac{\gamma}{2g}T_r T_{D_t}R$ and $i\frac{\gamma^2}{2g}T_{\partial_\alpha^{-1}w}T_{D_t}R$.
For these terms, we need either full expressions of para-material derivatives instead of just the leading terms, or use the leading terms but applying $L^4$ based estimates for the errors.
We also use the para-products rule
\eqref{ParaProducts} to simplify some terms.

$(1)$ We begin with the high-low corrections in $(\tilde{w}, \tilde{r})$ for the first equation.
Consider first  the contributions of non-vorticity terms.
\begin{align*}
 &- T_{D_t}\partial_\alpha T_w X = -\partial_\alpha T_{D_t}T_w X + T_{\ub_\alpha}\partial_\alpha T_w X \\
 = &-\partial_\alpha T_{D_t w} X - \partial_\alpha T_w T_{D_t}X + T_{\ub_\alpha} T_w X_\alpha +G \\
 = &\partial_\alpha T_{T_{1-\bar{Y}}(r_\alpha + T_w R_\alpha )}X - \partial_\alpha T_w(T_{T_{1-\bar{Y}}R_\alpha}X
 -\nP[(1-\bar{Y})R]-\Pi(X_\alpha, \ub)) \\
 &-i\frac{\gamma}{2}\partial_\alpha T_w(\nP[(1-\bar{Y})W] - X) + T_w T_{T_{1-\bar{Y}}R_\alpha + T_{1-Y}\bar{R}_\alpha}X_\alpha -i\frac{\gamma}{2}T_wT_{T_{1-\bar{Y}}\W - T_{1-Y}\bar{\W}}X_\alpha+G\\
 =& \partial_\alpha T_{T_{1-\bar{Y}}r_\alpha}X +\partial_\alpha T_w \nP\left[(1-\bar{Y})R-i\frac{\gamma}{2}(1-\bar{Y})W\right] +T_w\partial_\alpha\Pi(X_\alpha, \ub)+i\frac{\gamma}{2}\partial_\alpha T_w X\\
 & + T_w T_{T_{1-\bar{Y}}R_\alpha + T_{1-Y}\bar{R}_\alpha}X_\alpha -i\frac{\gamma}{2}T_wT_{T_{1-\bar{Y}}\W - T_{1-Y}\bar{\W}}X_\alpha+G\\
 =& \partial_\alpha T_{T_{1-\bar{Y}}r_\alpha}X + \nP T_{w_\alpha} \ub+ T_w \nP\left[(1-\bar{Y})R_\alpha-i\frac{\gamma}{2}\W\right] +i\frac{\gamma}{2}T_{w}\nP\partial_\alpha(\bar{Y}W) -T_w \nP(\bar{Y}_\alpha R)\\
 & +T_w\partial_\alpha\Pi(X_\alpha, \ub)+i\frac{\gamma}{2}\partial_\alpha T_w X + T_w T_{T_{1-\bar{Y}}R_\alpha + T_{1-Y}\bar{R}_\alpha}X_\alpha -i\frac{\gamma}{2}T_wT_{T_{1-\bar{Y}}\W - T_{1-Y}\bar{\W}}X_\alpha+G,\\
 =& \partial_\alpha T_{T_{1-\bar{Y}}r_\alpha}X + \nP T_{w_\alpha} \ub+ T_w \nP\left[(1-\bar{Y})R_\alpha-i\frac{\gamma}{2}\W\right] +i\frac{\gamma}{2}T_{w}\nP\partial_\alpha(\bar{Y}W) \\
 & +i\frac{\gamma}{2}\partial_\alpha T_w X + T_w T_{T_{1-\bar{Y}}R_\alpha }X_\alpha -i\frac{\gamma}{2}T_wT_{T_{1-\bar{Y}}\W - T_{1-Y}\bar{\W}}X_\alpha+G.
\end{align*}
Here for the last four lines, we use the fact that
\begin{equation*}
  T_{w_\alpha}\nP[(1-\bar{Y})R-i\frac{\gamma}{2}(1-\bar{Y})W] = \nP T_{w_\alpha}\ub.
\end{equation*}
In a similar manner, for other non-vorticity terms, we compute
\begin{align*}
    -T_{1-\bar{Y}}\partial_\alpha T_{r_\alpha}X &= -\partial_\alpha T_{T_{1-\bar{Y}}r_\alpha}X +G,\\
    -T_{(1-\bar{Y})R_\alpha}\partial_\alpha T_w X &= -T_w T_{T_{1-\bar{Y}}R_\alpha}X_\alpha +G.
\end{align*}

For the contributions of corrections with $\gamma$ coefficient, we have
\begin{align*}
-\gamma T_{\Im \W}\partial_\alpha T_w X =& i\frac{\gamma}{2}T_wT_{T_{1-\bar{Y}}\W - T_{1-Y}\bar{\W}}X_\alpha +G, \\
-\frac{\gamma}{2g}T_{D_t}T_{w_\alpha}Z =& - \frac{\gamma}{2g}T_{\partial_\alpha T_{D_t}w}Z - \frac{\gamma}{2g}T_{w_\alpha}T_{D_t}Z + G\\
=&\frac{\gamma}{2g}T_{T_{1-\bar{Y}}(r_{\alpha \alpha}+ T_w R_{\alpha\alpha} )}Z -i\frac{\gamma}{2}T_{w_\alpha}X +i\frac{\gamma^2}{2g}T_{w_\alpha} Z +G\\
 =& \frac{\gamma}{2g}T_{1-\bar{Y}}T_{r_{\alpha\alpha}}Z  -i\frac{\gamma}{2}T_{w_\alpha}X +i\frac{\gamma^2}{2g}T_{w_\alpha} Z +G,\\
 -\frac{\gamma}{2g}T_{D_t}T_r X_\alpha =& -\frac{\gamma}{2g}T_{T_{D_t}r}X_\alpha-\frac{\gamma}{2g}T_r T_{D_t}X_\alpha +G\\
 =& i\frac{\gamma^2}{2g}T_r X_\alpha-i\frac{\gamma}{2}T_{ w}X_\alpha + \frac{\gamma}{2g}T_rT_{1-\bar{Y}}R_\alpha +G,\\
 -\frac{\gamma}{2g}T_{D_t}T_{r_\alpha}X =& -\frac{\gamma}{2g}T_{\partial_\alpha T_{D_t}r}X-\frac{\gamma}{2g}T_{r_\alpha}T_{D_t}X +G \\
=& i\frac{\gamma^2}{2g}T_{r_\alpha}X-i\frac{\gamma}{2g}T_{T_{g+\ua} (x_\alpha -T_w Y_\alpha)}X +\frac{\gamma}{2g}T_{r_\alpha}T_{1-\bar{Y}}R +G\\
 =& i\frac{\gamma^2}{2g}T_{r_\alpha}X-i\frac{\gamma}{2} T_{w_\alpha} X +\frac{\gamma}{2g}T_{r_\alpha}T_{1-\bar{Y}}R +G,\\
-\frac{\gamma}{2g}T_{D_t}T_w R =& -\frac{\gamma}{2g}T_{D_t w}R -\frac{\gamma}{2g}T_w T_{D_t}R \\
=& \frac{\gamma}{2g}T_{T_{1-\bar{Y}}r_\alpha + T_w R_\alpha}R -\frac{\gamma}{2g}T_{w}(iT_{g+\ua}Y-i\gamma R)+G \\
=& \frac{\gamma}{2g}T_{1-\bar{Y}}T_{r_\alpha}R -i\frac{\gamma}{2g}T_w T_{g+\ua}Y +i\frac{\gamma^2}{2g}T_w R +G,\\
-\frac{\gamma}{2g}T_{1-\bar{Y}}\partial_\alpha (T_{r_\alpha} Z+T_r R) =& -\frac{\gamma}{2g} T_{1-\bar{Y}}T_{r_{\alpha\alpha}}Z - \frac{\gamma}{g}T_{1-\bar{Y}}T_{r_\alpha}R -\frac{\gamma}{2g}T_{r}T_{|1-Y|^2}R_\alpha +G,\\
 i\frac{\gamma}{2}T_{1-\bar{Y}}\partial_\alpha T_w X=& i\frac{\gamma}{2}T_{w_\alpha}X + i\frac{\gamma}{2}T_{w}T_{|1-Y|^2}\W +G.
\end{align*}
The other $\gamma$  contributions  can be absorbed into $G$.

As for the contributions of corrections with $\gamma^2$ coefficient, we compute
\begin{align*}
    i\frac{\gamma^2}{2g}T_{D_t}T_{w_\alpha}U =& i\frac{\gamma^2}{2g}T_{\partial_\alpha T_{D_t}w}U +i\frac{\gamma^2}{2g}T_{w_\alpha}T_{D_t}U +G\\
    =& -i\frac{\gamma^2}{2g}T_{T_{1-\bar{Y}}r_{\alpha\alpha}}U - i\frac{\gamma^2}{2g}T_{w_\alpha}T_{1-\bar{Y}}Z +G, \\
    i\frac{\gamma^2}{2g}T_{D_t}T_{\partial^{-1}_\alpha w}X_\alpha =& i\frac{\gamma^2}{2g}T_{T_{D_t}\partial^{-1}_\alpha w}X_\alpha + i\frac{\gamma^2}{2g}T_{\partial^{-1}_\alpha w}T_{D_t}X_\alpha +G\\
    =& -i\frac{\gamma^2}{2g}T_{T_{1-\bar{Y}}r}X_\alpha- i\frac{\gamma^2}{2g}T_{\partial^{-1}_\alpha w}T_{1-\bar{Y}}R_\alpha +G,\\
    i\frac{\gamma^2}{g}T_{D_t}T_w X =& i\frac{\gamma^2}{g}T_{D_t w} X+ i\frac{\gamma^2}{g}T_{w} T_{D_t}X + G\\
    =& -i\frac{\gamma^2}{g}T_{T_{1-\bar{Y}}r_\alpha } X- i\frac{\gamma^2}{g}T_{w} T_{1-\bar{Y}}R + G, \\
  -\frac{\gamma^2}{4g^2}T_{D_t}T_{r_\alpha}Z =&  -\frac{\gamma^2}{4g^2}T_{\partial_\alpha T_{D_t}r}Z-\frac{\gamma^2}{4g^2}T_{r_\alpha}T_{D_t}Z +G\\
    =& -\frac{\gamma^2}{4g^2}T_{-i\gamma r_\alpha +iT_{g+\ua} (x_\alpha -T_w Y_\alpha)}Z -\frac{\gamma^2}{4g^2}T_{r_\alpha}(igX-i\gamma Z)+G\\
    =& i\frac{\gamma^3}{2g^2}T_{r_\alpha}Z -i\frac{\gamma^2}{4g}T_{w_\alpha}Z -i\frac{\gamma^2}{4g}T_{r_\alpha}X+G,\\
 -\frac{\gamma^2}{4g^2}T_{D_t}T_{r}Z_\alpha =& -\frac{\gamma^2}{4g^2}T_{T_{D_t} r}Z_\alpha  -\frac{\gamma^2}{4g^2}T_{r}T_{D_t}Z_\alpha +G\\
=& -\frac{\gamma^2}{4g^2}T_{-i\gamma r +iT_{g+\ua} (x -T_w Y)}R-\frac{\gamma^2}{4g^2}T_{r}(igT_{1-Y}\W -i\gamma R)+G\\
=& i\frac{\gamma^3}{2g^2}T_r R -i\frac{\gamma^2}{4g}T_w R -i\frac{\gamma^2}{4g}T_r T_{1-Y}\W +G.
\end{align*}
Here we use \eqref{ZAlphaR} to change $Z_\alpha$ into $R$.
\begin{align*}
i\frac{\gamma^2}{2g}T_{1-\bar{Y}}\partial_\alpha T_{r_\alpha}U =& i\frac{\gamma^2}{2g}T_{T_{1-\bar{Y}}r_{\alpha\alpha}}U +i\frac{\gamma^2}{2g}T_{r_\alpha} X +G, \\
i\frac{\gamma^2}{2g}T_{1-\bar{Y}}\partial_\alpha T_{\partial_\alpha^{-1} w}R =& i\frac{\gamma^2}{2g}T_{1-\bar{Y}} T_{w}R + i\frac{\gamma^2}{2g} T_{\partial_\alpha^{-1} w}T_{1-\bar{Y}}R_\alpha + G,\\
 i\frac{\gamma^2}{4g}T_{1-\bar{Y}}\partial_\alpha T_{w}Z =& i\frac{\gamma^2}{4g} T_{w_\alpha} T_{1-\bar{Y}}Z + i\frac{\gamma^2}{4g} T_{w} R +G, \\
 i\frac{\gamma^2}{4g}T_{1-\bar{Y}}\partial_\alpha T_{r}X =& i\frac{\gamma^2}{4g}T_{1-\bar{Y}} T_{r_\alpha}X + i\frac{\gamma^2}{4g} T_{r}T_{|1-Y|^2}\W +G.
\end{align*}
The other $\gamma^2$ terms have ``over differentiated'' lowest frequency variables, and may be absorbed into $G$.

For the contributions of corrections with $\gamma^3$ coefficient,
\begin{align*}
i\frac{\gamma^3}{4g^2}T_{D_t}T_r X =&   i\frac{\gamma^3}{4g^2}T_{T_{D_t}r} X + i\frac{\gamma^3}{4g^2}T_r T_{D_t}X + G \\
=& i\frac{\gamma^3}{4g^2}T_{-i\gamma r +iT_{g+\ua} (x -T_w Y)}X - i\frac{\gamma^3}{4g^2}T_rT_{1-\bar{Y}}R +G \\
= & \frac{\gamma^4}{4g^2}T_r X - \frac{\gamma^3}{4g}T_w X- i\frac{\gamma^3}{4g^2}T_rT_{1-\bar{Y}}R +G,\\
i\frac{\gamma^3}{4g^2}T_{D_t}T_w Z =&   i\frac{\gamma^3}{4g^2}T_{D_t w} Z + i\frac{\gamma^3}{4g^2}T_w T_{D_t}Z + G \\
=& -i\frac{\gamma^3}{4g^2}T_{T_{1-\bar{Y}}(r_\alpha + T_w R_\alpha )} Z + i\frac{\gamma^3}{4g^2}T_w (igX-i\gamma Z
) + G\\
=& -i\frac{\gamma^3}{4g^2} T_{r_\alpha}Z -\frac{\gamma^3}{4g}T_w X +\frac{\gamma^4}{4g^2}T_w Z +G,\\
i\frac{\gamma^3}{4g^2}T_{D_t}T_{r_\alpha} U =&   i\frac{\gamma^3}{4g^2}T_{T_{D_t}r_\alpha} U + i\frac{\gamma^3}{4g^2}T_{r_\alpha} T_{D_t}U + G\\
=& i\frac{\gamma^3}{4g^2}T_{-i\gamma r_\alpha +iT_{g+\ua} (x_\alpha -T_w Y_\alpha)}U - i\frac{\gamma^3}{4g^2}T_{r_\alpha}T_{1-\bar{Y}}Z \\
=& \frac{\gamma^4}{4g^2}T_{r_\alpha}U -\frac{\gamma^3}{4g}T_{w_\alpha}U - i\frac{\gamma^3}{4g^2}T_{r_\alpha}T_{1-\bar{Y}}Z,\\
i\frac{\gamma^3}{4g^2}T_{D_t}T_{\partial_\alpha^{-1}w} R =&   i\frac{\gamma^3}{4g^2}T_{T_{D_t}\partial_\alpha^{-1}w} R + i\frac{\gamma^3}{4g^2}T_{\partial_\alpha^{-1}w} T_{D_t}R + G\\
=& -i\frac{\gamma^3}{4g^2}T_{T_{1-\bar{Y}}(r+ T_w R)} R + i\frac{\gamma^3}{4g^2}T_{\partial_\alpha^{-1}w}(iT_{g+\ua}Y-i\gamma R)  +G\\
=& -i\frac{\gamma^3}{4g^2}T_{T_{1-\bar{Y}}r} R-\frac{\gamma^3}{4g}T_{\partial_\alpha^{-1}w}Y +\frac{\gamma^4}{4g^2}T_{\partial_\alpha^{-1}w}R +G, \\
\frac{\gamma^3}{4g}T_{1-\bar{Y}}\partial_\alpha T_w U =& \frac{\gamma^3}{4g}T_{1-\bar{Y}} T_{w_\alpha} U + \frac{\gamma^3}{4g}T_{1-\bar{Y}} T_w X +G.
\end{align*}
Here we changed $U_\alpha$ to $X$ according to \eqref{UAlphaX}.
The other $\gamma^3$ terms have ``over differentiated'' lowest frequency variables, and we put them into $G$.

For the contributions of corrections with $\gamma^4$ coefficient,
\begin{align*}
\frac{\gamma^4}{4g^2}T_{D_t}T_w U =& \frac{\gamma^4}{4g^2}T_{D_t w} U+ \frac{\gamma^4}{4g^2}T_w T_{D_t}U+G\\
&= -\frac{\gamma^4}{4g^2}T_{ T_{1-\bar{Y}}(r_\alpha + T_w R_\alpha )} U- \frac{\gamma^4}{4g^2}T_w T_{1-\bar{Y}}Z+G\\
&= -\frac{\gamma^4}{4g^2}T_{ T_{1-\bar{Y}}r_\alpha } U- \frac{\gamma^4}{4g^2}T_w T_{1-\bar{Y}}Z+G,\\
\frac{\gamma^4}{4g^2}T_{D_t}T_{\partial_\alpha^{-1}w} X &= \frac{\gamma^4}{4g^2}T_{T_{D_t}\partial_\alpha^{-1}w} X + \frac{\gamma^4}{4g^2}T_{\partial_\alpha^{-1}w} T_{D_t}X+G\\
&= -\frac{\gamma^4}{4g^2}T_{T_{1-\bar{Y}}r} X - \frac{\gamma^4}{4g^2}T_{\partial_\alpha^{-1}w} T_{1-\bar{Y}}R+G.
\end{align*}
The rest of correction terms with  $\gamma^4$ or $\gamma^5$ coefficient are perturbative.

Summing all of above contributions, we obtain that the total correction is
\begin{equation*}
  \nP T_{w_\alpha}\ub  + T_w \nP\left((1-\bar{Y})R_\alpha - i\frac{\gamma}{2}\W\right) +G.
\end{equation*}

$(2)$  Continuing with the contributions for the frequency balanced corrections with holomorphic variables in the first equation.
The computation is similar to case $(1)$ in the sense that we simply replace the unbalanced par-Leibniz error estimate by the corresponding balanced para-Leibniz estimate in Lemma \ref{t:Leibniz}.
As a result, we get the total correction is given by
\begin{equation*}
    \Pi\left(w_\alpha, \nP\ub\right)+ \Pi\left(w, (1-\bar{Y})R_\alpha -i\frac{\gamma}{2} \W\right)+G.
\end{equation*}

$(3)$ Continuing with the frequency balanced corrections with one anti-holomorphic variable.
The computation is similar to case $(2)$.
\begin{align*}
-T_{D_t}\Pi(w_\alpha , \bar{X}) &= -\Pi(\partial_\alpha T_{D_t}w, \bar{X} )-\Pi(w_\alpha, T_{D_t}\bar{X}) +G\\
&=\Pi(r_{\alpha \alpha}, \bar{X})+ \Pi(w_\alpha ,T_{1-Y} \bar{R})+G,\\
-T_{1-\bar{Y}}\partial_\alpha \Pi(r_\alpha, \bar{X}) &= -T_{1-\bar{Y}}\Pi(r_{\alpha \alpha}, \bar{X}) -\Pi(r_{\alpha}, T_{1-\bar{Y}}\bar{X}_\alpha)+G\\
&= -T_{1-\bar{Y}}\Pi(r_{\alpha \alpha}, \bar{X}) -\Pi(r_{\alpha}, \bar{Y})+G.
\end{align*}
For the last term we use \eqref{XAlphaY} to replace $T_{1-\bar{Y}}\bar{X}_\alpha$ by $\bar{Y}$.
Then we have
\begin{align*}
-\frac{\gamma}{2g}T_{D_t}\Pi(w_\alpha, \bar{Z}) &= -\frac{\gamma}{2g}\Pi(\partial_\alpha T_{D_t}w, \bar{Z}) -\frac{\gamma}{2g}\Pi(w_\alpha, T_{D_t}\bar{Z})+ G\\
&= \frac{\gamma}{2g}\Pi(r_{\alpha \alpha} + T_w R_{\alpha\alpha}, \bar{Z}) -\frac{\gamma}{2g}\Pi(w_\alpha, -ig\bar{X}+i\gamma \bar{Z})+ G\\
&= \frac{\gamma}{2g}\Pi(r_{\alpha \alpha} , \bar{Z}) +i\frac{\gamma}{2}\Pi(w_\alpha, \bar{X})-i\frac{\gamma^2}{2g}\Pi(w_\alpha, \bar{Z})+ G,\\
-\frac{\gamma}{2g}T_{D_t}\Pi(r_\alpha, \bar{X}) &= -\frac{\gamma}{2g}\Pi(\partial_\alpha T_{D_t}r, \bar{X}) -\frac{\gamma}{2g}\Pi(r_\alpha, T_{D_t} \bar{X}) + G\\
&= -\frac{\gamma}{2g}\Pi(-i\gamma r_\alpha +iT_{g+\ua} (x_\alpha -T_w Y_\alpha), \bar{X}) +\frac{\gamma}{2g}\Pi(r_\alpha, T_{1-Y}\bar{R}) + G\\
&= i\frac{\gamma^2}{2g}\Pi(r_\alpha,\bar{X})-i\frac{\gamma}{2}\Pi(w_\alpha, \bar{X})+\frac{\gamma}{2g}\Pi(r_\alpha, T_{1-Y}\bar{R}) + G,\\
-\frac{\gamma}{2g}T_{1-\bar{Y}}\partial_\alpha \Pi(r_\alpha, \bar{Z})&= -\frac{\gamma}{2g}T_{1-\bar{Y}}\Pi(r_{\alpha \alpha}, \bar{Z})-\frac{\gamma}{2g}\Pi(r_\alpha, T_{1-\bar{Y}}\bar{R})+G,\\
i\frac{\gamma}{2}T_{1-\bar{Y}}\partial_\alpha \Pi(w, \bar{X})&= i\frac{\gamma}{2}\Pi(w_\alpha, T_{1-Y}\bar{X})+i\frac{\gamma}{2}\Pi(w, T_{1-\bar{Y}}\bar{X}_\alpha) +G.
\end{align*}
For the first term of the last contribution, and the second term of the first contribution, we can combine them to get
\begin{equation*}
 \Pi\left(w_\alpha, T_{1-Y}\bar{R}+i\frac{\gamma}{2}T_{1-Y}\bar{X}\right) = \Pi\left(w_\alpha, \bar{\nP}\left((1-Y)\bar{R}+ i\frac{\gamma}{2}(1-Y)\bar{W}\right)\right) +G = \Pi(w_\alpha, \bar{\nP}\ub) +G,
\end{equation*}
which is the other part of the desired balanced truncation $\Pi(w_\alpha, \ub)$ in $\nP\underline{\mathcal{G}}_1$.

For the contributions of $\gamma^2$ term, they are
\begin{align*}
-i\frac{\gamma^2}{2g}T_{D_t}\Pi(w_\alpha, \bar{U}) & = -i\frac{\gamma^2}{2g}\Pi( \partial_\alpha T_{D_t}w, \bar{U}) -i\frac{\gamma^2}{2g}\Pi(w_\alpha, T_{D_t}\bar{U}) +G \\
 & = i\frac{\gamma^2}{2g}\Pi( r_{\alpha \alpha}+ T_w R_{\alpha\alpha}, \bar{U}) +i\frac{\gamma^2}{2g}\Pi(w_\alpha, T_{1-Y}\bar{Z}) +G \\
 & = i\frac{\gamma^2}{2g}\Pi( r_{\alpha \alpha} , \bar{U}) +i\frac{\gamma^2}{2g}\Pi(w_\alpha, T_{1-Y}\bar{Z}) +G, \\
i\frac{\gamma^2}{4g}T_{D_t}\Pi(w, \bar{X}) &= i\frac{\gamma^2}{4g}\Pi(T_{D_t}w, \bar{X}) +i\frac{\gamma^2}{4g}\Pi(w, T_{D_t}\bar{X}) +G \\
&= -i\frac{\gamma^2}{4g}\Pi(r_\alpha + T_w R_{\alpha}, \bar{X}) -i\frac{\gamma^2}{4g}\Pi(w, T_{1-Y}\bar{R}) +G \\
&= -i\frac{\gamma^2}{4g}\Pi(r_\alpha, \bar{X}) -i\frac{\gamma^2}{4g}\Pi(w, T_{1-Y}\bar{R}) +G, \\
-\frac{\gamma^2}{4g^2}T_{D_t}\Pi(r_\alpha, \bar{Z})& = -\frac{\gamma^2}{4g^2}\Pi(\partial_\alpha T_{D_t}r, \bar{Z}) -\frac{\gamma^2}{4g^2}\Pi(r_\alpha, T_{D_t}\bar{Z}) +G \\
& = -\frac{\gamma^2}{4g^2}\Pi(-i\gamma r_\alpha +iT_{g+\ua} (x_\alpha -T_w Y_\alpha), \bar{Z}) -\frac{\gamma^2}{4g^2}\Pi(r_\alpha, -ig\bar{X}+i\gamma \bar{Z}) +G\\
& =-i\frac{\gamma^2}{4g}\Pi(w_\alpha, \bar{Z}) +i\frac{\gamma^2}{4g}\Pi(r_\alpha, \bar{X})+G,\\
-i\frac{\gamma^2}{2g}T_{1-\bar{Y}}\partial_\alpha \Pi(r_\alpha, \bar{U} ) &= -i\frac{\gamma^2}{2g}T_{1-\bar{Y}}\Pi(r_{\alpha \alpha}, \bar{U}) -i\frac{\gamma^2}{2g} \Pi(r_\alpha, T_{1-\bar{Y}}\bar{X} ) +G,\\
i\frac{\gamma^2}{4g}T_{1-\bar{Y}}\partial_\alpha\Pi(w, \bar{Z} ) &= i\frac{\gamma^2}{4g}T_{1-\bar{Y}}\Pi(w_\alpha, \bar{Z} ) + i\frac{\gamma^2}{4g}\Pi(w, T_{1-\bar{Y}}\bar{R} ) +G.
\end{align*}
The rest of the contributions are given by
\begin{align*}
i\frac{\gamma^3}{4g^2}T_{D_t}\Pi(w,\bar{Z}) &= i\frac{\gamma^3}{4g^2}\Pi( T_{D_t}w,\bar{Z}) + i\frac{\gamma^3}{4g^2}\Pi(w, T_{D_t}\bar{Z}) +G \\
&= -i\frac{\gamma^3}{4g^2}\Pi( r_\alpha + T_w R_{\alpha},\bar{Z}) + i\frac{\gamma^3}{4g^2}\Pi(w, -ig\bar{X}+i\gamma \bar{Z}) +G \\
& = -i\frac{\gamma^3}{4g^2}\Pi(r_\alpha, \bar{Z}) + \frac{\gamma^3}{4g}\Pi(w, \bar{X})-\frac{\gamma^4}{4g^2}\Pi(w, \bar{Z})+G,\\
-i\frac{\gamma^3}{4g^2}T_{D_t}\Pi(r_\alpha ,\bar{U}) &= -i\frac{\gamma^3}{4g^2}\Pi( \partial_\alpha T_{D_t}r,\bar{U}) - i\frac{\gamma^3}{4g^2}\Pi(r_\alpha, T_{D_t}\bar{U}) +G\\
&= -i\frac{\gamma^3}{4g^2}\Pi( -i\gamma r_\alpha +iT_{g+\ua} (x_\alpha -T_w Y_\alpha) ,\bar{U}) + i\frac{\gamma^3}{4g^2}\Pi(r_\alpha, T_{1-Y}\bar{Z}) +G\\
&= -\frac{\gamma^4}{4g^2}\Pi(r_\alpha, \bar{U})+ \frac{\gamma^3}{4g}\Pi(w_\alpha, \bar{U})+ i\frac{\gamma^3}{4g^2}\Pi(r_\alpha, T_{1-Y}\bar{Z}) +G,\\
-\frac{\gamma^4}{4g^2}T_{D_t}\Pi(w, \bar{U}) &= -\frac{\gamma^4}{4g^2}\Pi(T_{D_t} w, \bar{U}) -\frac{\gamma^4}{4g^2}\Pi(w, T_{D_t}\bar{U}) +G \\
&= \frac{\gamma^4}{4g^2}\Pi(r_\alpha, \bar{U}) + \frac{\gamma^4}{4g^2}\Pi(w, T_{1-Y}\bar{Z}) +G, \\
-\frac{\gamma^3}{4g}T_{1-\bar{Y}}\partial_\alpha\Pi(w, \bar{U})& =- \frac{\gamma^3}{4g}T_{1-\bar{Y}}\Pi(w_\alpha, \bar{U}) - \frac{\gamma^3}{4g}\Pi(w, T_{1-\bar{Y}}\bar{X}) +G.
\end{align*}
The other contributions have full or over differentiated variables at low frequency, and can be put into $G$.
Collecting all the corrections above, the total  frequency balanced corrections with one anti-holomorphic variable are given by
\begin{equation*}
    -\Pi(r_\alpha, \bar{Y})+ \Pi(w_\alpha, \bar{\nP}\ub)+ \Pi\left(w, i\frac{\gamma}{2}\bar{W}\right) +G.
\end{equation*}
Adding the contributions of corrections in cases $(1)$, $(2)$, $(3)$, the sum is exactly $-\nP\underline{\mathcal{G}}_1 +G$.

$(4)$  Next, we compute the second equation in \eqref{NormalFormGKOne}.
We begin with the high-low corrections.
Consider first  the contributions of non-vorticity terms,
\begin{align*}
-T_{D_t}T_{r_\alpha}X =& -T_{\partial_\alpha T_{D_t}r }X -T_{r_\alpha}T_{D_t}X +K\\
=&iT_{\gamma r_\alpha -T_{g+\ua} (x_\alpha -T_w Y_\alpha)}X +T_{r_\alpha}T_{1-\bar{Y}}R +K\\
=& i\gamma T_{r_\alpha}X -iT_{g+\ua}T_{((1-Y)w)_\alpha}X+T_{r_\alpha}T_{1-\bar{Y}}R +K,\\
iT_{1-Y}T_{g+\ua}\partial_\alpha T_w X =& iT_{1-Y}T_{g+\ua}T_{((1-Y)w)_\alpha}W + iT_{1-Y}T_{g+\ua}T_w T_{1-Y}\W +K \\
=& iT_{g+\ua}T_{((1-Y)w)_\alpha}X +K.
\end{align*}
Then for the contributions of the $\gamma$ terms, they are $-i\gamma T_{r_\alpha}X$ and
\begin{align*}
-\frac{\gamma}{2g}T_{D_t}T_{r_\alpha}Z&= -\frac{\gamma}{2g}T_{\partial_\alpha T_{D_t}r}Z -\frac{\gamma}{2g}T_{r_\alpha}T_{D_t}Z +K\\
&= i\frac{\gamma}{2g}T_{\gamma r_\alpha -T_{g+\ua} (x_\alpha -T_w Y_\alpha)}Z -\frac{\gamma}{2g}T_{r_\alpha}(igX-i\gamma Z) +K\\
&= i\frac{\gamma^2}{g}T_{r_\alpha}Z -i\frac{\gamma}{2}T_{w_\alpha}Z -i\frac{\gamma}{2}T_{r_\alpha}X+K,\\
-\frac{\gamma}{2g}T_{D_t}T_r R &= -\frac{\gamma}{2g}T_{D_t r}R -\frac{\gamma}{2g}T_r T_{D_t} R +K\\
&= i\frac{\gamma}{2g}T_{\gamma r -T_{g+\ua} (x -T_w Y)}R -\frac{\gamma}{2g}T_r (igT_{1-Y}\W -i\gamma R) +K\\
&=i\frac{\gamma^2}{g}T_r R-i\frac{\gamma}{2}T_w R -i\frac{\gamma}{2}T_r T_{1-Y}\W +K,\\
i\frac{\gamma}{2}T_{D_t}T_w X & = i\frac{\gamma}{2}T_{D_t w}X+ i\frac{\gamma}{2}T_w T_{D_t}X +K\\
&= -i\frac{\gamma}{2}T_{T_{1-\bar{Y}}(r_\alpha + T_w R_\alpha )}X- i\frac{\gamma}{2}T_w T_{1-\bar{Y}}R +K\\
&= -i\frac{\gamma}{2}T_{T_{1-\bar{Y}}r_\alpha }X- i\frac{\gamma}{2}T_w T_{1-\bar{Y}}R +K,\\
i\frac{\gamma}{2g}T_{1-Y}T_{g+\ua}T_{w_\alpha}Z = i\frac{\gamma}{2}T_{w_\alpha}Z& +K,\quad
i\frac{\gamma}{2g}T_{1-Y}T_{g+\ua}T_{r}X_\alpha = i\frac{\gamma}{2}T_r T_{1-Y}\W +K,\\
i\frac{\gamma}{2g}T_{1-Y}T_{g+\ua}T_{r_\alpha}X = i\frac{\gamma}{2}T_{r_\alpha} X& +K,\quad
i\frac{\gamma}{2g}T_{1-Y}T_{g+\ua}T_{w}R = i\frac{\gamma}{2}T_{w} R +K.
\end{align*}
Next, for the contributions of the $\gamma^2$ terms, they are given by
\begin{align*}
-i\frac{\gamma^2}{2g} T_{r_\alpha}Z &-i\frac{\gamma^2}{2g} T_{r}R -\frac{\gamma^2}{2}T_w X, \\
i\frac{\gamma^2}{2g}T_{D_t}T_{r_\alpha}U &=  i\frac{\gamma^2}{2g}T_{T_{D_t} r_\alpha}U + i\frac{\gamma^2}{2g}T_{r_\alpha}T_{D_t}U +K\\
&= i\frac{\gamma^2}{2g}T_{-i\gamma r_\alpha +iT_{g+\ua} (x_\alpha -T_w Y_\alpha)}U - i\frac{\gamma^2}{2g}T_{r_\alpha}T_{1-\bar{Y}}Z +K\\
&= \frac{\gamma^3}{2g}T_{r_\alpha}U -\frac{\gamma^2}{2}T_{T_{1-Y}w_\alpha} U - i\frac{\gamma^2}{2g}T_{r_\alpha}T_{1-\bar{Y}}Z +K,\\
i\frac{\gamma^2}{2g}T_{D_t}T_{\partial_\alpha^{-1}w}R &= i\frac{\gamma^2}{2g}T_{T_{D_t}\partial_\alpha^{-1}w}R+i\frac{\gamma^2}{2g}T_{\partial_\alpha^{-1}w}T_{D_t}R+K\\
&= -i\frac{\gamma^2}{2g}T_{T_{1-\bar{Y}}(r+ T_w R)}R+i\frac{\gamma^2}{2g}T_{\partial_\alpha^{-1}w}(iT_{g+\ua}Y-i\gamma R)+K\\
& = -i\frac{\gamma^2}{2g}T_{T_{1-\bar{Y}}r}R -\frac{\gamma^2}{2g}T_{\partial_\alpha^{-1}w}T_{g+\ua}Y+\frac{\gamma^3}{2g}T_{\partial_\alpha^{-1}w}R +K,\\
i\frac{\gamma^2}{4g}T_{D_t}T_w Z & = i\frac{\gamma^2}{4g}T_{D_t w} Z + i\frac{\gamma^2}{4g}T_w T_{D_t}Z +K\\
& = -i\frac{\gamma^2}{4g}T_{T_{1-\bar{Y}}(r_\alpha + T_w R_\alpha )} Z + i\frac{\gamma^2}{4g}T_w (igX-i\gamma Z) +K\\
& = -i\frac{\gamma^2}{4g}T_{T_{1-\bar{Y}}r_\alpha} Z- \frac{\gamma^2}{4}T_w X +\frac{\gamma^3}{4g}T_w Z +K,\\
i\frac{\gamma^2}{4g}T_{D_t}T_r X &= i\frac{\gamma^2}{4g}T_{D_t r} X + i\frac{\gamma^2}{4g}T_r T_{D_t}X +K \\
&= i\frac{\gamma^2}{4g}T_{-i\gamma r +iT_{g+\ua} (x -T_w Y)} X - i\frac{\gamma^2}{4g}T_r T_{1-\bar{Y}}R +K\\
&= \frac{\gamma^3}{4g}T_r X -\frac{\gamma^2}{4}T_w X - i\frac{\gamma^2}{4g}T_r T_{1-\bar{Y}}R +K,\\
\frac{\gamma^2}{2g}T_{1-Y}T_{g+\ua}T_{w_\alpha}U=& \frac{\gamma^2}{2}T_{1-Y}T_{w_\alpha}U+K, \quad \frac{\gamma^2}{2g}T_{1-Y}T_{g+\ua}T_{\partial_\alpha^{-1}w}X_\alpha =  \frac{\gamma^2}{2g}T_{g+\ua}T_{\partial_\alpha^{-1}w}Y +K, \\
\frac{\gamma^2}{g}T_{1-Y}T_{g+\ua}T_{w}X=& \gamma^2T_{1-Y}T_{w}X+K, \quad i\frac{\gamma^2}{4g^2}T_{1-Y}T_{g+\ua}T_{r_\alpha}Z = i\frac{\gamma^2}{4g}T_{1-Y}T_{r_\alpha}Z+K,\\
i\frac{\gamma^2}{4g^2}T_{1-Y}T_{g+\ua}T_{r}Z_\alpha =& i\frac{\gamma^2}{4g}T_{1-Y}T_{r}R+K.
\end{align*}
Continuing with the contributions of the $\gamma^3$ corrections, the first few terms are
\begin{equation*}
    -\frac{\gamma^3}{2g}T_{r_\alpha}U -\frac{\gamma^3}{2g}T_{\partial_\alpha^{-1}w}R-\frac{\gamma^3}{4g}T_w Z -\frac{\gamma^3}{4g}T_r X.
\end{equation*}
The contributions of other $\gamma^3$ terms are given by
\begin{align*}
\frac{\gamma^3}{4g}T_{D_t}T_w U &= \frac{\gamma^3}{4g}T_{D_t w} U + \frac{\gamma^3}{4g}T_w T_{D_t} U +K\\
& =  -\frac{\gamma^3}{4g}T_{T_{1-\bar{Y}}(r_\alpha + T_w R_\alpha )} U - \frac{\gamma^3}{4g}T_w T_{1-\bar{Y}}Z +K\\
& =  -\frac{\gamma^3}{4g}T_{T_{1-\bar{Y}}r_\alpha} U - \frac{\gamma^3}{4g}T_w T_{1-\bar{Y}}Z +K,\\
\frac{\gamma^3}{4g}T_{D_t}T_{\partial_\alpha^{-1}w}X &= \frac{\gamma^3}{4g}T_{T_{D_t}\partial_\alpha^{-1}w}X + \frac{\gamma^3}{4g}T_{\partial_\alpha^{-1}w}T_{D_t}X +K\\
&= -\frac{\gamma^3}{4g}T_{T_{1-\bar{Y}}(r+ T_w R)}X - \frac{\gamma^3}{4g}T_{\partial_\alpha^{-1}w}T_{1-\bar{Y}}R +K\\
&= -\frac{\gamma^3}{4g}T_{T_{1-\bar{Y}}r}X - \frac{\gamma^3}{4g}T_{\partial_\alpha^{-1}w}T_{1-\bar{Y}}R +K,\\
\frac{\gamma^3}{4g^2}T_{1-Y}T_{g+\ua}T_r X = & \frac{\gamma^3}{4g}T_{1-Y}T_r X +K, \quad \frac{\gamma^3}{4g^2}T_{1-Y}T_{g+\ua}T_{w}Z = \frac{\gamma^3}{4g}T_{1-Y}T_{w}Z+ K, \\
\frac{\gamma^3}{4g^2}T_{1-Y}T_{g+\ua}T_{r_\alpha} U = & \frac{\gamma^3}{4g}T_{1-Y}T_{r_\alpha} U +K, \quad \frac{\gamma^3}{4g^2}T_{1-Y}T_{g+\ua}T_{\partial_\alpha^{-1}w}R = \frac{\gamma^3}{4g}T_{1-Y}T_{\partial_\alpha^{-1}w}R+ K.
\end{align*}
For the contributions of corrections with $\gamma^4$ coefficient, they are given by
\begin{align*}
    i\frac{\gamma^4}{4g}T_w U + i\frac{\gamma^4}{4g}T_{\partial_\alpha^{-1}w} X,& \\
    -i\frac{\gamma^4}{4g^2}T_{1-Y}T_{g+\ua}(T_w U +\partial_\alpha^{-1}X) &= -i\frac{\gamma^4}{4g}T_w U -i\frac{\gamma^4}{4g}T_{\partial_\alpha^{-1}w} X+K.
\end{align*}
Adding all the contributions above, we get the total correction
\begin{equation*}
    \nP T_{r_\alpha}\ub-  i\frac{\gamma}{2}T_{1-Y}T_w R + T_{(g+\ua)w}Y+K.
\end{equation*}

$(5)$  Continuing with the contributions for the frequency balanced corrections with holomorphic variables for the second equation of \eqref{NormalFormGKOne}.
The computation is similar to case $(4)$ in the sense that we simply replace the unbalanced par-Leibniz error estimates by the corresponding balanced para-Leibniz estimates in Lemma \ref{t:Leibniz}.
Consequently, the total correction is given by
\begin{equation*}
    \Pi\left(r_\alpha, \nP(R(1-\bar{Y}))-i\frac{\gamma}{2}\nP(W(1-\bar{Y}))\right)-i\frac{\gamma}{2}T_{1-Y} \Pi(w, R) + \Pi((g+\ua)w, Y)+K.
\end{equation*}

$(6)$ Finally, we  compute the frequency balanced corrections with one anti-holomorphic variable for the second equation of \eqref{NormalFormGKOne}.
\begin{align*}
-T_{D_t}\Pi(r_\alpha, \bar{X}) &= -\Pi(\partial_\alpha T_{D_t}r, \bar{X})-\Pi(r_\alpha, T_{D_t}\bar{X}) +K\\
&= -\Pi(-i\gamma r_\alpha +iT_{g+\ua} (x_\alpha -T_w Y_\alpha), \bar{X})+\Pi(r_\alpha, T_{1-Y}\bar{R}) +K\\
&= i\gamma \Pi(r_\alpha , \bar{X})- ig\Pi(T_{1-Y}w_\alpha, \bar{X})+\Pi(r_\alpha, T_{1-Y}\bar{R}) +K,\\
iT_{1-Y}T_{g+\ua}\Pi(w_\alpha, \bar{X})&= ig \Pi(T_{1-Y}w_\alpha, \bar{X}) +K,\quad -i\gamma\Pi(r_\alpha, \bar{X}),\\
-\frac{\gamma}{2g}T_{D_t}\Pi(r_\alpha, \bar{Z})& = -\frac{\gamma}{2g}\Pi(\partial_\alpha T_{D_t} r, \bar{Z}) -\frac{\gamma}{2g}\Pi(r_\alpha, T_{D_t}\bar{Z}) +K\\
& = -\frac{\gamma}{2g}\Pi(-i\gamma r_\alpha +iT_{g+\ua} (x_\alpha -T_w Y_\alpha), \bar{Z}) -\frac{\gamma}{2g}\Pi(r_\alpha, -ig\bar{X}+i\gamma \bar{Z}) +K\\
& = -i\frac{\gamma}{2}\Pi(T_{1-Y}w_\alpha, \bar{Z})+ i\frac{\gamma}{2}\Pi(r_\alpha, \bar{X})+K,\\
i\frac{\gamma}{2}T_{D_t}\Pi(w, \bar{X}) &= i\frac{\gamma}{2}\Pi(T_{D_t}w, \bar{X}) + i\frac{\gamma}{2}\Pi(w, T_{D_t}\bar{X}) +K \\
&= -i\frac{\gamma}{2}\Pi(T_{1-\bar{Y}}(r_\alpha + T_w R_\alpha ), \bar{X}) - i\frac{\gamma}{2}\Pi(w, T_{1-Y}\bar{R}) +K\\
&= -i\frac{\gamma}{2}\Pi(T_{1-\bar{Y}}r_\alpha, \bar{X}) - i\frac{\gamma}{2}T_{1-Y}\Pi(w, \bar{R}) +K,\\
i\frac{\gamma}{2g}T_{1-Y}T_{g+\ua}\Pi(w_\alpha, \bar{Z}) & = i\frac{\gamma}{2}\Pi(T_{1-Y}w_\alpha, \bar{Z})+K,\\
i\frac{\gamma}{2g}T_{1-Y}T_{g+\ua}\Pi(r_\alpha, \bar{X})  &= i\frac{\gamma}{2}\Pi(r_\alpha, T_{1-Y}\bar{X})+K.
\end{align*}
The contributions of $\gamma^2$ correction terms are given by
\begin{align*}
-i\frac{\gamma^2}{2g}\Pi(r_\alpha, \bar{Z})-&i\frac{\gamma^2}{2}\Pi(w, \bar{X}),\\
-i\frac{\gamma^2}{2g}T_{D_t}\Pi(r_\alpha, \bar{U})&= -i\frac{\gamma^2}{2g}\Pi(\partial_\alpha T_{D_t} r, \bar{U}) -\frac{\gamma^2}{2g}\Pi(r_\alpha, T_{D_t}\bar{U})+K\\
&= -i\frac{\gamma^2}{2g}\Pi(-i\gamma r_\alpha +iT_{g+\ua} (x_\alpha -T_w Y_\alpha), \bar{U}) +i\frac{\gamma^2}{2g}\Pi(r_\alpha, T_{1-Y}\bar{Z})+K\\
&= -\frac{\gamma^3}{2g}\Pi(r_\alpha, \bar{U})+ \frac{\gamma^2}{2}\Pi(T_{1-Y}w_\alpha , \bar{U}) +i\frac{\gamma^2}{2g}\Pi(r_\alpha, T_{1-Y}\bar{Z})+K,\\
i\frac{\gamma^2}{4g}T_{D_t}\Pi(w, \bar{Z})&= i\frac{\gamma^2}{4g}\Pi(T_{D_t} w, \bar{Z}) + i\frac{\gamma^2}{4g}\Pi(w, T_{D_t}\bar{Z}) +K\\
&= -i\frac{\gamma^2}{4g}\Pi(T_{1-\bar{Y}}(r_\alpha + T_w R_\alpha ), \bar{Z}) + i\frac{\gamma^2}{4g}\Pi(w, -ig\bar{X}+i\gamma \bar{Z}) +K\\
&= -i\frac{\gamma^2}{4g}\Pi(T_{1-\bar{Y}}r_\alpha , \bar{Z}) +\frac{\gamma^2}{4}\Pi(w, \bar{X})-\frac{\gamma^3}{4g}\Pi(w,\bar{Z})+K,\\
-\frac{\gamma^2}{2g}T_{1-Y}T_{g+\ua}\Pi(w_\alpha, \bar{U}) & = -\frac{\gamma^2}{2}\Pi(T_{1-Y} w_\alpha, \bar{U}) +K,\\
\frac{\gamma^2}{4g}T_{1-Y}T_{g+\ua}\Pi(w,\bar{X}) & = \frac{\gamma^2}{4}\Pi(T_{1-Y}w,\bar{X})+K,\\
i\frac{\gamma^2}{4g^2}T_{1-Y}T_{g+\ua}\Pi(r_\alpha, \bar{Z}) &= i\frac{\gamma^2}{4g}\Pi(T_{1-Y}r_\alpha, \bar{Z})+K.
\end{align*}
Continuing with the contributions of $\gamma^3$ and $\gamma^4$ terms, they are
\begin{align*}
\frac{\gamma^3}{2g}\Pi(r_\alpha, \bar{U})-&\frac{\gamma^3}{4g}\Pi(w,\bar{Z}) -i\frac{\gamma^4}{4g}\Pi(w,\bar{U}),\\
-\frac{\gamma^3}{4g}T_{D_t}\Pi(w, \bar{U}) &= -\frac{\gamma^3}{4g}\Pi(T_{D_t}w, \bar{U})-\frac{\gamma^3}{4g}\Pi(w, T_{D_t}\bar{U})+K\\
&= \frac{\gamma^3}{4g}\Pi(T_{1-\bar{Y}}(r_\alpha + T_w R_\alpha ), \bar{U})+\frac{\gamma^3}{4g}\Pi(w, T_{1-Y}\bar{Z})+K\\
&= \frac{\gamma^3}{4g}\Pi(T_{1-\bar{Y}}r_\alpha , \bar{U})+\frac{\gamma^3}{4g}\Pi(w, T_{1-Y}\bar{Z})+K,\\
\frac{\gamma^3}{4g^2}T_{1-Y}T_{g+\ua}\Pi(w, \bar{Z})&= \frac{\gamma^3}{4g}\Pi(T_{1-Y}w, \bar{Z})+K, \\
-\frac{\gamma^3}{4g^2}T_{1-Y}T_{g+\ua}\Pi(r_\alpha, \bar{U})&= -\frac{\gamma^3}{4g}\Pi(T_{1-Y}r_\alpha, \bar{U})+K,\\
i\frac{\gamma^4}{4g^2}T_{1-Y}T_{g+\ua}\Pi(w, \bar{U})&= i\frac{\gamma^4}{4g}\Pi(T_{1-Y}w, \bar{U})+K.
\end{align*}
Summing all the contributions above and using the definitions of $\ua$ and $\ub$, we get that the total correction in this case is given by
\begin{equation*}
\Pi\left(r_\alpha, T_{1-Y}\bar{R}+i\frac{\gamma}{2}T_{1-Y}\bar{X}\right)-i\frac{\gamma}{2}T_{1-Y}\Pi(w, \bar{R})+K = \Pi(r_\alpha, \bar{\nP}\ub)-iT_{1-Y}\Pi(w, \bar{\nP}\ua)+K.
\end{equation*}

Collecting the contributions of corrections in cases $(4)$, $(5)$, $(6)$ together, the sum is exactly $-\nP\underline{\mathcal{K}}_1 +K$.
\end{proof}

\subsection{Normal form analysis for $(\underline{\mathcal{G}}_0, \underline{\mathcal{K}}_0)$} \label{s:GKZero}
In this subsection, we use the paradifferential normal form corrections to remove the main  part of the source terms $(\nP\underline{\mathcal{G}}_0, \nP\underline{\mathcal{K}}_0)$ in the linearized equations.

For simplicity regarding the balanced cases, we let $\Pi = \nP\Pi$ to include the Littlewood-Paley projection $\nP$.
We will  then prove the following result:
\begin{proposition}
Assume that $(w,r)$ solve the paradifferential equations \eqref{ParadifferentialLinearEqn}, then there exists a linear paradifferential normal form correction
\begin{equation*}
    (\tilde{w}, \tilde{r}) = NF(w,r)
\end{equation*}
that solves the paradifferential equations
\begin{equation}
\left\{
             \begin{array}{lr}
             T_{D_t}\tilde{w}+T_{1-\bar{Y}}\partial_\alpha \tilde{r} + T_{(1-\bar{Y})R_\alpha }\tilde{w} +\gamma T_{\Im \W}\tilde{w} = -\nP \underline{\mathcal{G}}_0(w,r) +G &\\
           T_{D_t}\tilde{r} +i\gamma \tilde{r} -iT_{1-Y}T_{g+\ua}\tilde{w}=-\nP \underline{\mathcal{K}}_0(w,r) +K,&  \label{NormalFormGKZero}
             \end{array}
\right.
\end{equation}
with the following properties:
\begin{enumerate}
\item Quadratic correction bound:
\begin{equation*}
    \|NF(w,r)\|_{\doth{\frac{1}{4}}} \lesssim_{\uA} \uAS\left(\| (w,r)\|_{\doth{\frac{1}{4}}}+ \gamma^2 \|(w,r) \|_{\doth{-\frac{3}{4}}}\right).
\end{equation*}
\item Secondary correction bound:
\begin{equation*}
\|NF(w,r)\|_{\doth{\frac{1}{4}}}\lesssim_\uA \Astar\left(\|(w,r)\|_{\mathcal{H}^0}+ \gamma^2 \|(w,r) \|_{\doth{-1}}\right).
\end{equation*}
\item Cubic error bound:
\begin{equation*}
    \|(G, K)\|_{\doth{\frac{1}{4}}} \lesssim_{\uAS} \underline{A}_{\frac{1}{4}}\uAStar\left(\| (w,r)\|_{\doth{\frac{1}{4}}}+ \gamma^2 \|(w,r) \|_{\doth{-\frac{3}{4}}}\right).
\end{equation*}
\end{enumerate}
\end{proposition}
The idea to construct the paradifferential corrections here is similar to that for $(\nP\underline{\mathcal{G}}_1, \nP\underline{\mathcal{K}}_1)$ corrections in Section \ref{s:GKOne}.
Inspired by the normal form transformation \eqref{e:NormalForm}, we define the  corrections $(\tilde{w}, \tilde{r})$ by
\begin{align*}
\tilde{w} =& -\nP(\bar{x}\W)-\frac{\gamma}{2g}\nP(\bar{r}\W)-\frac{\gamma}{2g}\nP(\bar{x}T_{1+\W}R)-i\frac{\gamma^2}{2g}\nP(\bar{u}\W)
+i\frac{\gamma^2}{4g}\nP(\bar{x}W)\\
&-\frac{\gamma^2}{4g^2}\nP(\bar{r}T_{1+\W}R)+i\frac{\gamma^3}{4g^2}\nP(\bar{r} W)-i\frac{\gamma^3}{4g^2}\nP(\bar{u}T_{1+\W}R)-\frac{\gamma^4}{4g^2}\nP(\bar{u}W),\\
\tilde{r} =&-\nP(\bar{x}T_{1+\W}R)-\frac{\gamma}{2g}\nP(\bar{r}T_{1+\W}R)+i\frac{\gamma}{2}\nP(\bar{x}W)-i\frac{\gamma^2}{2g}\nP(\bar{u}T_{1+\W}R)\\
&+i\frac{\gamma^2}{4g}\nP(\bar{r}W)-\frac{\gamma^3}{4g}\nP(\bar{u}W).
\end{align*}
We remark that this paradifferential correction  is roughly the anti-holomorphic part of the linearization of the normal form transformation \eqref{e:NormalForm} at the quadratic level.

These correction bounds are easily verified and are similar  to the estimates in Section \ref{s:GKOne} using Sobolev embeddings.
Lemma \ref{t:G0K0Estimate} shows that the main part of $(\nP\underline{\mathcal{G}}_0, \nP\underline{\mathcal{K}}_0)$ is $(\mathcal{G}_{0,0}, \mathcal{K}_{0,0})$.
Therefore, it remains to show that these corrections  eliminate  $(\mathcal{G}_{0,0}, \mathcal{K}_{0,0})$ modulo perturbative terms.
\begin{proof}
We follow the same strategy in Section \ref{s:GKOne}, and put all the balanced cubic and higher order terms into $(G,K)$.
We will consider first the contributions of corrections for the first equation of \eqref{NormalFormGKZero}, then for the second equation of \eqref{NormalFormGKZero}.
For each case, we again compute according to the power of the vorticity $\gamma$.
When computing the para-material derivatives of paraproducts, we apply the para-Leibniz rule Lemma \ref{t:Leibniz} to distribute para-material derivatives to each variable.
For the expressions of material derivatives and para-material derivatives of each variable, we use the leading terms plus errors measured in either $BMO$ or $L^4$ based spaces in corresponding Lemmas in Section \ref{s:TDtBound} and Section \ref{s:SourceBound}.

$(1)$ We start with the first equation of \eqref{NormalFormGKZero}.
We begin by computing the contributions of corrections of non-vorticity and $\gamma$ terms.
\begin{align*}
-T_{D_t}\nP(\bar{x}\W) &= -\nP(D_t\bar{x}\W) -\nP(\bar{x}T_{D_t}\W)+G\\
&= \nP(T_{1-Y}(T_{1-\bar{Y}}\bar{r}_\alpha + T_{\bar{x}} \bar{R}_\alpha)\W) +\nP(\bar{x}T_{1+\W}T_{1-\bar{Y}}R_\alpha)+G\\
 &= \nP(T_{1-\bar{Y}}\bar{r}_\alpha T_{1-Y}\W) +\nP(\bar{x}T_{1+\W}T_{1-\bar{Y}}R_\alpha)+G\\
&= \nP(T_{1-\bar{Y}}\bar{r}_\alpha T_{1+\W}Y) +T_{1-\bar{Y}}\nP(\bar{x}T_{1+\W}R_\alpha)+G.
\end{align*}
Here we use \eqref{ParaAssociateTwo} to move para-coefficients $T_{1-Y}$, $T_{1-\bar{Y}}$.
We also use \eqref{XAlphaY} to change $T_{1-Y}\W$ to $T_{1+\W}Y$.
\begin{align*}
-T_{1-\bar{Y}}\partial_\alpha \nP(\bar{x}T_{1+\W}R) =& - T_{1-\bar{Y}}\nP(\bar{x}_\alpha T_{1+\W}R)  - T_{1-\bar{Y}}\nP(\bar{x}T_{1+\W}R_\alpha) +G, \\
-\frac{\gamma}{2g}T_{D_t}\nP(\bar{r}\W) =& -\frac{\gamma}{2g}\nP(D_t\bar{r}\W) -\frac{\gamma}{2g}\nP(\bar{r}T_{D_t}\W)+G\\
 =& -\frac{\gamma}{2g}\nP((i\gamma  \bar{r} - iT_{g+\ua}(\bar{x} -T_{\bar{w}} \bar{Y}))\W) \\
&+\frac{\gamma}{2g}\nP(\bar{r}T_{1+\W}T_{1-\bar{Y}}R_\alpha )+G \\
=& -i\frac{\gamma^2}{2g}\nP(\bar{r} \W)+i\frac{\gamma}{2}\nP(\bar{x}\W) +\frac{\gamma}{2g}T_{1-\bar{Y}}\nP(\bar{r}T_{1+\W}R_\alpha)+G,   \\
-\frac{\gamma}{2g}T_{D_t}\nP(\bar{x}T_{1+\W}R) =& -\frac{\gamma}{2g}\nP(D_t\bar{x}T_{1+\W}R) -\frac{\gamma}{2g}\nP(\bar{x}T_{1+\W}T_{D_t}R) -\frac{\gamma}{2g}\nP(\bar{x}T_{D_t\W}R)+G\\
=&  \frac{\gamma}{2g}\nP(T_{1-Y}(T_{1-\bar{Y}}\bar{r}_\alpha + T_{\bar{x}} \bar{R}_\alpha)T_{1+\W}R) \\
&-i\frac{\gamma}{2g}\nP(\bar{x}T_{1+\W}(T_{g+\ua}Y-\gamma R))+ \frac{\gamma}{2g}\nP(\bar{x}T_{T_{1+\W}T_{1-\bar{Y}}R_\alpha}R)+G\\
=& \frac{\gamma}{2g}T_{1-\bar{Y}}\nP(\bar{r}_\alpha R) -i\frac{\gamma}{2}\nP(\bar{x}T_{1+\W}Y)+i\frac{\gamma^2}{2g}\nP(\bar{x}T_{1+\W}R)+G,\\
-\frac{\gamma}{2g}T_{1-\bar{Y}}\partial_\alpha \nP(\bar{r}T_{1+\W}R)=& -\frac{\gamma}{2g} T_{1-\bar{Y}}\nP(\bar{r}_\alpha R) -\frac{\gamma}{2g}T_{1-\bar{Y}} \nP(\bar{r}T_{1+\W}R_\alpha) +G,\\
i\frac{\gamma}{2}T_{1-\bar{Y}}\partial_\alpha \nP(\bar{x}W) =& i\frac{\gamma}{2}T_{1-\bar{Y}}\nP(\bar{x}_\alpha W) + i\frac{\gamma}{2} T_{1-\bar{Y}}\nP(\bar{x}\W) +G.
\end{align*}

Continuing with the contributions of corrections of $\gamma^2$ terms, we compute
\begin{align*}
-i\frac{\gamma^2}{2g}T_{D_t}\nP(\bar{u}\W)=& -i\frac{\gamma^2}{2g}\nP(T_{D_t}\bar{u}\W)-i\frac{\gamma^2}{2g}\nP(\bar{u}T_{D_t}\W)+G\\
=& i\frac{\gamma^2}{2g}\nP(T_{1-Y}(T_{1-\bar{Y}}\bar{r}+ T_{\bar{x}} \bar{R})\W)+i\frac{\gamma^2}{2g}\nP(\bar{u}T_{1+\W}T_{1-\bar{Y}}R_\alpha)+G\\
=& i\frac{\gamma^2}{2g}\nP(\bar{r}\W)+i\frac{\gamma^2}{2g}T_{1-\bar{Y}}\nP(\bar{u}T_{1+\W}R_\alpha)+G,\\
i\frac{\gamma^2}{4g}T_{D_t}\nP(\bar{x}W)=& i\frac{\gamma^2}{4g}\nP(D_t \bar{x}W)+i\frac{\gamma^2}{4g}\nP(\bar{x}T_{D_t}W)+G\\
=& -i\frac{\gamma^2}{4g}\nP(T_{1-Y}(T_{1-\bar{Y}}\bar{r}_\alpha + T_{\bar{x}} \bar{R}_\alpha)W)-i\frac{\gamma^2}{4g}\nP(\bar{x}T_{1+\W}T_{1-\bar{Y}}R)+G\\
=& -i\frac{\gamma^2}{4g}\nP(T_{1-\bar{Y}}\bar{r}_\alpha T_{1-Y}W)-i\frac{\gamma^2}{4g}T_{1-\bar{Y}}\nP(\bar{x}T_{1+\W}R)+G,\\
-\frac{\gamma^2}{4g^2}T_{D_t}\nP(\bar{r}T_{1+\W}R) =&-\frac{\gamma^2}{4g^2}\nP(D_t\bar{r}T_{1+\W}R)-\frac{\gamma^2}{4g^2}\nP(\bar{r}T_{1+\W}T_{D_t}R)\\
&-\frac{\gamma^2}{4g^2}\nP(\bar{r}T_{D_t\W}R) +G\\
=&-\frac{\gamma^2}{4g^2}\nP((i\gamma \bar{r}-iT_{g+\ua}(\bar{x}-T_{\bar{w}}\bar{Y}))T_{1+\W}R)\\
&-\frac{\gamma^2}{4g^2}\nP(\bar{r}T_{1+\W}(iT_{g+\ua}Y -i\gamma R)) +\frac{\gamma^2}{4g^2}\nP(\bar{r}T_{T_{1+\W}T_{1-\bar{Y}}R_\alpha}R) +G\\
=& i\frac{\gamma^2}{4g}\nP(\bar{x}T_{1+\W}R)-\frac{\gamma^2}{4g}i\nP(\bar{r}T_{1+\W}Y)+G,\\
-i\frac{\gamma^2}{2g}T_{1-\bar{Y}}\partial_\alpha\nP(\bar{u}T_{1+\W}R)=& -i\frac{\gamma^2}{2g}T_{1-\bar{Y}}\nP(\bar{x}T_{1+\W}R)-i\frac{\gamma^2}{2g}T_{1-\bar{Y}}\nP(\bar{u}T_{1+\W}R_\alpha)+G, \\
i\frac{\gamma^2}{4g}T_{1-\bar{Y}}\partial_\alpha\nP(\bar{r}W)=& i\frac{\gamma^2}{4g}T_{1-\bar{Y}}\nP(\bar{r}_\alpha W) +i\frac{\gamma^2}{4g}T_{1-\bar{Y}}\nP(\bar{r}\W)+G.
\end{align*}
For the contributions of corrections of $\gamma^3$ and $\gamma^4$ terms, we compute
\begin{align*}
i\frac{\gamma^3}{4g^2}T_{D_t}\nP(\bar{r} W) =& i\frac{\gamma^3}{4g^2}\nP(D_t\bar{r} W) +i\frac{\gamma^3}{4g^2}\nP(\bar{r} T_{D_t}W) +G\\
=& i\frac{\gamma^3}{4g^2}\nP((i\gamma \bar{r}-iT_{g+\ua}(\bar{x}-T_{\bar{w}}\bar{Y})) W) -i\frac{\gamma^3}{4g^2}\nP(\bar{r} T_{1+\W}T_{1-\bar{Y}}R) +G\\
=& -\frac{\gamma^4}{4g^2}\nP(\bar{r}W)+\frac{\gamma^3}{4g}\nP(\bar{x}W)-i\frac{\gamma^3}{4g^2}T_{1-\bar{Y}}\nP(\bar{r} T_{1+\W}R) +G,\\
-i\frac{\gamma^3}{4g^2}T_{D_t}\nP(\bar{u}T_{1+\W}R) =& -i\frac{\gamma^3}{4g^2}\nP(T_{D_t}\bar{u}T_{1+\W}R) -i\frac{\gamma^3}{4g^2}\nP(\bar{u}T_{1+\W}T_{D_t}R)\\
&-i\frac{\gamma^3}{4g^2}\nP(\bar{u}T_{D_t\W}R)+G\\
=& i\frac{\gamma^3}{4g^2}\nP(T_{1-Y}(T_{1-\bar{Y}}\bar{r}+ T_{\bar{x}} \bar{R})T_{1+\W}R)\\
&-i\frac{\gamma^3}{4g^2}\nP(\bar{u}T_{1+\W}(iT_{g+\ua}Y-i\gamma R))+i\frac{\gamma^3}{4g^2}\nP(\bar{u}T_{T_{1+\W}T_{1-\bar{Y}}R_\alpha}R)+G\\
=& i\frac{\gamma^3}{4g^2}\nP(T_{1-\bar{Y}}\bar{r}R)+\frac{\gamma^3}{4g}\nP(\bar{u}T_{1+\W}Y)-\frac{\gamma^4}{4g^2}\nP(\bar{u}T_{1+\W}R)+G,\\
-\frac{\gamma^3}{4g}T_{1-\bar{Y}}\partial_\alpha\nP(\bar{u}W) =& -\frac{\gamma^3}{4g}\nP(T_{1-\bar{Y}}\bar{x}W) -\frac{\gamma^3}{4g}\nP(T_{1-\bar{Y}}\bar{u}\W) +G,\\
-\frac{\gamma^4}{4g^2}T_{D_t}\nP(\bar{u}W) =& -\frac{\gamma^4}{4g^2}\nP(T_{D_t}\bar{u}W) -\frac{\gamma^4}{4g^2}\nP(\bar{u}T_{D_t}W)+G  \\
=& \frac{\gamma^4}{4g^2}\nP(T_{1-Y}(T_{1-\bar{Y}}\bar{r}+ T_{\bar{x}} \bar{R})W) +\frac{\gamma^4}{4g^2}\nP(\bar{u}T_{1+\W}T_{1-\bar{Y}}R)+G\\
=& \frac{\gamma^4}{4g^2}\nP(\bar{r}W) +\frac{\gamma^4}{4g^2}T_{1-\bar{Y}}\nP(\bar{u}T_{1+\W}R)+G.
\end{align*}
The contributions of other corrections are perturbative and can be absorbed into $G$.
Collecting all the contributions of the corrections together, we find  that the total contributions for the first equation is exactly $-\mathcal{G}_{0,0} +G$.

$(2)$ Then we consider the second equation of \eqref{NormalFormGKZero}.
We begin by computing the contributions
of corrections of non-vorticity and $\gamma$ terms.
\begin{align*}
-T_{D_t}\nP(\bar{x}T_{1+\W}R) =& -\nP(D_t\bar{x}T_{1+\W}R) -\nP(\bar{x}T_{D_t\W}R) -\nP(\bar{x}T_{1+\W}T_{D_t}R) +K\\
=& \nP(T_{1-Y}(T_{1-\bar{Y}}\bar{r}_\alpha + T_{\bar{x}}\bar{R}_\alpha)T_{1+\W}R) +\nP(\bar{x}T_{T_{1+\W}T_{1-\bar{Y}}R_\alpha}R)\\
&-i\nP(\bar{x}T_{1+\W}(T_{g+\ua}Y -\gamma R)) +K\\
=& \nP(T_{1-\bar{Y}}\bar{r}_\alpha R)-iT_{g+\ua}\nP(\bar{x}T_{1+\W}Y)+i\gamma \nP(\bar{x}T_{1+\W}R)+K,\\
iT_{1-Y}T_{g+\ua}\nP(\bar{x}\W)=& iT_{g+\ua}\nP(\bar{x}T_{1+\W}Y) +K,\\
-\frac{\gamma}{2g}T_{D_t}\nP(\bar{r}T_{1+\W}R)=& -\frac{\gamma}{2g}\nP(D_t\bar{r}T_{1+\W}R) -\frac{\gamma}{2g}\nP(\bar{r}T_{D_t\W}R) -\frac{\gamma}{2g}\nP(\bar{r}T_{1+\W}T_{D_t}R)+K\\
=&-i\frac{\gamma}{2g}\nP((\gamma \bar{r}-T_{g+\ua}(\bar{x}-T_{\bar{w}}\bar{Y}))T_{1+\W}R) +\frac{\gamma}{2g}\nP(\bar{r}T_{T_{1+\W}T_{1-\bar{Y}}R_\alpha}R)\\
&-i\frac{\gamma}{2g}\nP(\bar{r}T_{1+\W}(T_{g+\ua}Y-\gamma R))+K\\
=& i\frac{\gamma}{2}\nP(\bar{x}T_{1+\W}R) -i\frac{\gamma}{2g}T_{g+\ua}\nP(\bar{r}T_{1+\W}Y)+K,\\
i\frac{\gamma}{2}T_{D_t}\nP(\bar{x}W)=& i\frac{\gamma}{2}\nP(D_t\bar{x}W)+i\frac{\gamma}{2}\nP(\bar{x}T_{D_t}W) +K\\
=& -i\frac{\gamma}{2}\nP(T_{1-Y}(T_{1-\bar{Y}}\bar{r}_\alpha + T_{\bar{x}}\bar{R}_\alpha)W)-i\frac{\gamma}{2}\nP(\bar{x}T_{1+\W}T_{1-\bar{Y}}R) +K\\
=& -i\frac{\gamma}{2}\nP(T_{1-\bar{Y}}\bar{r}_\alpha W)-i\frac{\gamma}{2}\nP(\bar{x}T_{1+\W}R) +K,\\
i\frac{\gamma}{2g}T_{1-Y}T_{g+\ua}\nP(\bar{r}\W)=& i\frac{\gamma}{2g}T_{g+\ua}\nP(\bar{r}T_{1-Y}\W)+K = i\frac{\gamma}{2g}T_{g+\ua}\nP(\bar{r}T_{1+\W}Y)+K,\\
i\frac{\gamma}{2g}T_{1-Y}T_{g+\ua}\nP(\bar{x}T_{1+\W}R)=& i\frac{\gamma}{2}\nP(\bar{x}T_{1+\W}R)+K, \quad \text{ and }
-i\gamma\nP(\bar{x}T_{1+\W}R).
\end{align*}
Continuing with the contributions of corrections of $\gamma^2$  terms, we compute
\begin{align*}
-i\frac{\gamma^2}{2g}T_{D_t}\nP(\bar{u}T_{1+\W}R) =& -i\frac{\gamma^2}{2g}\nP(T_{D_t}\bar{u}T_{1+\W}R) -i\frac{\gamma^2}{2g}\nP(\bar{u}T_{D_t\W}R) \\
&-i\frac{\gamma^2}{2g}\nP(\bar{u}T_{1+\W}T_{D_t}R)+K\\
=& i\frac{\gamma^2}{2g}\nP(T_{1-Y}(T_{1-\bar{Y}}\bar{r}+T_{\bar{x}}\bar{R})T_{1+\W}R) +i\frac{\gamma^2}{2g}\nP(\bar{u}T_{T_{1+\W}T_{1-\bar{Y}}R_\alpha}R) \\
&+\frac{\gamma^2}{2g}\nP(\bar{u}T_{1+\W}(T_{g+\ua}Y-\gamma R))+K\\
=& i\frac{\gamma^2}{2g}\nP(T_{1-\bar{Y}}\bar{r}R)+\frac{\gamma^2}{2g}T_{g+\ua}\nP(\bar{u}T_{1+\W}Y)-\frac{\gamma^3}{2g}\nP(\bar{u}T_{1+\W}R)+K,\\
i\frac{\gamma^2}{4g}T_{D_t}\nP(\bar{r}W) =& i\frac{\gamma^2}{4g}\nP(D_t\bar{r}W) + i\frac{\gamma^2}{4g}\nP(\bar{r}T_{D_t}W) +K\\
=& -\frac{\gamma^2}{4g}\nP((\gamma \bar{r}-T_{g+\ua}(\bar{x}-T_{\bar{w}}\bar{Y}))W) - i\frac{\gamma^2}{4g}\nP(\bar{r}T_{1+\W}T_{1-\bar{Y}}R) +K\\
=& -\frac{\gamma^3}{4g}\nP(\bar{r}W)+\frac{\gamma^2}{4}\nP(\bar{x}W)- i\frac{\gamma^2}{4g}\nP(\bar{r}T_{1+\W}R) +K,\\
-\frac{\gamma^2}{2g}T_{1-Y}T_{g+\ua}\nP(\bar{u}\W)=& -\frac{\gamma^2}{2g}T_{g+\ua}\nP(\bar{u}T_{1+\W}Y)+K,\\
\frac{\gamma^2}{4g}T_{1-Y}T_{g+\ua}\nP(\bar{x}W)=&\frac{\gamma^2}{4}\nP(\bar{x}W)+K, \\
i\frac{\gamma^2}{4g^2}T_{1-Y}T_{g+\ua}\nP(\bar{r}T_{1+\W}R) =& i\frac{\gamma^2}{4g}\nP(\bar{r}T_{1+\W}R)+K,\\
-i\frac{\gamma^2}{2g}\nP(\bar{r}T_{1+\W}R)-&\frac{\gamma^2}{2}\nP(\bar{x}W).
\end{align*}
Finally, for the contributions of corrections of $\gamma^3$ and $\gamma^4$ terms, we get
\begin{align*}
-\frac{\gamma^3}{4g}T_{D_t}\nP(\bar{u}W)=& -\frac{\gamma^3}{4g}\nP(T_{D_t}\bar{u}W) -\frac{\gamma^3}{4g}\nP(\bar{u}T_{D_t}W) +K \\
=& \frac{\gamma^3}{4g}\nP(T_{1-Y}(T_{1-\bar{Y}}\bar{r}+T_{\bar{x}}\bar{R})W) + \frac{\gamma^3}{4g}\nP(\bar{u}T_{1+\W}T_{1-\bar{Y}}R) +K \\
=& \frac{\gamma^3}{4g}\nP(\bar{r}W) + \frac{\gamma^3}{4g}\nP(\bar{u}T_{1+\W}R) +K, \\
\frac{\gamma^3}{4g^2}T_{1-Y}T_{g+\ua}\nP(\bar{r} W) =& \frac{\gamma^3}{4g}\nP(\bar{r} W)+K,\\
-\frac{\gamma^3}{4g^2}T_{1-Y}T_{g+\ua}\nP(\bar{u}T_{1+\W}R)=& -\frac{\gamma^3}{4g}\nP(\bar{u}T_{1+\W}R)+K,\\
i\frac{\gamma^4}{4g^2}T_{1-Y}T_{g+\ua}\nP(\bar{u}W)=& i\frac{\gamma^4}{4g}\nP(\bar{u}W)+K,\\
\frac{\gamma^3}{2g}\nP(\bar{u}T_{1+\W}R)-&\frac{\gamma^3}{4g}\nP(\bar{r}W)-i\frac{\gamma^4}{4g}\nP(\bar{u}W).
\end{align*}
Putting together all the contributions of the corrections, we get that the total contributions for the second equation is exactly $-\mathcal{K}_{0,0} +K$.
\end{proof}

\section{Energy estimates for the full system} \label{s:FullEqn}
In this section, we prove energy estimates for the water wave system \eqref{differentiatedEqn}, as in the statement of Theorem \ref{t:MainEnergyEstimate}.

As $(\W, R)$ solve the linearized system \eqref{linearizedeqn}, the results in the previous section show that the full system is well-posed for initial data in $\doth{\frac{1}{4}}$.
However, the estimates for the linearized system  no longer hold for other Sobolev index $s\neq \frac{1}{4}$.
We consider in this section for general $s\geq 0$.

\subsection{Reduction to paradifferential linearized $(\W,R)$ system}
The system \eqref{linearizedeqn} follows from linearizing \eqref{e:CVWW1} and diagonalizing to switch to good variables.
In this section, we instead consider direct linearization from \eqref{differentiatedEqn}.
Let  the linearized variables  around a solution $(\W,R)$ of the water waves system \eqref{differentiatedEqn} be $(\hat{w},\hat{r})$.
We begin by computing the linearizations of the auxiliary functions:
\begin{align*}
&\delta a = i[\bar{\nP}(\bar{\hr}R_\alpha)+\bar{\nP}(\bar{R}\hr_\alpha)-\nP(\hr \bar{R}_\alpha)-\nP(R\bar{\hr}_\alpha)], \\
&\delta \ub = \delta b -i\frac{\gamma}{2}\delta b_1, \quad \delta b = 2\Re \nP[(1-\bar{Y})\hr - (1-\bar{Y})^2 R \bar{\hw}], \\
&\delta b_1 = 2i\Im \nP[(1-\bar{Y})\partial^{-1}_\alpha\hw-(1-\bar{Y})^2W\bar{\hw}], \\
&\delta N = 2\Re \nP[\partial^{-1}_\alpha\hw \bar{R}_\alpha + W\bar{\hr}_\alpha - \bar{\hw}R-\bar{W}_\alpha \hr], \\
&\delta M = 2\Re \nP[\hr \bar{Y}_\alpha -2R(1-\bar{Y})\bar{\hw}\bar{Y}_\alpha + R(1-\bar{Y})^2\bar{\hw}_\alpha -\bar{\hr}_\alpha Y - \bar{R}_\alpha(1-Y)^2\hw], \\
&\delta M_1 = 2i\Im\nP\partial_\alpha[\partial^{-1}_\alpha\hw \bar{Y}+W(1-\bar{Y})^2\bar{\hw}], \quad \delta \underline{M} = \delta M - i\frac{\gamma}{2}\delta M_1.
\end{align*}
Then the linearized system can be written as
\begin{equation*}
\left\{
 \begin{aligned}
    D_t \hw &+ (1-\bar{Y})(1+W_\alpha)\hr_\alpha + (1-\bar{Y})R_\alpha\hw = (1+W_\alpha)\delta\underline{M} + \underline{M}\hw - W_{\alpha\alpha}\delta \ub  \\
    &+(1-\bar{Y})^2(1+W_\alpha)R_\alpha \bar{\hw}+i\gamma W_\alpha \hw - i\frac{\gamma}{2}\bar{W}_\alpha \hw -i\frac{\gamma}{2}W_\alpha \bar{\hw},\\
    D_t \hr &+i\gamma \hr - i(g+a)(1-Y)^2\hw = -R_\alpha\delta \ub-i(1-Y)\delta a + i\frac{\gamma}{2}(1-Y)^2\hw(R+\bar{R}) \\
    &+i\frac{\gamma}{2}W_\alpha (1-Y)(\hr + \bar{\hr})+i\frac{\gamma}{2}(1-Y)\delta N- i\frac{\gamma}{2}N(1-Y)^2 \hw.
\end{aligned}
\right.
\end{equation*}
Next, we rewrite the linearized system in the paradifferential framework.
We consider the cases where $\hw$ and $\hr$ terms are the highest frequencies for the paraproduct types contributions.
After applying the Littlewood-Paley projection $\nP$, we can eliminate all $(\bar{\hw}, \bar{\hr})$ terms, as well as all $(\hw, \hr)$ terms inside the anti-holomorphic projection $\bar{\nP}$.
Using \eqref{ubalpha} to simplify, we obtain the following linear paradifferential equations:
 \begin{equation*}
\left\{
    \begin{array}{lr}
    T_{D_t} \hw + T_{\ub_\alpha} \hw +\partial_\alpha T_{1-\bar{Y}}T_{1+W_\alpha}\hr -i\frac{\gamma}{2}\partial_\alpha T_{1+W_\alpha}T_{1-\bar{Y}}\partial^{-1}_\alpha\hw + i\frac{\gamma}{2}\hw = 0  &\\
    T_{D_t} \hr + T_{\ub_\alpha} \hr + i\gamma \hr -iT_{(1-Y)^2}T_{g+\ua}\hw +T_{\underline{M}}\hr +i\frac{\gamma}{2}T_{b_\alpha + M} \partial^{-1}_\alpha\hw + \gamma T_{\Im W_\alpha Y}\hr =0 .&
             \end{array}
\right.
\end{equation*}
We claim that  the differentiated system \eqref{differentiatedEqn} can be recast in the above paradifferential type equations with source terms.
We let $(G, K)$  be the favorable balanced cubic or higher source terms that satisfy
\begin{equation}
    \|(G,K)\|_{\doth{s}}\lesssim_\uA \underline{A}^2_{\frac{1}{4}}\left(\|(\W, R)\|_{\doth{s}}+\gamma \|(\W, R)\|_{\doth{s-\frac{1}{2}}}\right). \label{GoodBalanceCubic}
\end{equation}
The rest of the source terms are either quadratic in $(\W,R)$ or unbalanced cubic terms.
These unfavorable source terms will be later removed by paradifferential normal form corrections.
This reduction result is stated as following:
\begin{lemma}
 The differentiated system \eqref{differentiatedEqn} can be rewritten as  paradifferential equations for the variable $(\hw, \hr) = (\W, R)$ of the form
 \begin{equation}
 \left\{
    \begin{array}{lr}
    T_{D_t} \hw + T_{\ub_\alpha} \hw +\partial_\alpha T_{1-\bar{Y}}T_{1+W_\alpha}\hr-i\frac{\gamma}{2}\partial_\alpha T_{1+W_\alpha}T_{1-\bar{Y}}\partial^{-1}_\alpha\hw + i\frac{\gamma}{2}\hw = \mathcal{G}(\W,R)+G  &\\
    T_{D_t} \hr +i\gamma\hr+ T_{\ub_\alpha} \hr -iT_{(1-Y)^2}T_{g+\ua}\hw -i\frac{\gamma}{2}T_{b_\alpha} \partial^{-1}_\alpha\hw = \mathcal{K}(\W, R)+K ,&
             \end{array}
\right.  \label{WRSystemParaLin}
\end{equation}
where $(G,K)$ are balanced cubic source terms that satisfy \eqref{GoodBalanceCubic}, and non-perturbative source terms are given by
\begin{equation*}
\left\{
\begin{aligned}
 \mathcal{G}(\W, R) =& -\partial_\alpha[\Pi(\W,  T_{1-\bar{Y}}R)+ \Pi(\W,  T_{1-Y}\bar{R})-\Pi(\bar{Y},T_{1+\W}R)] \\
 &+ i\frac{\gamma}{2}\partial_\alpha[\Pi(\W, T_{1-\bar{Y}}W)-\Pi(\W, T_{1-Y}\bar{W})-\Pi(X,\bar{Y})], \\
 \mathcal{K}(\W, R) =& -T_{1-\bar{Y}}\Pi(R_\alpha, R) -T_{1-Y}\Pi(R_\alpha, \bar{R})-T_{1-Y}\Pi(\bar{R}_\alpha, R)\\
 &+i\frac{\gamma}{2}\Pi(R_\alpha, T_{1-\bar{Y}}W-T_{1-Y}\bar{W})-iT_{1-Y}T_{g+\ua}\Pi(Y,\W)\\
 &+i\frac{\gamma}{2}T_{1-Y}\Pi(W,\bar{R}_\alpha)- i\frac{\gamma}{2}T_{1-Y}\Pi(\bar{\W},R)+i\frac{\gamma}{2}\Pi(R+\bar{R},Y).
\end{aligned}
\right.
\end{equation*}
\end{lemma}
Again we include  the Littlewood-Paley projection $\nP$ in $\Pi$ implicitly as in the previous section.

\begin{proof}
Recall that the differentiated system \eqref{differentiatedEqn} can be rewritten as
\begin{equation*}
\left\{
             \begin{array}{lr}
                D_t \W + (1+\mathbf{W})(1-\bar{Y})R_\alpha = (1+\mathbf{W})\underline{M}+i\dfrac{\gamma}{2}\mathbf{W}(\mathbf{W}-\Bar{\mathbf{W}})  &\\
             D_t R +i\gamma R  = i[g-(g+\ua)(1-Y)] + i\frac{\gamma}{2}(R+\bar{R}).&
             \end{array}
\right.
\end{equation*}
Here the system is written in an algebraic way for convenience.
We  apply the Littlewood-Paley projection $\nP$ to above system to eliminate the anti-holomorphic parts.
For our computation below, the balanced cubic terms are put in $(G,K)$ without further specification.
\begin{enumerate}
\item For the first term in the first equation, we expand $\ub$ to write
\begin{align*}
    \nP D_t\W &= T_{D_t}\W + T_{\W_\alpha} \nP \ub + \Pi(\W_\alpha, \ub) \\
    & = T_{D_t}\W + T_{\W_\alpha} T_{1-\bar{Y}}R-i\frac{\gamma}{2}T_{\W_\alpha} T_{1-\bar{Y}}W \\
    &+ \Pi(\W_\alpha, T_{1-\bar{Y}}R + T_{1-Y}\bar{R}) - i\frac{\gamma}{2}\Pi(\W_\alpha, T_{1-\bar{Y}}W - T_{1-Y}\bar{W}) +G.
\end{align*}
For the second term in the first equation, we expand using paraproducts,
\begin{align*}
    \nP[(1+\W)(1-\bar{Y})R_\alpha] &= T_{1+\W}T_{1-\bar{Y}}R_\alpha + T_{1-\bar{Y}}T_{R_\alpha}\W \\
    &- T_{1+\W}\Pi(\bar{Y}, R_\alpha)+ T_{1-\bar{Y}}\Pi(\W, R_\alpha)+G.
\end{align*}
For the first source term in the first equation, we use \eqref{ubalpha} and \eqref{uMBound} to write
\begin{align*}
    (1+\W)\underline{M} &= T_{1+\W}\underline{M} + G = T_{1+\W}\nP\left[R\bar{Y}_\alpha - \bar{R}_\alpha Y - i\frac{\gamma}{2}(W\bar{Y})_\alpha\right] + G \\
    &=T_{1+\W}\left(T_{\bar{Y}_\alpha}R - T_{\bar{R}_\alpha}Y - i\frac{\gamma}{2}T_{\bar{Y}_\alpha}W -i\frac{\gamma}{2}T_{\bar{Y}}\W\right)\\
    &+T_{1+\W}\left[\Pi(\bar{Y}_\alpha, R)-\Pi(\bar{R}_\alpha,Y)-i\frac{\gamma}{2}\partial_\alpha\Pi(W, \bar{Y})\right]+ G.
\end{align*}
For the last term in the first equation,
\begin{equation*}
i\dfrac{\gamma}{2}\mathbf{W}(\mathbf{W}-\Bar{\mathbf{W}}) = iT_{\frac{\gamma}{2}(\W -\bar{\W})}\W + i\frac{\gamma}{2}T_\W \W + i\frac{\gamma}{2}\Pi(\W, \W-\bar{\W}).
\end{equation*}
Combining all these terms, we use \eqref{XAlphaY} to get
\begin{align*}
 &T_{1-\bar{Y}}T_{R_\alpha}\W + T_{1+\W}T_{\bar{R}_\alpha}Y-iT_{\frac{\gamma}{2}(\W -\bar{\W})}\W  = T_{\ub_\alpha} +G, \\
 &T_{1+\W}T_{1-\bar{Y}}R_\alpha -T_{1+\W}T_{\bar{Y}_\alpha}R+T_{\W_\alpha}T_{1-\bar{Y}}R = \partial_\alpha T_{1+\W}T_{1-\bar{Y}}R, \\
 &i\frac{\gamma}{2}(-T_{\W_\alpha}T_{1-\bar{Y}}W +T_{1+\W_\alpha}T_{\bar{Y}_\alpha}W + T_{\bar{Y}}\W -T_{\W}\W) = -i\frac{\gamma}{2}\partial_\alpha T_{1+\W}T_{1-\bar{Y}}W + i\frac{\gamma}{2}W.
\end{align*}
Hence, we obtain the paradifferential $\W$ equation.
\item For the second $R$ equation, we expand in the similar way to write
\begin{align*}
    \nP D_tR &= T_{D_t}R + T_{R_\alpha} \nP \ub + \Pi(R_\alpha, \ub) \\
    & = T_{D_t}R + T_{R_\alpha} T_{1-\bar{Y}}R-i\frac{\gamma}{2}T_{R_\alpha} T_{1-\bar{Y}}W \\
    &+ \Pi(R_\alpha, T_{1-\bar{Y}}R + T_{1-Y}\bar{R}) - i\frac{\gamma}{2}\Pi(R_\alpha, T_{1-\bar{Y}}W - T_{1-Y}\bar{W}) +K.
\end{align*}
For the first term on the right-hand side of the second equation,
\begin{align*}
     i\nP[g-(g+\ua)(1-Y)] &= iT_{g+\ua}Y-iT_{1-Y}\nP \ua+ i\Pi(\ua, Y)\\
     &= iT_{g+\ua}Y -T_{1-Y}T_{\bar{R}_\alpha}R -T_{1-Y}\Pi(\bar{R}_\alpha, R)\\
      &-i\frac{\gamma}{2}T_{1-Y}R+ i\frac{\gamma}{2}T_{1-Y}T_{\bar{R}_\alpha}W -i\frac{\gamma}{2}T_{1-Y}T_{\bar{\W}}R \\
      &+i\frac{\gamma}{2}T_{1-Y}\Pi(W,\bar{R}_\alpha)-i\frac{\gamma}{2}T_{1-Y}\Pi(\bar{\W},R)+ i\frac{\gamma}{2}\Pi(R+\bar{R},Y)+K,
\end{align*}
where again using  $(4)$ of Lemma \ref{t:XParaMaterial}, we get
\begin{equation*}
    iT_{g+\ua}Y = iT_{(1-Y)^2}T_{g+\ua}\W -iT_{1-Y}T_{g+\ua}\Pi(Y,\W)+K.
\end{equation*}
Similar to the case of the first equation, we combine using \eqref{ubalpha} and \eqref{uMBound},
\begin{align*}
&-T_{R_\alpha} T_{1-\bar{Y}}R - T_{1-Y}T_{\bar{R}_\alpha}R -i\frac{\gamma}{2}T_{1-Y}R -i\frac{\gamma}{2}T_{1-Y}T_{\bar{\W}}R + i\frac{\gamma}{2}R \\
=& -T_{R_\alpha(1-\bar{Y})+\bar{R}_\alpha(1-Y)}R -i\frac{\gamma}{2}T_{1-Y}R + i\frac{\gamma}{2}T_{1-Y}T_{1+\W}R-i\frac{\gamma}{2}T_{1-Y}T_{\bar{\W}}R + K\\
=& -T_{R_\alpha(1-\bar{Y})+\bar{R}_\alpha(1-Y)}R +i\frac{\gamma}{2}T_{\W-\bar{\W}}R +K = -T_{\ub_\alpha}R +K,\\
&-i\frac{\gamma}{2}T_{R_\alpha} T_{1-\bar{Y}}W -i\frac{\gamma}{2}T_{1-Y}T_{\bar{R}_\alpha}W =-i\frac{\gamma}{2}T_{b_\alpha}W +K.
\end{align*}
Hence, we obtain the paradifferential $R$ equation.
\end{enumerate}
\end{proof}
The rest of Section \ref{s:FullEqn} is devoted into the study of paradifferential system \eqref{WRSystemParaLin}.

\subsection{Well-posedness for the paradifferential flow}
In this subsection, we first consider the linear part of the paradifferential equations \eqref{WRSystemParaLin}, namely
\begin{equation}
 \left\{
    \begin{array}{lr}
    T_{D_t} \hw + T_{\ub_\alpha} \hw +\partial_\alpha T_{1-\bar{Y}}T_{1+W_\alpha}\hr-i\frac{\gamma}{2}\partial_\alpha T_{1+W_\alpha}T_{1-\bar{Y}}\partial^{-1}_\alpha\hw + i\frac{\gamma}{2}\hw = 0  &\\
    T_{D_t} \hr +i\gamma\hr+ T_{\ub_\alpha} \hr -iT_{(1-Y)^2}T_{g+\ua}\hw -i\frac{\gamma}{2}T_{b_\alpha} \partial^{-1}_\alpha\hw = 0. &
             \end{array}
\right.  \label{WRSystemParaLinLeft}
\end{equation}
Although it is possible to perform the analysis similar as previouly done in Section \ref{s:LinearEstimate}, a more convenient way is to reduce \eqref{WRSystemParaLinLeft} to the paradifferential equations \eqref{SourceParadifferential}.
The result is as follows:
\begin{proposition}
 Suppose that $(\hw,\hr)$ solve \eqref{WRSystemParaLinLeft}, then there exists a bounded linear transformation which is independent of the Sobolev index $s$, such that it turns $(\hw,\hr)$ into $(w,r)$ that solve \eqref{SourceParadifferential}.
 In addition, for any $s\in \mathbb{R}$, we have
 \begin{enumerate}
\item Invertibility:
\begin{equation*}
    \| (w_\alpha, r_\alpha)-(\hw, \hr)\|_{\doth{s}}\lesssim_\uA \uA\|(w,r)\|_{\doth{s}}.
\end{equation*}
\item Perturbative source term:
\begin{equation*}
    \|(G,K)\|_{\doth{s}}\lesssim_A \underline{A}^2_{\frac{1}{4}}\left(\|(w, r)\|_{\doth{s}}+\gamma \|(w, r)\|_{\doth{s-\frac{1}{2}}}\right).
\end{equation*}
 \end{enumerate}
As a consequence, Proposition \ref{t:wellposedflow} holds with \eqref{WRSystemParaLinLeft} in place of  \eqref{ParadifferentialFlow}. \label{t:hwhrWellPosed}
\end{proposition}
\begin{proof}
When  computing the linearized system, the linearization of $R$ is given by
\begin{equation*}
    \delta R = \frac{r_\alpha +R_\alpha w}{1+\W}.
\end{equation*}
This suggests that the connection between $(w,r)$ and $(\hw, \hr)$ is given by the relation
\begin{equation*}
    (\hw, \hr) = \left(w_\alpha, \frac{r_\alpha +R_\alpha w}{1+\W}\right).
\end{equation*}
Given a solution $(\hw, \hr)$ to \eqref{WRSystemParaLinLeft}, we therefore define $(w,r)$
at the paradifferential level:
\begin{equation}
    (w,r) := (\partial^{-1}_\alpha \hw, \partial^{-1}_\alpha T_{1+\W}\hr - \partial^{-1}_\alpha T_{R_\alpha}\partial^{-1}_\alpha \hw). \label{whatwhatr}
\end{equation}
Clearly, this definition of $(w,r)$ satisfies the invertibility property.
It remains to show that $(w,r)$ defined in \eqref{whatwhatr} satisfy the paradifferential system \eqref{SourceParadifferential} with perturbative source terms.

We plug in the relation \eqref{whatwhatr} into \eqref{ParadifferentialFlow} and compute the corresponding source terms.
Here, the source terms are  acceptable if they satisfy
\begin{equation*}
    \|G\|_{\dot{H}^s}+\|K\|_{\dot{H}^{s+\frac{1}{2}}} \lesssim_\uA \uA^2_{\frac{1}{4}}\|(w,r)\|_{\doth{s+1}}.
\end{equation*}
For the first equation,
\begin{align*}
    &\partial_\alpha(T_{D_t}w+T_{1-\bar{Y}}\partial_\alpha r + T_{(1-\bar{Y})R_\alpha }w +\gamma T_{\Im \W}w) \\
    = &T_{D_t}\hw + T_{\ub_\alpha}\hw + \partial_\alpha T_{1-\bar{Y}}T_{1+W_\alpha}\hr +\gamma \partial_\alpha T_{\Im W_\alpha}\partial_\alpha^{-1}\hw\\
 =  &T_{D_t} \hw + T_{\ub_\alpha} \hw +\partial_\alpha T_{1-\bar{Y}}T_{1+W_\alpha}\hr-i\frac{\gamma}{2}\partial_\alpha T_{1+W_\alpha}T_{1-\bar{Y}}\partial^{-1}\hw + i\frac{\gamma}{2}\hw \\
 &+ \left(\gamma \partial_\alpha T_{\Im W_\alpha}\partial_\alpha^{-1}\hw +i\frac{\gamma}{2}\partial_\alpha T_{1+W_\alpha}T_{1-\bar{Y}}\partial^{-1}\hw - i\frac{\gamma}{2}\hw\right)\\
 = & \left(T_{D_t} \hw + T_{\ub_\alpha} \hw +\partial_\alpha T_{1-\bar{Y}}T_{1+W_\alpha}\hr-i\frac{\gamma}{2}\partial_\alpha T_{1+W_\alpha}T_{1-\bar{Y}}\partial^{-1}\hw + i\frac{\gamma}{2}\hw \right) +G\\
 = & G.
\end{align*}
For the second equation, we use Lemma \ref{t:Leibniz} to distribute para-material derivatives, and use para-commutators \eqref{ParaCommutator}, para-products \eqref{ParaProducts} to write
\begin{align*}
    &\partial_\alpha(T_{D_t}r +i\gamma r -iT_{1-Y}T_{g+\ua}w)\\
    =& T_{D_t}(T_{1+\W}\hr - T_{R_\alpha}\partial_\alpha^{-1}\hw) + T_{\ub_\alpha }(T_{1+\W}\hr - T_{R_\alpha}\partial_\alpha^{-1}\hw)\\
    &+ i\gamma (T_{1+\W}\hr - T_{R_\alpha}\partial_\alpha^{-1}\hw)- i\partial_\alpha T_{1-Y}T_{g+\ua}w\\
     =& T_{D_t \W}\hr + T_{1+\W}T_{D_t}\hr - T_{R_\alpha}T_{D_t}\partial_\alpha^{-1}\hw -T_{\partial_\alpha D_t R}w + T_{1+\W}T_{\ub_\alpha}\hr\\
     &+  i\gamma T_{1+\W}\hr -i\gamma T_{R_\alpha}\partial_\alpha^{-1}\hw +iT_{Y_\alpha}T_{g+\ua}w -iT_{1-Y}T_{g+\ua}w_\alpha-iT_{1-Y}T_{\ua_\alpha}w +K\\
     =& T_{1+\W}(T_{D_t}+T_{\ub_\alpha})\hr - T_{R_\alpha}\partial^{-1}_\alpha(T_{D_t}+T_{\ub_\alpha})\hw + T_{D_t \W}\hr - T_{\partial_\alpha D_t R}w \\
     &+  i\gamma T_{1+\W}\hr -i\gamma T_{R_\alpha}\partial_\alpha^{-1}\hw +iT_{Y_\alpha}T_{g+\ua}w -iT_{1-Y}T_{g+\ua}w_\alpha-iT_{1-Y}T_{\ua_\alpha}w+K \\
     =& iT_{1+\W}T_{(g+\ua)(1-Y)^2}\hw -i\gamma T_{1+\W}\hr + i\frac{\gamma}{2}T_{1+\W}T_{b_\alpha}\partial_\alpha^{-1}\hw +T_{R_\alpha}T_{1-\bar{Y}}T_{1+\W}\hr\\
     &-i\frac{\gamma}{2}T_{R_\alpha}T_{1-\bar{Y}}T_{1+\W} \partial_\alpha^{-1}\hw + i\frac{\gamma}{2}T_{R_\alpha}\partial_\alpha^{-1}\hw -T_{(1+\W)(1-\bar{Y})R_\alpha}\hr +i\gamma T_{R_\alpha}w +iT_{\partial_\alpha(g+\ua)(1-Y)}w \\
     &-i\frac{\gamma}{2} T_{R_\alpha +\bar{R}_\alpha}w +i\gamma T_{1+\W}\hr -i\gamma T_{R_\alpha}\partial_\alpha^{-1}\hw -iT_{\partial_\alpha(g+\ua)(1-Y)}w -iT_{1-Y}T_{g+\ua}w_\alpha +K\\
     =& i\frac{\gamma}{2}T_{1+\W}T_{b_\alpha}\partial_\alpha^{-1}\hw -i\frac{\gamma}{2}T_{R_\alpha}T_{1-\bar{Y}}T_{1+\W} \partial_\alpha^{-1}\hw-i\frac{\gamma}{2}T_{\bar{R}_\alpha}w + K\\
     = & K.
\end{align*}
Here we have harmlessly replaced $T_{D_t}\W$ and $T_{D_t}R$ by $D_t \W$ and $D_t R$.

Finally, the $\doth{s+1}$ well-posedness of \eqref{ParadifferentialFlow} proved in Proposition \ref{t:wellposedflow} implies the $\doth{s}$ well-posedness of
\eqref{WRSystemParaLinLeft}.
\end{proof}

\subsection{The paradifferential normal form transformation}
Having settled the local well-posedness of the linear system \eqref{WRSystemParaLinLeft}, we now consider the right-hand side source terms of \eqref{WRSystemParaLin}.
Although the source terms $(\mathcal{G}, \mathcal{K})$ are not directly perturbative, we are able to use a paradifferential normal form transformation to eliminate them modulo perturbative terms.
Precisely, we can construct paradifferential normal form variables $(\W_{NF},R_{NF})$ that satisfy the following result:
\begin{proposition}
 Suppose that $(\W,R)$ solve the system \eqref{differentiatedEqn}, then there exist paradifferential normal form variables $(\W_{NF},R_{NF})$ that satisfy the following system
 \begin{equation*}
 \left\{
    \begin{array}{lr}
    T_{D_t} \W_{NF} + T_{\ub_\alpha} \W_{NF} +\partial_\alpha T_{1-\bar{Y}}T_{1+W_\alpha}R_{NF}-i\frac{\gamma}{2}\partial_\alpha T_{1+W_\alpha}T_{1-\bar{Y}}\partial^{-1}_\alpha\W_{NF} + i\frac{\gamma}{2}\W_{NF} = \tilde{\mathcal{G}}  &\\
    T_{D_t} R_{NF} +i\gamma R_{NF}+ T_{\ub_\alpha} R_{NF} -iT_{(1-Y)^2}T_{g+\ua}\W_{NF} -i\frac{\gamma}{2}T_{b_\alpha} \partial^{-1}_\alpha\W_{NF} = \tilde{\mathcal{K}} ,&
             \end{array}
\right.
\end{equation*}
such that for any $s\geq 0$, we have
\begin{enumerate}
\item Invertibility:
\begin{equation*}
    \|(\W_{NF},R_{NF})-(\W,R)\|_{\doth{s}}\lesssim_\uA \uA\left(\|(\W,R)\|_{\doth{s}}+\gamma^2\|(\W,R)\|_{\doth{s-1}}\right).
\end{equation*}
\item Perturbative source terms:
\begin{equation*}
    \|(\tilde{\mathcal{G}},\tilde{\mathcal{K}})\|_{\doth{s}}\lesssim_\uA \uA^2_{\frac{1}{4}} \left(\|(\W,R)\|_{\doth{s}}+\gamma^2\|(\W,R)\|_{\doth{s-1}}\right).
\end{equation*}
\end{enumerate}
\end{proposition}
\begin{proof}
We choose the paradifferential normal form variables $(\W_{NF},R_{NF})$ to be $(\W_{NF},R_{NF})$ $:= (\W +\tilde{\W}, R +\tilde{R})$, where the paradifferential corrections $(\tilde{\W}, \tilde{R})$ are given by
\begin{align*}
  \tilde{\W} =& -\partial_\alpha\Pi(\W, 2\Re X)- \frac{\gamma}{2g}\partial_\alpha\Pi(\W, 2\Re Z) -\frac{\gamma}{2g}\partial_\alpha\Pi(R, 2\Re X) - \frac{\gamma^2}{2g}\partial_\alpha\Pi(\W, 2\Im U)\\
  & +i\frac{\gamma^2}{4g}\partial_\alpha \Pi(X,X) + i\frac{\gamma^2}{4g}\partial_\alpha\Pi(X, 2\Re X) - \frac{\gamma^2}{4g^2}\partial_\alpha \Pi(R, 2\Re Z)+i\frac{\gamma^3}{4g^2}\partial_\alpha\Pi(X, 2\Re Z)\\
  &  -\frac{\gamma^3}{4g^2}\partial_\alpha\Pi(R, 2\Im U)+i\frac{\gamma^4}{2g^2}\partial_\alpha \Pi(X, 2\Im U),\\
  \tilde{R} =& -\Pi(R_\alpha, 2\Re X) -\Pi(T_{1-\bar{Y}}\bar{\W}, R)- \frac{\gamma}{2g}\partial_\alpha\Pi(R, 2\Re Z) + i\frac{\gamma}{2}\partial_\alpha\Pi(X, 2\Re X)\\
  &-i\frac{\gamma}{4}\partial_\alpha \Pi(X,X) - \frac{\gamma^2}{2g}\partial_\alpha \Pi(R, 2\Im U) +i\frac{\gamma^2}{4g}\partial_\alpha \Pi(X, 2\Re Z)+i\frac{\gamma^3}{4g}\partial_\alpha \Pi(X, 2\Im U).
\end{align*}
We remark that the quadratic part of this paradifferential correction  is nothing but roughly the balanced part of the derivative of the normal form transformation \eqref{e:NormalForm}, after switching to the good variables $(\W, R)$.

Clearly, $(\tilde{\W},\tilde{R})$ satisfy the invertibility property using direct computation, it only suffices to check that the source terms $(\tilde{\mathcal{G}},\tilde{\mathcal{K}})$ are perturbative.

We  insert these corrections into the system \eqref{WRSystemParaLinLeft} and compute the corresponding source terms.
For both equations, the terms having $T_{\ub_\alpha}$ are perturbative.
For convenience of the computation below, we recall in Section \ref{s:TDtBound}, the leading terms of para-material derivatives.
\begin{alignat*}{2}
&T_{D_t} \W +T_{1+\W}T_{1-\bar{Y}}R_\alpha = E_1, \quad &&\|E_1\|_{BMO}\lesssim_\uA \uA^2_{\frac{1}{4}}, \\
&T_{D_t} R -iT_{g+\ua}Y +i\gamma R = E_2,  \quad &&\|E_2\|_{BMO^{\frac{1}{2}}}\lesssim_\uA \uA^2_{\frac{1}{4}},  \\
&T_{D_t}X +T_{1-\bar{Y}}R  = E_3,  \quad &&\||D|E_3\|_{BMO}+\gamma^2 \|E_3\|_{BMO}\lesssim_{\uAS} \underline{A}^2_{\frac{1}{4}}, \\
&T_{D_t}Z - igX +i\gamma Z = E_4,  \quad &&\gamma^2\||D|^{\frac{1}{2}}E_4\|_{BMO}+\gamma^3\|E_4\|_{BMO}\lesssim_{\uAS} \underline{A}^2_{\frac{1}{4}}, \\
&  T_{D_t}U +T_{1-\bar{Y}}Z= E_5,  \quad &&\gamma^2\||D|E_5\|_{BMO}\lesssim_{\uAS} \underline{A}^2_{\frac{1}{4}}.
\end{alignat*}
We will use Lemma \ref{t:Leibniz} to distribute para-material derivatives.
Let $(G,K)$ be the good source terms that satisfy the perturbative source terms bound in the proposition.
In the following, we will put the perturbative source terms into $(G,K)$ for simplicity.

We first compute the source term of the first equation of \eqref{WRSystemParaLinLeft}.
For the first term of the first equation,
\begin{align*}
-T_{D_t}\partial_\alpha \Pi(\W, 2\Re X) &= \partial_\alpha\left(\Pi(T_{(1+\W)(1-\bar{Y})}R_\alpha, 2\Re X) + \Pi(\W, T_{1-\bar{Y}}R+T_{1-Y}\bar{R})\right)+G, \\
-\frac{\gamma}{2g}T_{D_t}\partial_\alpha \Pi(\W, 2\Re Z)&= \frac{\gamma}{2g}\partial_\alpha \left(\Pi(T_{(1+\W)(1-\bar{Y})}R_\alpha, 2\Re Z) + \Pi(\W, 2g\Im X -2\gamma \Im Z)\right)+ G, \\
-\frac{\gamma}{2g}T_{D_t}\partial_\alpha \Pi(R, 2\Re X) &= \frac{\gamma}{2g}\partial_\alpha\left(\Pi(-iT_{g+\ua}Y+i\gamma R, 2\Re X)+\Pi(R, T_{1-\bar{Y}}R+T_{1-Y}\bar{R}) \right)+G, \\
-\frac{\gamma^2}{2g}T_{D_t}\partial_\alpha \Pi(\W, 2\Im U) &= \frac{\gamma^2}{2g}\partial_\alpha \left(\Pi(T_{(1+\W)(1-\bar{Y})}R_\alpha, 2\Im U) + \Pi(\W, 2\Im T_{1-\bar{Y}}Z)\right)+G,\\
i\frac{\gamma^2}{4g}T_{D_t}\partial_\alpha\Pi(X,X) &= -i\frac{\gamma^2}{2g}\partial_\alpha\Pi(T_{1-\bar{Y}}R,X)+G, \\
i\frac{\gamma^2}{4g}T_{D_t}\partial_\alpha\Pi(X,2\Re X) &= -i\frac{\gamma^2}{4g}\partial_\alpha \left(\Pi(T_{1-\bar{Y}}R,2\Re X)+\Pi(X, T_{1-\bar{Y}}R + T_{1-Y}\bar{R})\right)+G,\\
-\frac{\gamma^2}{4g^2}T_{D_t}\partial_\alpha \Pi(R, 2\Re Z) &= \frac{\gamma^2}{4g^2}\partial_\alpha \left(\Pi(-iT_{g+\ua}Y +i\gamma R, 2\Re Z)+ \Pi(R, 2g\Im X -2\gamma \Im Z)\right) +G,\\
i\frac{\gamma^3}{4g^2}T_{D_t}\partial_\alpha\Pi(X, 2\Re Z) &= -i\frac{\gamma^3}{4g^2}\partial_\alpha \left(\Pi(T_{1-\bar{Y}}R, 2\Re Z)+\Pi(X, 2g\Im X -2\gamma \Im Z)\right) +G,\\
-\frac{\gamma^3}{4g^2}T_{D_t}\partial_\alpha\Pi(R, 2\Im U) & = \frac{\gamma^3}{4g^2}\partial_\alpha \left(\Pi(-iT_{g+\ua}Y+i\gamma R, 2\Im U)+\Pi(R, 2\Im T_{1-\bar{Y}}Z)\right)+G,\\
i\frac{\gamma^4}{2g^2}T_{D_t}\partial_\alpha \Pi(X, 2\Im U) & = -i\frac{\gamma^4}{2g^2}\partial_\alpha \left(\Pi(T_{1-\bar{Y}}R, 2\Im U)+\Pi(X, 2\Im T_{1-\bar{Y}}Z)\right)+G.
\end{align*}
For the third term of the first equation, we apply the para-associativity \eqref{ParaAssociateOne} and the relations \eqref{XAlphaY}, \eqref{ZAlphaR}, \eqref{UAlphaX},
\begin{align*}
-\partial_\alpha T_{1-\bar{Y}}T_{1+\W}\Pi(R_\alpha, 2\Re X) &=-\partial_\alpha \Pi(T_{(1-\bar{Y})(1+\W)}R_\alpha, 2\Re X) +G,\\
-\partial_\alpha T_{1-\bar{Y}}T_{1+\W}\Pi(T_{1-\bar{Y}}\bar{\W}, R) &= -\partial_\alpha \Pi(\bar{Y}, T_{1+\W}R)+G,\\
-\frac{\gamma}{2g}\partial_\alpha T_{1-\bar{Y}}T_{1+\W}\partial_\alpha\Pi(R, 2\Re Z) &= -\frac{\gamma}{2g}\partial_\alpha \left(\Pi(T_{(1-\bar{Y})(1+\W)}R_\alpha, 2\Re Z)+ \Pi(R, 2\Re T_{1-\bar{Y}}R)\right)+G,\\
i\frac{\gamma}{2}\partial_\alpha T_{1-\bar{Y}}T_{1+\W}\partial_\alpha\Pi(X, 2\Re X) &= i\frac{\gamma}{2}\partial_\alpha \left(\Pi(Y, 2\Re X) + \Pi(X, 2\Re \W)\right)+G,\\
-i\frac{\gamma}{4}\partial_\alpha T_{1-\bar{Y}}T_{1+\W}\partial_\alpha \Pi(X,X) &= -i\frac{\gamma}{2}\partial_\alpha \Pi(\W, T_{1-\bar{Y}}W)+G,\\
-\frac{\gamma^2}{2g}\partial_\alpha T_{1-\bar{Y}}T_{1+\W}\partial_\alpha \Pi(R, 2\Im U) &= - \frac{\gamma^2}{2g}\partial_\alpha \left(\Pi(T_{(1+\W)(1-\bar{Y})}R_\alpha, 2\Im U) + \Pi(T_{1-\bar{Y}}R, 2\Im X)\right)+G,\\
i\frac{\gamma^2}{4g}\partial_\alpha T_{1-\bar{Y}}T_{1+\W}\partial_\alpha \Pi(X, 2\Re Z) &= i\frac{\gamma^2}{4g}\partial_\alpha \left(\Pi(\W, 2\Re Z)+ \Pi(X,  T_{1-\bar{Y}}R + T_{1-Y}\bar{R})\right)+G,\\
i\frac{\gamma^3}{4g}\partial_\alpha T_{1-\bar{Y}}T_{1+\W}\partial_\alpha \Pi(X, 2\Im U) &= i\frac{\gamma^3}{4g}\partial_\alpha\left(\Pi(\W, 2\Im U)+ \Pi(X, 2\Im X)\right)+G.
\end{align*}
Here, when the derivatives  fall on the paradifferential coefficients $T_{1+\W}$ or $T_{1-\bar{Y}}$, these terms are perturbative and may go to $G$.

For the last two terms of the first equation, they are
\begin{equation*}
-i\frac{\gamma}{2}\partial_\alpha T_{1+\W}T_{1-\bar{Y}}\partial^{-1}_\alpha \tilde{\W} + i\frac{\gamma}{2}\tilde{\W} =   -i\frac{\gamma}{2}T_{\W-\bar{Y}-\W\bar{Y}} \tilde{\W} -i\frac{\gamma}{2}T_{\W_\alpha}T_{1-\bar{Y}}\tilde{\W} + i\frac{\gamma}{2}T_{1+\W}T_{\bar{Y}_\alpha}\tilde{\W}.
\end{equation*}
Since the paradifferential correction $\tilde{\W}$ satisfies
\begin{equation*}
    \|\tilde{\W}\|_{\doth{s}}\lesssim_\uA \uA\left(\|(\W,R)\|_{\doth{s}}+\gamma^2\|(\W,R)\|_{\doth{s-1}}\right),
\end{equation*}
the sum of the last two terms can be absorbed into $G$.

Gathering all terms for the first equation, we see that the contribution of paradifferential corrections cancel the source term $\tilde{\mathcal{G}}$ modulo perturbative terms.

Next, we perform the computation for the contribution of the paradifferential corrections in the second equation of \eqref{WRSystemParaLinLeft}.
For the first term of the second equation,
\begin{align*}
-T_{D_t}\Pi(R_\alpha, 2\Re X) &= -i\Pi(T_{g+\ua}Y_\alpha - \gamma R_\alpha, 2\Re X) +\Pi(R_\alpha, T_{1-\bar{Y}}R +T_{1-Y}\bar{R}) +K,\\
-T_{D_t}\Pi(T_{1-\bar{Y}}\bar{\W},R) &= T_{1-Y}\Pi(\bar{R}_\alpha, R)-i \Pi(T_{1-\bar{Y}}\bar{\W},T_{g+\ua}Y-\gamma R) +K,\\
-\frac{\gamma}{2g}T_{D_t}\partial_\alpha\Pi(R,2\Re Z) &= -\frac{\gamma}{2g}\partial_\alpha \left(\Pi(iT_{g+\ua}Y-i\gamma R, 2\Re Z)+\Pi(R, -2g\Im X+2\gamma \Im Z)\right) +K,\\
i\frac{\gamma}{2}T_{D_t}\partial_\alpha\Pi(X,\bar{X}) &= -i\frac{\gamma}{2}\partial_\alpha \left(\Pi(T_{1-\bar{Y}}R, \bar{X})+ \Pi (T_{1-Y}\bar{R},X)\right) +K,\\
-i\frac{\gamma}{4}T_{D_t}\partial_\alpha\Pi(X,X) &= i\frac{\gamma}{2}\partial_\alpha \Pi(T_{1-\bar{Y}}R, X) +K,\\
-\frac{\gamma^2}{2g}T_{D_t}\partial_\alpha\Pi(R,2\Im U) &=
 \frac{\gamma^2}{2g}\partial_\alpha\left(\Pi(i\gamma R-iT_{g+\ua}Y, 2\Im U)+\Pi(R,2\Im T_{1-\bar{Y}}Z)\right)+K,\\
i\frac{\gamma^2}{4g}T_{D_t}\partial_\alpha\Pi(X,2\Re Z) &= -i\frac{\gamma^2}{4g}\partial_\alpha \left(\Pi(T_{1-\bar{Y}}R,2\Re Z) +\Pi(X,2g\Im X -2\gamma \Im Z)\right)+K,\\
i\frac{\gamma^3}{4g}T_{D_t}\partial_\alpha\Pi(X, 2\Im U) &= -i\frac{\gamma^3}{4g}\partial_\alpha\left(\Pi(T_{1-\bar{Y}}R, 2\Im U)+ \Pi(X, 2\Im T_{1-\bar{Y}}Z)\right)+K.
\end{align*}
The second term of the second equation is given by
\begin{align*}
&-i\gamma\Pi(R_\alpha, 2\Re X) -i\gamma\Pi(T_{1-\bar{Y}}\bar{\W}, R)- i\frac{\gamma^2}{2g}\partial_\alpha\Pi(R, 2\Re Z) -\frac{\gamma^2}{2}\partial_\alpha\Pi(X, 2\Re X)\\
  &+\frac{\gamma^2}{4}\partial_\alpha \Pi(X,X) - i\frac{\gamma^3}{2g}\partial_\alpha \Pi(R, 2\Im U) -\frac{\gamma^3}{4g}\partial_\alpha \Pi(X, 2\Re Z)-\frac{\gamma^4}{4g}\partial_\alpha \Pi(X, 2\Im U).
\end{align*}
For the fourth term of the second equation, the contribution of $T_\ua$ is perturbative, and may be absorbed into $K$.
We apply the para-associativity \eqref{ParaAssociateOne} to write,
\begin{align*}
  iT_{(1-Y)^2}T_{g+\ua}\partial_\alpha\Pi(\W, 2\Re X) &= i\Pi(T_{g+\ua}Y_\alpha, 2\Re X)\\
  &+i T_{g+\ua}\left(T_{1-Y}\Pi(\W, Y)+\Pi(Y, \bar{X}_\alpha)\right) +K ,\\
  i\frac{\gamma}{2g}T_{(1-Y)^2}T_{g+\ua}\partial_\alpha\Pi(\W,2\Re Z) &= \frac{\gamma}{2g}\partial_\alpha \Pi(iT_{g+\ua}Y, 2\Re Z)+K ,\\
  i\frac{\gamma}{2g}T_{(1-Y)^2}T_{g+\ua}\partial_\alpha\Pi(R,2\Re X) &= i\frac{\gamma}{2}\partial_\alpha\Pi(T_{1-\bar{Y}}R, 2\Re X)+K ,\\
  i\frac{\gamma^2}{2g}T_{(1-Y)^2}T_{g+\ua}\partial_\alpha\Pi(\W, 2\Im U) &= \frac{\gamma^2}{2g}\partial_\alpha\Pi(iT_{g+\ua}Y, 2\Im U)+K ,\\
  \frac{\gamma^2}{4g}T_{(1-Y)^2}T_{g+\ua}\partial_\alpha\Pi(X,X) &= \frac{\gamma^2}{4}\partial_\alpha\Pi(X,X)+K ,\\
  \frac{\gamma^2}{4g}T_{(1-Y)^2}T_{g+\ua}\partial_\alpha\Pi(X,2\Re X) &= \frac{\gamma^2}{4}\partial_\alpha(X,2\Re X)+K ,\\
  i\frac{\gamma^2}{4g^2}T_{(1-Y)^2}T_{g+\ua}\partial_\alpha\Pi(R, 2\Re Z) &= i\frac{\gamma^2}{4g}\partial_\alpha\Pi(R, 2\Re Z)+K ,\\
  \frac{\gamma^3}{4g^2}T_{(1-Y)^2}T_{g+\ua}\partial_\alpha\Pi(X,2\Re Z) &= \frac{\gamma^3}{4g}\partial_\alpha\Pi(X,2\Re Z)+K ,\\
   i\frac{\gamma^3}{4g^2}T_{(1-Y)^2}T_{g+\ua}\partial_\alpha\Pi(R, 2\Im U) &= i \frac{\gamma^3}{4g}\partial_\alpha\Pi(R, 2\Im U) +K ,\\
  \frac{\gamma^4}{2g^2}T_{(1-Y)^2}T_{g+\ua}\partial_\alpha\Pi(X, 2\Im U) &= \frac{\gamma^4}{2g}\partial_\alpha \Pi(X, 2\Im U) +K.
\end{align*}
The last term of the left-hand side of the second equation $-i\frac{\gamma}{2}T_{b_\alpha} \partial^{-1}_\alpha \tilde{\W}$ is perturbative and may absorbed into $K$.
After cancellations, the source term of the paradifferential corrections of the second equation equals $-\tilde{\mathcal{K}}$ modulo perturbative $K$ terms.
Therefore, we construct the paradifferential normal form variables $(\W_{NF},R_{NF})$ that satisfy the desired properties.
\end{proof}
Having switched into the paradifferential normal form variables $(\W_{NF},R_{NF})$, the differentiated system \eqref{differentiatedEqn} is reduced to the paradifferential equations \eqref{WRSystemParaLinLeft} with perturbative source terms.
Theorem \ref{t:MainEnergyEstimate} then follows directly from the last part of Proposition \ref{t:hwhrWellPosed}.

\section{The proof of local well-posedness} \label{s:Proof}
In this section, we prove the main result of this article, namely the low regularity well-posedness Theorem \ref{t:MainWellPosedness}.
The result from \cite{MR3869381} asserts that the local well-posedness holds for more regular data.
We will first establish the $\mathcal{H}^n$ bounds for  regular solutions.
Then, we use those  regular solutions as  starting points to construct rough solutions.
Precisely, we obtain regular solutions by truncating the rough initial data in frequency,  so that we get a continuous family of solutions, thereby estimating only a solution for  linearized equations at each step.
Finally, we prove the continuous dependence on the initial data in $\mathcal{H}^{\frac{3}{4}}$.

Note that we can make  use of the space-time scaling \eqref{SpacetimeScaling}.
By choosing the scaling parameter $\lambda$ small enough, we can make $\uAS \ll 1$, at the price to turn the vorticity $\gamma$ into $\lambda\gamma$.
The argument for the proof below is similar to the proof in \cite{ai2023dimensional}.
For simplicity, we only give an outline of the approach here.
Please check Section 7 of \cite{ai2023dimensional} for detailed proof such as the use of \textit{frequency envelopes}.

\begin{proof}[Outline of proof of Theorem \ref{t:MainWellPosedness}]
 The  proof of the theorem is divided into the following four steps:

\textbf{$(1)$ $\mathcal{H}^s$ bounds for regular solutions.}
For simplicity, here we only outline the case for integer $n\geq 1$.
Suppose we have an $\mathcal{H}^n$ solution $(\W, R)$ that satisfies  the initial condition
\begin{equation*}
\|(\W_0, R_0)\|_{\doth{\frac{3}{4}}}+ \gamma^2 \|(\W_0, R_0)\|_{\doth{-\frac{1}{4}}}\leq \mathcal{M}_0 \ll 1.
\end{equation*}
In order to show that there exists a time $T = T(\mathcal{M}_0)$ such that the solution exists in $C([0,T]; \mathcal{H}^n)$ and satisfies the bounds
\begin{align}
&\|(\W, R)\|_{L^\infty(0,T;\doth{\frac{3}{4}})} + \gamma^2 \|(\W, R)\|_{L^\infty(0,T;\doth{-\frac{1}{4}})}<  \mathcal{M}(\mathcal{M}_0), \label{ThreeQuarter} \\
&\|(\W, R)\|_{L^\infty(0,T;\mathcal{H}^n)}\leq  C(\mathcal{M}_0) \|(\W_0,R_0
)\|_{\mathcal{H}^n}, \label{nEstimate}
\end{align}
we make the bootstrap assumption
\begin{equation*}
    \|(\W, R)\|_{L^\infty(0,T;\doth{\frac{3}{4}})} + \gamma^2 \|(\W, R)\|_{L^\infty(0,T;\doth{-\frac{1}{4}})}<  2\mathcal{M}.
\end{equation*}
Using the bootstrap assumption and Sobolev embedding, we are able to bound the control parameters
\begin{equation*}
    \Astar\lesssim \mathcal{M}, \quad  \uAStar\lesssim \mathcal{M}.
\end{equation*}
Applying the energy estimate with $s= \frac{3}{4}$  in Theorem \ref{t:MainEnergyEstimate}, we have
\begin{equation*}
 \|(\W, R)(t)\|_{\doth{\frac{3}{4}}} + \gamma^2 \|(\W, R)(t)\|_{\doth{-\frac{1}{4}}} \lesssim e^{Ct} \left( \|(\W_0, R_0)\|_{\doth{\frac{3}{4}}} + \gamma^2 \|(\W_0, R_0)\|_{\doth{-\frac{1}{4}}}\right) \lesssim e^{Ct}\mathcal{M}_0.
\end{equation*}
By choosing large $\mathcal{M}$ and suitable $T$, the bound \eqref{ThreeQuarter} holds for $t\in [0,T]$.

Applying Theorem \ref{t:MainEnergyEstimate} and Gronwall’s inequality for each integer $k$ between $1$ and $n$, we have
\begin{equation*}
\| (\W ,R)(t)\|_{\doth{k}} + \gamma^2 \| (\W ,R)(t)\|_{\doth{k-1}} \lesssim e^{Ct} \left(\| (\W_0 ,R_0)\|_{\doth{k}} + \gamma^2 \| (\W_0 ,R_0)\|_{\doth{k-1}}\right).
\end{equation*}
Summing above inequalities for $k$ from $1$ to $n$, we get \eqref{nEstimate}.
Therefore we have the $\mathcal{H}^n$ bounds for regular solutions up to time $T$.

\textbf{$(2)$  Construction of rough solutions, $(\W, R)\in \mathcal{H}^{\frac{3}{4}}$.}
We regularize the initial data, $(W_{<k}(0), Q_{<k}(0))$ and $(\W_{<k}(0), R_{<k}(0))$ by truncating at frequency $2^{k}$.
We are able to establish the bound of regularized initial data using frequency envelopes.
The corresponding solutions will be regular, with a uniform lifespan bound.
Here $k$ can be viewed as a continuous parameter rather than a
discrete parameter.
Then
\begin{equation*}
    (w^k, r^k) = (\partial_k W_{<k}, \partial_k Q_{<k} - R_{<k} \partial_k W_{<k} )
\end{equation*}
solve the corresponding linearized equations around $(\W_{<k}, R_{<k})$.
For the high-frequency part of the regularized solutions, we apply the energy estimates for the full equation Theorem \ref{t:MainEnergyEstimate}.
Next, using Theorem \ref{t:LinearizedWellposed} for the linearized variables $(w^k, r^k)$, one can establish the difference bound $(\W_{<k+1}-\W_{<k}, R_{<k+1}-R_{<k})$ in $\mathcal{H}^{\frac{3}{4}}$.
Summing up with respect to $k$, it follows that the sequence $(\W_{<k}, R_{<k})$ converges to a solution $(\W, R)$ with uniform $\mathcal{H}^{\frac{3}{4}}$ bound in time interval $[0,T]$.

\textbf{$(3)$  Continuous dependence on the data for rough solutions.}
Consider an arbitrary sequence $(\W_j, R_j)(0)$ that converges to $(\W_0, R_0)$ in $\mathcal{H}^{\frac{3}{4}}$ topology.
Using again the frequency truncation, we get the approximate solutions $(\W_j^k, R_j^k)$, respectively $(\W^k, R^k)$.
Due to the continuous dependence for the regular solutions which is proved in Theorem \ref{t:AiResult} of \cite{MR4462478}, we have for each $k$
\begin{equation*}
(\W^k_j, R^k_j)- (\W^k, R^k) \rightarrow 0 \quad \text{in } \mathcal{H}^n,\quad  n\geq 1.
\end{equation*}
On the other hand, letting $k$ go to infinity, we have for the initial data
\begin{equation*}
(\W^k_j, R^k_j)(0)- (\W_j, R_j)(0) \rightarrow 0 \quad \text{in } \mathcal{H}^\frac{3}{4}, \quad \text{uniformly in } j.
\end{equation*}
Using the frequency envelope analysis, we further get the uniform convergence for the solution:
\begin{equation*}
(\W^k_j, R^k_j)- (\W_j, R_j) \rightarrow 0 \quad \text{in } \mathcal{H}^\frac{3}{4}, \quad \text{uniformly in } j.
\end{equation*}
 We can again let $k$ go to infinity to conclude that
 \begin{equation*}
  (\W_j, R_j)-(\W, R)\rightarrow 0 \quad \text{in } \mathcal{H}^{\frac{3}{4}},
 \end{equation*}
 which shows the continuous dependence on the data in $\mathcal{H}^{\frac{3}{4}}$.

\textbf{$(4)$   Continuation of solutions.}
Here we show that the solution can be continued for as long as $\max \{\uAS, \uA\}$ remains small and $\Astar(\gamma^{\frac{1}{4}}+\Astar) \in L^1_t$.
This is a direct consequence of the energy estimates Theorem \ref{t:MainEnergyEstimate}.
\end{proof}

\bibliographystyle{plain}
\bibliography{refs}
\end{document}